\documentclass[12pt,reqno]{amsart}
\usepackage {amssymb}
\usepackage {amsmath}
\usepackage{amsthm}
\usepackage{mathtools}
\usepackage{graphicx}
\usepackage[alphabetic]{amsrefs}
\usepackage[colorlinks, citecolor=blue]{hyperref}
\usepackage{verbatim,color,geometry}
\usepackage[all]{xy}

\newtheorem{prop}{Proposition}[section]
\newtheorem{thm}[prop]{Theorem}
\newtheorem{defn-thm}[prop]{Theorem-Definition}
\newtheorem{defn-lem}[prop]{Lemma-Definition}
\newtheorem{cor}[prop]{Corollary}
\newtheorem{conj}[prop]{Conjecture}
\newtheorem{lem}[prop]{Lemma}

\theoremstyle{definition}

\newtheorem{defn}[prop]{Definition}
\newtheorem{expl}[prop]{Example}
\newtheorem{rem}[prop]{\it Remark}

\newtheorem{notation}[prop]{Notation}

\newtheorem*{claim*}{Claim}

\newcommand{\bP}{\mathbb{P}}
\newcommand{\PP}{\mathbb{P}}
\newcommand{\bC}{\mathbb{C}}
\newcommand{\bR}{\mathbb{R}}
\newcommand{\bA}{\mathbb{A}}
\newcommand{\bQ}{\mathbb{Q}}

\newcommand{\bZ}{\mathbb{Z}}
\newcommand{\bN}{\mathbb{N}}

\newcommand{\bG}{\mathbb{G}}

\newcommand{\bfJ}{\mathbf{J}}

\newcommand{\oY}{\overline{Y}}
\newcommand{\oT}{\overline{T}}

\newcommand{\oS}{\overline{S}}

\newcommand{\tcI}{\widetilde{\mathcal{I}}}

\newcommand{\tI}{\widetilde{I}}
\newcommand{\tdelta}{{\delta_T}}
\newcommand{\tv}{\tilde{v}}
\newcommand{\la}{\langle}
\newcommand{\ra}{\rangle}

\newcommand{\tbP}{\widetilde{\mathbb{P}}}

\newcommand{\cX}{\mathcal{X}}
\newcommand{\cY}{\mathcal{Y}}
\newcommand{\cD}{\mathcal{D}}
\newcommand{\cO}{\mathcal{O}}
\newcommand{\cL}{\mathcal{L}}
\newcommand{\cI}{\mathcal{I}}
\newcommand{\cR}{\mathcal{R}}
\newcommand{\cM}{\mathcal{M}}
\newcommand{\cF}{\mathcal{F}}
\newcommand{\cG}{\mathcal{G}}

\newcommand{\cE}{\mathcal{E}}

\newcommand{\fa}{\mathfrak{a}}

\newcommand{\fX}{\mathfrak{X}}
\newcommand{\fD}{\mathfrak{D}}

\newcommand{\Supp}{\mathrm{Supp}}
\newcommand{\Proj}{\mathrm{Proj}}
\newcommand{\Hom}{\mathrm{Hom}}
\newcommand{\mult}{\mathrm{mult}}
\newcommand{\lct}{\mathrm{lct}}

\newcommand{\vol}{\mathrm{vol}}
\newcommand{\ord}{\mathrm{ord}}

\newcommand{\Gr}{\mathrm{Gr}}

\newcommand{\wt}{\mathrm{wt}}
\newcommand{\Val}{\mathrm{Val}}
\newcommand{\Aut}{\mathrm{Aut}}
\newcommand{\Fut}{\mathrm{Fut}}
\newcommand{\HN}{\mathrm{HN}}
\newcommand{\rd}{\mathrm{d}}
\newcommand{\DNA}{\mathbf{D}^{\mathrm{NA}}}
\newcommand{\JNA}{\mathbf{J}^{\mathrm{NA}}}
\newcommand{\LNA}{\mathbf{L}^{\mathrm{NA}}}
\newcommand{\Sym}{\mathrm{Sym}}
\newcommand{\rk}{\mathrm{rank}}
\newcommand{\CM}{\mathrm{CM}}

\numberwithin{equation}{section}

\begin{document}

\title{On positivity of the CM line bundle on K-moduli spaces}

\author{Chenyang Xu}
\address{Current Address: Department of Mathematics, Princeton University, Princeton, NJ, 08544, USA.}
\email{chenyang@princeton.edu}

\address{Department of Mathematics, MIT, Cambridge, MA, 02139, USA.}
\email{cyxu@math.mit.edu}

\address{Beijing International Center for Mathematical Research, Beijing, 100871, China}
\email{cyxu@math.pku.edu.cn}

\author{Ziquan Zhuang}
\address{Department of Mathematics, MIT, Cambridge, MA, 02139, USA.}
\email{ziquan@mit.edu}

\date{}

\maketitle

{\let\thefootnote\relax\footnotetext{CX is partially supported by the NSF (No. 1901849 and No. 1952531). Both authors also received support
from the grant DMS-1440140 and CX also received support from Chern Professorship while in residence at MSRI during the Spring 2019 semester.  


}
\marginpar{}
}

\begin{abstract}{
In this paper, we consider the CM line bundle on the K-moduli space, i.e., the moduli space parametrizing K-polystable Fano varieties. We prove it is ample on any proper subspace parametrizing reduced uniformly K-stable Fano varieties which conjecturally should be the entire moduli space. As a corollary, we prove that the moduli space parametrizing smoothable K-polystable Fano varieties is projective.  

During the course of proof, we develop a new invariant for filtrations which can be used to test various K-stability notions of Fano varieties. }
\end{abstract}

\setcounter{tocdepth}{1}
\tableofcontents

\section{Introduction}
Throughout the paper, we work over an algebraically closed field $k$ with ${\rm char}(k)=0$. 
\subsection{Positivity of CM line bundle}

The aim of the current paper is to study the positivity of CM line bundle on the moduli space parametrizing K-polystable Fano varieties. 
The CM line bundle was introduced algebraically in \cite{Tia97} on the base of a family of (possibly singular) projective polarized varieties, as a generalization of a well studied Hermitian line bundle on the moduli space parametrizing polarized manifolds with constant scalar curvature (see \cites{Koi83, Schu83, Tian87, FS90} etc.). Its current formulation using the Knudsen-Mumford expansion was introduced in \cite{PT09} (see also \cites{FR06, PRS08} etc.). In many moduli problems, the CM line bundle is expected to give a natural polarization of the corresponding moduli spaces. In particular, \cite{PX17} shows that this holds on the KSB moduli spaces, i.e., on the moduli space parametrizing canonically polarized varieties with semi-log-canonical singularities, the CM line bundle is ample. 

In the Fano case, however, positivity of the CM line bundle has been a long standing question. By recent works \cites{Jia17, BX19, ABHX19, BLX19, Xu19}, we know that the good moduli space $M^{\rm Kps}_{n,v}$ parametrizing $n$-dimensional K-polystable Fano varieties with a given volume $v$ exists as a separated algebraic space. Conjecturally it is also proper, and the ultimate goal is to show that the CM line bundle is ample on the (conjecturally proper) space $M^{\rm Kps}_{n,v}$, which endows it with a projective scheme structure. 

The K-moduli space $M^{\rm Kps}_{n,v}$ contains an interesting `main component' $\overline{M}^{\rm sm, Kps}_{n,v}$ (see \cites{LWX19, Oda15}) parametrizing smoothable K-polystable Fano varieties, which is known to be proper (see e.g. \cite{DS14}). In \cite{LWX18}, the CM line bundle $\Lambda_{\CM}$ was proved to be big and nef on $\overline{M}^{\rm sm, Kps}_{n,v}$ and ample on the open locus $M^{\rm sm, Kps}_{n,v}$ parametrizing smooth K-polystable Fano varieties. However, the argument is of an analytic nature, as it uses the positivity of Weil-Petersson metric and heavily depends on the fact that the generic point of the moduli space parametrizes a Fano manifold. Also see its log version extension in \cite{ADL19}.

An important progress towards the positivity of CM line bundle, using only algebro-geometric arguments, is achieved in \cite{CP18}. There they show that for a general family $f\colon X\to B$ of (possibly singular) Fano varieties over a proper base, the CM line bundle $\lambda_{f}$ on $B$ is nef if the fibers are K-semistable, and big if the family has maximal variation (Definition \ref{defn:max var}) and the fibers are uniformly K-stable. In their proof, the characterization of K-stability using the log canonical thresholds with respect to basis type divisors, as developed in \cites{FO18, BJ17}, plays an essential role.  

\medskip

In this paper, we aim to prove positivity results of the CM line bundle which conjecturally gives the full projectivity of $M^{\rm Kps}_{n,v}$. Our first main statement goes as follows.

\begin{thm}\label{t-main2}
Let $M\subset M^{\rm Kps}_{n,v}$ be a proper algebraic subspace such that every geometric point $s\in M$ parametrizes a reduced uniformly K-stable Fano variety $X_{s}$.
Then the CM $(\bQ$-$)$line bundle $\Lambda_{\CM}|_{M}$ is ample. \end{thm}

Our proof of Theorem \ref{t-main2} is purely algebraic. Here the algebro-geometric concept of  {\it reduced uniform K-stability} was introduced in \cite{His16} and developed systematically in \cite{Li19}. It serves as an analogue of uniform K-stability when ${\rm Aut}(X)$ is of positive dimension. Recall that a Fano variety $X$ is said to be uniformly K-stable if there exists some $\eta>0$ such that $\DNA(\cX,\cL)\ge \eta\cdot \JNA(\cX,\cL)$ for any test configuration $(\cX,\cL)$ of $X$, where $\DNA(\cX,\cL)$ is the Ding-invariant and $\JNA(\cX,\cL)$ is a norm on the space of test configurations that vanishes  exactly on (almost) trivial test configurations (see \cites{Fuj19, Li17} or Definition \ref{defn:D^NA}). To define reduced uniform K-stability, we simply replace the $\JNA$-functional by a reduced version  $\JNA_T$ which is the infimum of $\JNA$ among all test configurations that can be  obtained by `twisting' the given one (this way we get a norm functional on the space of test configurations that vanishes exactly on product test configurations). We refer to Section \ref{ss-ruk} for the precise definition. 

By the recent work \cite{Li19} (see also \cites{BBJ15, LTW19}), reduced uniform K-stability is equivalent to the existence of (singular) K\"ahler-Einstein metric. As all smoothable K-polystable Fano varieties admit K\"ahler-Einstein metrics by \cite{CDS, Tia15} (see also \cites{LWX19, SSY16}), we get the following immediate consequence, which affirmatively answers a question asked by many people (see e.g.  \cite{Don18a},  \cite[\S4.2 the last paragraph]{Don18} or  \cite[\S 4]{Sun18}).

\begin{cor} \label{c-smkps}
Let $k=\bC$. The CM $\bQ$-line bundle $\Lambda_{\rm CM}$ is ample on the proper moduli space $\overline{M}^{\rm sm, Kps}_{n,v}$. 
\end{cor}

In general, all K-polystable Fano varieties are expected to be reduced uniformly K-stable (see Conjecture \ref{c-kpoly}), just as K-stability is conjectured to be equivalent to uniform K-stability (see e.g. \cite[Conjecture 1.5]{BX19}). 
Therefore, at least conjecturally, the assumption of Theorem \ref{t-main2} should be satisfied for all proper algebraic subspaces of $M^{\rm Kps}_{n,v}$. 
Theorem \ref{t-main2} and Corollary \ref{c-smkps} can also be extended to the corresponding log version (see Theorem \ref{t-main2log} and \ref{t-logample}), where the log counterpart of  $\overline{M}^{\rm sm, Kps}_{n,v}$ is constructed in \cite{ADL19}*{Theorem 1.1} (see Theorem \ref{t-logsmoothable}). 

\subsection{Non-negativity of $\beta_{X,\Delta,\delta}(\cF)$}

During the course of the proof, we develop some new invariants that characterize various K-stability notions (including reduced uniform K-stability in particular), which we believe merit independent interests.

To explain the results, we fix some notation. Let $(X,\Delta)$ be a log Fano pair and $r$ a positive integer such that $-r(K_X+\Delta)$ is Cartier. Let $\cF$ be a (possibly non-finitely generated) linearly bounded multiplicative filtration of 
$$R:=\bigoplus_{m\in \bN} H^0(X,-rm(K_X+\Delta)).$$
We define a family of invariants $\beta_{X,\Delta,\delta}(\cF)$ parametrized by  ${\delta}>0$.

\begin{defn}[= Definition \ref{d-beta}]\label{d-beta2}
Given a filtration $\cF$ of $R$ and some $\delta\in\bR_+$, we define the \emph{$\delta$-log canonical slope} (or simply \emph{log canonical slope} when $\delta=1$) $\mu_{X,\Delta,\delta}(\cF)$ as
\begin{equation}
\mu_{X,\Delta,\delta}(\cF) = \sup \left\{t\in\bR\,|\,\lct(X,\Delta;I^{(t)}_\bullet)\ge \frac{\delta}{r}\right\}
\end{equation}
where $I_{m,\lambda}(\cF):={\rm Im} \left(  \cF^\lambda R_m \otimes \cO_{X}(rm(K_X+\Delta)) \to \cO_X\right)$ are the base ideals of $\cF$
and $I^{(t)}_\bullet$ is the graded sequence of ideals given by $I^{(t)}_m := I_{m, tm}(\cF)$. Then we define $$\beta_{X,\Delta,\delta}(\cF):= \frac{\mu_{X,\Delta,\delta}(\cF)-S(\cF)}{r}$$
where $S(\cF)$ is the $S$-invariant of the filtration $\cF$ (see \S\ref{ss-filt}).
\end{defn}

In \cites{Fuj19,Li17}, valuative criteria of K-stability and related notions have been developed, which are further extended in \cites{BJ17, BX19} etc. The analogous criterion for reduced uniform K-stability  is systematically studied in \cite{Li19}. Our second main theorem extends these previous works and provides a new criterion.

\begin{thm}\label{t-main1}
Let $(X,\Delta)$ be a log Fano pair. Then
\begin{enumerate}
\item $(X,\Delta)$ is K-semistable if and only if 
$$\beta_{X,\Delta}(\cF)(:=\beta_{X,\Delta,1}(\cF))\ge 0$$ for any linearly bounded multiplicative filtration $\cF$. 
\item $(X,\Delta)$ is uniformly K-stable if and only if there exists a constant $\delta>1$, such that
$\beta_{X,\Delta,\delta}(\cF)\ge 0$ for any linearly bounded multiplicative filtration  $\cF$. 
\item Fix a maximal torus $T\subset \Aut(X,\Delta)$. Then $(X,\Delta)$ is reduced uniformly K-stable if and only if there exists a constant $\delta>1$ such that  for any linearly bounded multiplicative filtration $\cF$,  
$$\beta_{X,\Delta,\delta}(\cF_{\xi})\ge 0\qquad \mbox{for some vector field }\xi \in \Hom(\bG_m,T)\otimes_\bZ {\bR}.$$ 
\end{enumerate}
\end{thm}

Here $\cF_\xi$ is the `twist' of the filtration $\cF$ and we refer to Section \ref{ss-ruk} for its construction.

\begin{rem}
Our invariant $\beta_{X,\Delta,\delta}$ is a natural extension of the $\beta$-invariant, defined in \cites{Fuj19, Li17} for valuations, to more general filtrations. Indeed, it is not hard to check (see Proposition \ref{prop:tbeta property}) that $\beta_{X,\Delta}(E)\ge\beta_{X,\Delta}(\cF_{\ord_E})$ where $\cF_{\ord_E}$ is the filtration induced by $E$ (Example \ref{ex: filtration from val}). On the other hand, we will show that $\beta_{X,\Delta}(\cF)$ is bounded from below by the Ding-invariant $\DNA(\cF)$ (Theorem \ref{thm:tbeta>=D^NA}). In particular, Theorem \ref{t-main1}(1) follows immediately from the corresponding non-negativity statement for $\DNA(\cF)$ as proved in \cite{Fuj18b, Li17}.  

While Theorem \ref{t-main1}(2) can be established in a similar way, the proof of the last statement Theorem \ref{t-main1}(3) is harder and requires a more technical analysis of the properties of Duistermaat–Heckman measure.
\end{rem}

The connection between Theorem \ref{t-main1} and the positivity of CM line bundle is provided by a special filtration: the {\it Harder-Narashimhan filtration $\cF_{\HN}$} (see $\S$\ref{sec:HN}). Using  the Harder-Narashimhan filtration to study positivity of the CM line bundle is a novel idea initiated in \cite{CP18}. In fact, our definition of $\beta_{X,\Delta,\delta}(\cF)$ in some sense is tailor-made to investigate $\cF_{\HN}$.

\subsection{Overview of proof}

In what follows, we sketch our strategy for the proof of Theorem \ref{t-main2}.

Part of our proof is inspired by the work in \cite{CP18}. Recall that the arguments in \cite{CP18} for the positivity of CM line bundle $\lambda_{f,\Delta}$ for a projective family of log Fano pairs $f\colon (X,\Delta)\to B$ can be divided into several steps:

\begin{enumerate}
\item[Step 1$_{\rm u}$:] one gets the semipositivity of the CM line bundle when the fibers $(X_t,\Delta_t)$ are K-semistable; 
\item[Step 2$_{\rm u}$:] one concentrates on the case that $B$ is a smooth curve and obtains a uniform nef threshold $c$ (i.e. some $c\in\bQ$ such that $-(K_{X/B}+\Delta)+c\cdot f^*\lambda_{f,\Delta}$ is nef) when a general fiber $(X_t,\Delta_t)$ is uniformly K-stable, where $c$ only depends on $\delta(X_t,\Delta_t)$ and $(-(K_{X_t}+\Delta_t))^n$; 
\item[Step 3$_{\rm u}$:] an ampleness lemma is used to get the strict positivity for higher dimensional base $B$, when the fibers are assumed to be uniformly K-stable and $\Delta=0$.
\end{enumerate}

In \cite{Pos19}, complementary arguments are also given to show that

\begin{enumerate}
\item[Step 4$_{\rm u}$:]  Step 3$_{\rm u}$ can be extended to the log case (i.e. when $\Delta\neq 0$). 
\end{enumerate}

\medskip

In our argument, we follow these steps. However,  since we want to treat the more general case when the fibers are only reduced uniformly K-stable, we have to develop a number of new tools to substantially enhance the arguments in \cite{CP18}.

\medskip

\noindent {Step $1$$_{\rm r}$:} Our first new input is going from basis type divisors to filtrations. This has been known as a natural step to extend results from uniformly K-stable Fano varieties to K-polystable ones. See \cite{BX19} for another example. By doing so, we also find a more conceptual proof (in our opinion) of Step $1_{\rm u}$ in \cite{CP18}.

More precisely, a crucial observation in \cite{CP18} is that for a $\mathbb{Q}$-Gorenstein family of log Fano pairs $f\colon (X,\Delta)\to C$ over a smooth projective curve $C$, the subbundles of positive slopes in the Harder-Narasimhan filtration of
$$\mathcal{R}:=\bigoplus_{m\in \bN} f_*\cO_X(-mr(K_{X/C}+\Delta))$$
only give boundary divisors $D\sim_\bQ -(K_{X/C}+\Delta)$ that fail to be log canonical along the general fiber of $f$. Using our $\beta$-invariants, this translates to the following inequality (Proposition \ref{prop:cm>=beta}):
\[
\deg \lambda_{f,\Delta} \ge (n+1)  (-K_{X_t}-\Delta_t)^n  \cdot \beta_{X_t,\Delta_t}(\cF_{\HN})
\]
where $t\in C$ and $\cF_\HN$ is the restriction of the Harder-Narasimhan filtration on $\cR$ to $R_t:=\bigoplus_m H^0(X_t,-mr(K_{X_t}+\Delta_t))$. Combining with Theorem \ref{t-main1}(1), we immediately recover  the nefness of $\lambda_{\rm CM}$ as proved in \cite{CP18}*{Theorem 1.7}. This completes our treatment of the first step. See  \S\ref{ss-tbetaredK} for details.


\medskip

\noindent {Step 2$_{\rm r}$:} Nevertheless, the full power of considering filtrations $\cF_{\HN}$ (as opposed to basis type divisors) and the corresponding invariants $\beta_{X,\Delta,\delta}(\cF_{\HN})$ will only be seen when one attacks the ampleness. The main issue is that nef thresholds as in Step $2_{\rm u}$ may not exist when the general fibers $(X_t,\Delta_t)$ are only reduced uniformly K-stable, e.g. there are families $f\colon (X,\Delta)\to C$ with $\deg \lambda_{f,\Delta}=0$ but $-(K_{X/C}+\Delta)$ is not nef. To overcome this, we need a stronger statement Theorem \ref{t-main1}(3) which guarantees that if a general fiber $(X_t,\Delta_t)$ is reduced uniformly K-stable, then we can find a \emph{rational} vector $\xi \in \Hom(\bG_m,T)\otimes_\bZ \bQ$ (where $T$ is a maximal torus of $\Aut(X_t,\Delta_t)$) such that after twisting $\cF_{\HN}$ by $\xi$, the resulting filtration $(\cF_{\HN})_{\xi}$ satisfies
$\beta_{X_t,\Delta_t,\delta}((\cF_{\HN})_{\xi})\ge 0$ for some uniform $\delta>1$ which only depends on $(X_t, \Delta_t)$ (but not $\cF_{\HN}$ and $\xi$). These results are proved in Sections \ref{ss-tbetaredK} and \ref{sec:twisted family}.

Geometrically, after possibly passing to a finite base change of $C$, we can realize the twisted filtration $(\cF_{\HN})_{\xi}$ as the Harder-Narasimhan filtration of a twisted family $f_{\xi}\colon (X_{\xi},\Delta_{\xi})\to C$ (roughly speaking, it is obtained from the original family $f\colon (X,\Delta)\to C$ via a birational modification that is analogous to elementary transformations on ruled surfaces; see Section \ref{sec:twisted family} for the actual construction). But since $\beta_{X_t,\Delta_t,\delta}((\cF_{\HN})_{\xi})\ge 0$, a similar argument as in Step $1_{\rm r}$ allows us to conclude that a uniform nef threshold (depending only on a general fiber $(X_t,\Delta_t)$) exists on the twisted family. For families with K-semistable fibers, the CM line bundle remains the same after the twist; thus for our ampleness question, we may replace the original family by a twisted one that achieves the nef threshold.



\medskip

\noindent {Step 3$_{\rm r}$:}
Over a higher dimensional base, the Harder-Narasimhan filtration depends on the choice of a covering family of curves and so does the corresponding twisted family. Therefore to proceed, we need to strengthen the key ingredient of Step 3$_{\rm u}$, namely the ampleness lemma (which is originally based on works in \cite{Kol90, KP17}), to a version that incorporates all twists. For this purpose, we carefully track the original argument of the ampleness lemma, and show that it works birationally in a suitably technical sense. This is done in Section \ref{s-ampleness}.

\medskip

\noindent {Step 4$_{\rm r}$:}
Finally to prove the log version of our theorem, we combine the ampleness lemma obtained in the previous step together with an argument  using perturbation of the coefficients of the boundary divisors, similar to the one used in \cite{KP17,Pos19}. This is done in Section \ref{s-CMpositive}. The main observation is that although $(X_t,\Delta_t)$ may no longer be reduced uniformly K-stable after perturbing the boundary coefficients, the divisor $-(K_{X/C}+\Delta)+c\cdot f^*\lambda_{f,\Delta}$ from Step $2_{\rm r}$ remains pseudo-effective and this is enough for our purpose. 

\medskip

Finally, in the Appendix \ref{a-reduced}, we define {\it the $T$-reduced $\delta$-invariant} $\tdelta(X,\Delta)$ for a log Fano pair $(X,\Delta)$ with a torus $T$-action. Then in Theorem \ref{t-redqm}, we show that if $(X,\Delta)$ is K-semistable and $\tdelta(X,\Delta)=1$, it can always be computed by a quasi-monomial valuation which is not of the form $\wt_{\xi}$ $(\xi\in \Hom(\bG_m, T)_{\mathbb R})$. This establishes in the reduced version the analogous result proved in \cite[Theorem 1.5]{BLX19}. 


\medskip

\noindent {\bf Acknowledgement:}  We thank Dan Abramovich, Jarod Alper, Chi Li, Yan Li and Zsolt Patakfalvi for helpful discussions. We are also grateful to the anonymous referees for useful suggestions. The work on this paper was started while the authors enjoyed the hospitality of the MSRI, which is gratefully acknowledged.

\section{Preliminaries}\label{s-preliminary}

In this section, we recall some basic preliminaries related to the study of K-stability questions on log Fano pairs.

\subsection{Notation and conventions}

We work over an algebraically closed field $k$ of characteristic zero. 
We follow the terminologies in \cite{KM98}. A \emph{pair} $(X,\Delta)$ consists of a normal variety $X$ and an effective $\bQ$-divisor $\Delta\subseteq X$ such that $K_X+\Delta$ is $\bQ$-Cartier. See \cite{KM98}*{Definition 2.34} for the definition of \emph{klt} and \emph{log canonical} (\emph{lc}) singularities. 
A projective variety $X$ is $\bQ$-\emph{Fano} if $X$ has klt singularities and $-K_X$ is ample. A pair $(X,\Delta)$ is \emph{log Fano} if $X$ is projective, $-K_X-\Delta$ is $\bQ$-Cartier ample and $(X,\Delta)$ is klt. A big open set $U$ of a variety $X$ is an open set whose complement has codimension at least two.

\subsection{Families of pairs}

\begin{defn}\label{d-genericlog}
A \emph{generic log Fano locally stable family $f\colon (X,\Delta)\to S$ of normal pairs over a normal base} (or abbreviated as \emph{a generic log Fano family $f\colon (X,\Delta)\to S$}) consists of a pair $(X,\Delta)$ and a flat projective morphism $f\colon X\to S$ to a normal variety $S$ such that $f$ has connected and normal fibers, $\Supp(\Delta)$ does not contain any fiber of $f$, $-(K_{X/S}+\Delta)$ is $f$-ample, $(X_s,\Delta_s)$ is log canonical for any $s\in S$ and a general fiber $(X_s,\Delta_s)$ is log Fano. It is called a \emph{$\bQ$-Gorenstein family of log Fano pairs} if every fiber $(X_s,\Delta_s)$ is log Fano.
\end{defn}
The local conditions guarantee that in the above definition $(X,\Delta)\to S$ yields {\it a locally stable family} over the normal base $S$. For the definition of locally stable families over more general bases, see \cite{Kol17, Kol19}. 

\begin{defn}[Base change]
Let $f\colon X\to S$ be a flat projective morphism between normal varieties with normal fibers and let $\Delta$ be an effective Weil $\bQ$-divisor on $X$ whose support does not contain any fiber of $f$.
Let $S'\to S$ be a morphism from another (normal) variety. Let $U\subseteq X$ be the smooth locus of $f$. As in \cite[4.1.5]{Kol17} (see also \cite{CP18}*{\S 2.4.1}), the base change of $f$ to $S'$ is set to be the family $f'\colon (X',\Delta')\to S'$ of normal pairs where $X'=X\times_S S'$ and $\Delta'$ is the unique extension of the pullback of $\Delta|_U$ to $U'=U\times_S S'$ which is a big open subset in $X'$. We call $\Delta'$ {\it the divisorial pull back} of $\Delta$.  If $K_{X/S}+\Delta$ is $\bQ$-Cartier then we have $K_{X'/S'}+\Delta'\sim_\bQ \pi^*(K_{X/S}+\Delta)$ where $\pi\colon X'\to X$ is the induced morphism. In particular, being a generic log Fano family is preserved under base change.
\end{defn}

\begin{defn}[Maximal variation] \label{defn:max var}
Let $f\colon X\to S$ be a flat projective morphism between normal varieties with normal fibers and let $\Delta$ be an effective Weil $\bQ$-divisor on $X$, whose support does not contain any fiber. We say that the family $f\colon (X,\Delta)\to S$ is isotrivial if $(X_s,\Delta_s)\cong (X_t,\Delta_t)$ for any two general points $s,t\in S$; we say that {\it $f$ has maximal variation} if for any curve $C\subseteq S$ containing a general point, the base change of $f$ to $C$ is not isotrivial.
\end{defn}

\begin{lem} \label{lem:max var to reduced bdy}
Let $f\colon X\to S$ be a flat projective morphism between normal varieties with normal fibers, let $\Delta$ be an effective Weil $\bQ$-divisor on $X$ and let $D=\Supp(\Delta)$. Then $f\colon (X,\Delta)\to S$ has maximal variation if and only if $g\colon (X,D)\to S$ has maximal variation.
\end{lem}

\begin{proof}
Let $C\subseteq S$ be a curve passing through a general point. It suffices to show that the base change of $f$ to $C$ is isotrivial if and only if the same holds for $g$. For this we may assume that $S=C$ is a curve. Clearly if $(X_s,\Delta_s)\cong (X_t,\Delta_t)$ for any two general points $s,t\in C$, then we also have $(X_s,D_s=\Supp(\Delta_s))\cong (X_t,D_t=\Supp(\Delta_t))$. 

For the reverse implication, after a finite dominant base change of $C$, we can first assume all components of $D$ have geometrically irreducible fibers over the generic point of $C$. Then we use induction on the number of irreducible components of $D$. The statement is clear when $D$ is irreducible. In general, let $D^i$ ($i=1,\cdots,N$) be the irreducible components of $D$. Let $Z_{ij}\subseteq C\times C$ ($1\le i,j\le N$) be the image of 
$${\rm Isom}_{C\times C}\big((X,D^i)\times C, C\times (X,D^j)\big)\to C\times C$$ which is a countable union of constructible subsets such that $(X_s,D^i_s)\cong (X_t,D^j_t)$ if and only if $(s,t)\in Z_{ij}$. Since $g\colon (X,D)\to C$ is isotrivial, $\cup_{i,j} Z_{ij}$ contains a dense open subset of $C\times C$ and thus the same is true for some $Z_{ij}$. It follows that $(X_s,D^i_s)\cong (X_t,D^j_t)$ for two arbitrary general points $s,t\in C$ and therefore $(X,D^i)$ and $(X,D^j)$ are both isotrivial over $C$ (it can happen that $i=j$). The result now follows by induction hypothesis after removing the component $D^i$ from both $\Delta$ and $D$.
\end{proof}

\subsection{Filtrations} \label{ss-filt}
Applying filtrations to study K-stability questions has been explored in many recent works. We recall some basic definitions here. For more background, see e.g.  \cite{BHJ17}. 

\bigskip

Let $L$ be an ample line bundle on a projective variety $X$ and let 
\[
R = R(X,L) :=\bigoplus_{m\in\bN} R_m := \bigoplus_{m\in \bN} H^0(X,mL)
\]
be its section ring.

\begin{defn} \label{defn:filtration}
By a (\emph{linearly bounded multiplicative}) \emph{filtration} $\cF$ of $R$, we mean the data of a family of vector subspaces 
$\cF^\lambda R_m \subseteq R_m$  
for $m \in \bN$ and $\lambda \in \bR$ satisfying
\begin{itemize}
\item[(1)]
$\cF^\lambda R_m  \subseteq \cF^{\lambda'} R_m$ when $\lambda \geq \lambda'$;
\item[(2)] 
$\cF^\lambda R_m = \cap_{\lambda ' < \lambda} \cF^{\lambda'} R_m$ for  all $\lambda$;
\item[(3)]
There exists $e_-,e_+\in \bR$ such that $\cF^{mx}R_m=0$ for all $x\ge e_+$ and $\cF^{mx}R_m=R_m$ for all $x\le e_-$
\item[(4)]
$\cF^\lambda R_m \cdot \cF^{\lambda'} R_{m'} \subseteq \cF^{\lambda +\lambda'} R_{m+m'}$.
\end{itemize}
A filtration $\cF$ of $R$ is a called a \emph{$\bZ$-filtration} if $\cF^\lambda R_m = \cF^{\lceil \lambda \rceil} R_m$ for all $m\in \bZ$ and $\lambda \in \bR$. It is called an $\bN$-filtration if in addition $\cF^0 R_m = R_m$ for all $m$. To give a $\bZ$-filtration $\cF$, it suffices to give a family of subspaces $\cF^p R_m \subseteq R_m$ for $m\in \bN$ and $p\in \bZ$ satisfying (1), (3), and (4). In particular, every filtration $\cF$ on $R$ induces a $\bZ$-filtration $\cF_\bZ$ which satisfies that  $\cF_\bZ^{\lambda}=\cF^{\lceil \lambda \rceil}$. We say that an $\bN$-filtration $\cF$ is finitely generated if the algebra $\bigoplus_{m,i\in\bN}\cF^i R_m$ is finitely generated.
\end{defn}

\begin{expl} \label{ex: filtration from val}
Let $v$ be a valuation on $X$ (i.e. a valuation $k(X)^\times\to \bR$ that is trivial on $k$). Then it induces a filtration $\cF_v$ of $R$ by setting $\cF_v^\lambda R_m = \{ s\in R_m\,|\,v(s)\ge \lambda\}$. It is linearly bounded if $v$ has linear growth (see \cite{BJ17}*{Section 3.1}).
\end{expl}

\begin{expl} \label{expl:filtration from tc}
Let $(X,\Delta)$ be a log Fano pair. By \cite{BHJ17}*{Proposition 2.15}, there exists a one-to-one correspondence between test configurations of $(X,\Delta)$ (see Definition \ref{def:tc}) and finitely generated $\bZ$-filtrations of $R=R(X,-r(K_X+\Delta))$ for some sufficiently divisible positive integer $r$.
\end{expl}

Let $\cF$ be a linearly bounded multiplicative filtration on $R$. Let 
$$\Gr^\lambda_\cF R_m = \cF^\lambda R_m / \bigcup_{\lambda'>\lambda} \cF^{\lambda'} R_m.$$
We define (c.f. \cite{BJ17}*{\S 2.3-2.6})
\[
S_m(\cF):=\frac{1}{m\dim R_m}\sum_{\lambda\in\bR} \lambda \dim\Gr^\lambda_\cF R_m
\]
and $S(\cF)=\lim_{m\to \infty} S_m(\cF)$. Note that the above expression is a finite sum since there are only finitely many $\lambda$ for which $\Gr^\lambda_\cF R_m\neq 0$ and the limit exists by \cite{BC11}. For $x\in \bR$, we set
\[
\vol(\cF R^{(x)})=\lim_{m\to \infty} \frac{\dim \cF^{mx}R_m}{m^n/n!}
\]
where $n=\dim X$ (the limit exists by \cite{LM09}). Then 
$$\nu:=\frac{1}{(L^n)}\frac{\rd}{\rd x}\vol(\cF R^{(x)})$$ is {\it the Duistermaat-Heckman measure} of the filtration (see \cite{BHJ17}*{\S 5}) and we denote by $[\lambda_{\min}(\cF),$ $\lambda_{\max}(\cF)]$ its support. We also have 
\[
S_m(\cF) = e_- + \frac{1}{\dim R_m}\int_{e_-}^{e_+} \dim \cF^{mx}R_m {\rm d}x
\]
and
\[
S(\cF) = \lambda_{\min}(\cF) + \frac{1}{(L^n)} \int_{\lambda_{\min}(\cF)}^{\lambda_{\max}(\cF)} \vol(\cF R^{(x)}) {\rm d}x = \int_{\bR} x\ \rd \nu.
\]
We will need the following result from \cite{BHJ17}.
\begin{lem}[{\cite[Theorem 5.3]{BHJ17}}]\label{l-concavedistri}
The function $x\to \vol(\cF R^{(x)})^{\frac{1}{n}}$ is concave on $(-\infty, \lambda_{\max})$.
\end{lem}

We also need the notion of translations of filtrations. For $c\in\bR$, we let $\cF_c$ be the filtration on $R$ defined by
\[
\cF_c^\lambda R_m := \cF^{\lambda-cm} R_m
\]
and call it the translation of $\cF$ by $c$.

\begin{lem} \label{lem:S after translation}
We have $S(\cF_\bZ)=S(\cF)$ and $S(\cF_c)=S(\cF)+c$ for any $c\in \bR$. 
\end{lem}

\begin{proof}
By definition we have $S_m(\cF_c)=S_m(\cF)+c$. Letting $m\to \infty$ we get the second equality. Replacing $\cF$ by a translation we may assume that $\cF^0 R_m = R_m$ for all $m\in \bN$. The first equality then follows from \cite{BJ17}*{Corollary 2.12}.
\end{proof}

\begin{defn}[Base ideals] \label{defn:base ideal}
To a filtration $\cF$ of $R$, we associate a family of graded sequence of base ideals.
For $m\geq 0$ and $\lambda\in\bR$, set 
$$I_{m,\lambda} = I_{m,\lambda}(\cF) := {\rm Im} \left(  \cF^\lambda R_m \otimes \cO_{X}(-mL) \to \cO_X\right)$$
where the map is induced by the natural evaluation $H^0(X,mL)\otimes \cO_X\to \cO_X(mL)$. Let $t\in \bR$ and let $I^{(t)}_m = I^{(t)}_m(\cF) := I_{m, tm}(\cF)$. Then $I^{(t)}_\bullet$ is a graded sequence of ideals on $X$.
\end{defn}

\subsection{Invariants associated to log Fano pairs}In this section, we recall some invariants which are introduced in previous works (e.g. \cite{BHJ17, Fuj18b, Li17}).

Let $(X,\Delta)$ be a log Fano pair and let $r>0$ be an integer such that $L:=-r(K_X+\Delta)$ is Cartier. Let $R=R(X,L)$.

\begin{defn}
Let $v$ be a valuation of linear growth on $X$ (by \cite{BJ17}*{Lemma 3.1}, this is the case if $A_{X,\Delta}(v)<\infty$, where $A_{X,\Delta}(v)$ denotes the log discrepancy). Then we define
\[
S(v)=S_{X,\Delta}(v):=r^{-1}S(\cF_v), \quad\mbox{and}\quad T(v)=T_{X,\Delta}(v):=r^{-1}\lambda_{\max}(\cF_v).
\]
If $E$ is a divisor over $X$ (i.e. $E$ is a prime divisor on some proper birational model $\mu\colon Y\to X$), we set $S_{X,\Delta}(E)=S(\ord_E)$ and $T_{X,\Delta}(E)=T(\ord_E)$ where $\ord_E\colon k(X)^\times \to \bZ$ is the discrete valuation given by the order of vanishing along $E$. It is not hard to check that
\begin{equation}\label{e-sinvariant}
    S_{X,\Delta}(E) = \frac{1}{\vol(-K_X-\Delta)}\int_0^\infty \vol(-\mu^*(K_X+\Delta)-xE) \rd x
\end{equation}
and 
\begin{equation}\label{e-tinvariant}
    T_{X,\Delta}(E) = \sup\{x\in\bR\,|\,\vol(-\mu^*(K_X+\Delta)-xE)>0\}.
\end{equation}
\end{defn}

\begin{defn}[$\beta$-invariant, alpha invariant and stability threshold]
Let $v$ be a valuation with $A_{X,\Delta}(v)<\infty$ on $X$ and let $E$ be a divisor over $X$. We define (c.f. \cite{Fuj19, Li17})
\[
 \beta_{X,\Delta}(v)(=\beta(v)) := A_{X,\Delta}(v)-S_{X,\Delta}(v) \quad \mbox{and} \quad
\beta(E) = \beta(\ord_E).
\]
The alpha invariant $\alpha(X,\Delta)$ and the stability threshold $\delta(X,\Delta)$ of the log Fano pair $(X,\Delta)$ is defined as
\[
\alpha(X,\Delta)=\inf_E \frac{A_{X,\Delta}(E)}{T_{X,\Delta}(E)},\quad \delta(X,\Delta) = \inf_E \frac{A_{X,\Delta}(E)}{S_{X,\Delta}(E)}
\]
where both infima run over all divisors $E$ over $X$ (see \cite{BJ17}).
\end{defn}

\begin{defn}[Non-Archimedean invariants] \label{defn:D^NA}
Let $(X_{\bA^1},\Delta_{\bA^1})=(X,\Delta)\times \bA^1$ and $X_0=X\times \{0\}$. Let $\cF$ be a filtration on $R$ and choose $e_-$ and $e_+$ as in Definition \ref{defn:filtration} such that $e_-,e_+\in\bZ$.  Let $e=e_+-e_-$ and for each $m\in\bN$, set
\[
\tI_m := \tI_m(\cF) := I_{m,me_+} + I_{m,me_+-1}\cdot t + \cdots + I_{m,me_-+1}\cdot t^{me-1} + (t^{me}) \subseteq \cO_{X\times \bA^1}.
\]
It is not hard to verify that $\tI_\bullet$ is a graded sequence of ideals. Let
\begin{align*}
    c_m & = \lct(X_{\bA^1}, \Delta_{\bA^1}\cdot (\tI_m)^\frac{1}{mr};X_0)\\
     & = \sup\{ c\in\bR \,|\, (X_{\bA^1}, (\Delta_{\bA^1}+cX_0)\cdot (\tI_m)^\frac{1}{mr}) \text{ is sub log canonical} \} 
\end{align*}
and $c_\infty = \lim_{m\to \infty} c_m$. We then define (c.f. \cite{BHJ17})
\begin{align*}
    \LNA(\cF) & = c_\infty + \frac{e_+}{r} -1, \\ 
    \DNA(\cF) & = \LNA(\cF) - \frac{S(\cF)}{r},\\
    \JNA(\cF) & = \frac{\lambda_{\max}(\cF)-S(\cF)}{r}.
\end{align*}
It is not hard to see from the definition that 
\[
c_\infty\le 1-\frac{\ord_{X_0}(\tI_\bullet)}{r} \le 1-\frac{e_+ - \lambda_{\max}(\cF)}{r},
\]
hence $\DNA(\cF)\le \JNA(\cF)$.
\end{defn}

\begin{lem}
We have $\LNA(\cF_\bZ)=\LNA(\cF)$ and $\LNA(\cF_c)=\LNA(\cF)+\frac{c}{r}$ for any $c\in \bR$. 
\end{lem}

\begin{proof}
By definition we have $\tI_m(\cF)=\tI_m(\cF_\bZ)$ hence the first equality follows. Choose a sufficiently large common $e_+$ for $\cF$ and $\cF_c$, we then have 
$$I_{m,i}(\cF_c)=I_{m,i-cm}(\cF)\subseteq I_{m,i-\lceil cm \rceil}(\cF)$$ and hence $\tI_m(\cF_c)\subseteq t^{-\lceil cm \rceil}\cdot \tI_m(\cF)$. It follows that $c_m(\cF_c)\le c_m(\cF)+\frac{\lceil cm \rceil}{mr}$ and therefore $c_\infty(\cF_c)\le c_\infty(\cF)+\frac{c}{r}$. Interchanging the role of $\cF$ and $\cF_c$ (note that $\cF=(\cF_c)_{-c}$) we also obtain $c_\infty(\cF)\le c_\infty(\cF_c)-\frac{c}{r}$. Thus equality holds and $$\LNA(\cF_c)=\LNA(\cF)+\frac{c}{r}$$ as desired.
\end{proof}

Combining with Lemma \ref{lem:S after translation} we immediately see that

\begin{cor} \label{cor:DNA & JNA after translation}
For any $c\in \bR$ we have $$\DNA(\cF)=\DNA(\cF_c)=\DNA(\cF_\bZ)$$ and the same equalities hold with $\DNA$ replaced by $\JNA$. \qed
\end{cor}

\subsection{K-stability}
In this section, we recall the definitions of various K-stability notions for log Fano pairs, which were first introduced in \cite{Tia97} and algebraically formulated in \cite{Don01}. Here instead of the original definitions, we will use some equivalent forms developed later (see e.g. \cite{Fuj19, Li17, LWX18b, BX19}). 

Let $(X,\Delta)$ be a log Fano pair. 

\begin{defn} \label{def:tc}
A (normal) test configuration $(\cX,\cL)$ of $(X,\Delta)$ consists of the following data:
\begin{itemize}
    \item a normal variety $\cX$ together with a flat projective morphism $\pi:\cX \to \bA^1$ and a $\pi$-ample line bundle $\cL$;
    \item a $\bG_m$-action on $(\cX,\cL)$ lifting the canonical action on $\bA^1$ such that $(\cX,\cL)|_{\pi^{-1}(\bA^1\setminus \{0\})}$ $\cong (X,-r(K_X+\Delta))\times (\bA^1\setminus \{0\})$ for some $r\in\bN_+$.
\end{itemize}
There is a natural $\bG_m$-equivariant compactification $(\overline{\cX},\overline{\cL})\to\bP^1$ of $\pi$ by gluing it with $(X,L)\times (\mathbb{P}^1\setminus\{0\})$. Let $n=\dim X$ and let $\Delta_\cX$ (resp. $\Delta_{\overline{\cX}}$) be the closure of $\Delta\times (\bA^1\setminus\{0\})$ in $\cX$ (resp. $\overline{\cX}$). The {\it generalized Futaki invariant} of $(\cX,\cL)$ is defined to be
\[
\Fut(\cX,\cL):=\frac{1}{(-K_X-\Delta)^n}\left(\frac{n}{n+1}\cdot\frac{(\overline{\cL}^{n+1})}{r^{n+1}}+\frac{(\overline{\cL}^n\cdot (K_{\overline{\cX}/\bP^1}+\Delta_{\overline{\cX}}))}{r^n}\right).
\]
We call $(\cX,\cL)$ a {\it product test configuration} if $(\cX,\Delta_\cX)\cong (X,\Delta)\times \bA^1$. Every one-parameter subgroup $\xi\colon \bG_m\to \Aut(X,\Delta)$ induces a product test configuration and we denote the corresponding generalized Futaki invariant by $\Fut_{X,\Delta}(\xi)$ (or simply $\Fut(\xi)$).

\end{defn}
\begin{rem}In the original definition of K-(semi,poly)stability in \cite{Tia97, Don01}, one looks at the sign of $\Fut(\cX,\cL)$ for all test configurations $(\cX,\cL)$. 
In the below we use a form which fits better for arguments in this paper. 
\end{rem}

\begin{defn-thm}
We say that $(X,\Delta)$ is
    \begin{enumerate}
        \item K-semistable if $\beta_{X,\Delta}(E)\ge 0$ for all divisors $E$ over $X$;
        \item K-stable if $\beta_{X,\Delta}(E)>0$ for all divisors $E$ over $X$;
        \item uniformly K-stable if there exists a constant $c>0$ depending only on $(X,\Delta)$ such that $\beta_{X,\Delta}(E)\ge c\cdot S_{X,\Delta}(E)$ for all divisors $E$ over $X$;
        \item K-polystable if it is K-semistable and any test configuration $(\cX, \Delta_{\cX},\cL)$ of $(X,\Delta)$ with K-semistable central fiber is a product test configuration.
    \end{enumerate}
\end{defn-thm}
\begin{proof} The equivalences of the original definitions of K-semistability, K-stability and uniform K-stability as in \cite{Tia97, Don01, BHJ17}  to (1)-(3) are proved in \cites{Fuj19,Li17, BX19}. And the equivalence of K-polystability with (4) is proved in \cite{LWX18b}.
\end{proof}

It is also known (see \cite{BJ17}) that $(X,\Delta)$ is K-semistable (resp. uniformly K-stable) if and only if $\delta(X,\Delta) \ge 1$ (resp. $>1$). Moreover, it is proved in \cite{BBJ15,Fuj19, Li17, Li19} the various K-stability notions are indeed equivalent to the corresponding Ding-stability notions. 

\subsection{K-moduli} \label{ss-kmoduli}
By combining the works in \cites{Jia17, BX19, ABHX19, BLX19, Xu19}, we have the following theorem.

\begin{thm}[K-moduli]\label{c-kmoduli} 
The moduli functor $\cM^{\rm Kss}_{n,v}$ of $n$-dimensional K-semistable $\mathbb{Q}$-Fano varieties of volume $v$, which sends $S\in {\sf Sch}_k$ to
 \[
\mathcal{M}^{\rm Kss}_{n,v}(S) = 
 \left\{
  \begin{tabular}{c}
\mbox{Flat proper morphisms $X\to S$, whose geometric fibers }\\
\mbox{are $n$-dimensional K-semistable $\bQ$-Fano varieties }\\
\mbox{with volume $v$, satisfying Koll\'ar's condition}
  \end{tabular}
\right\}
\]
is an Artin stack of finite type and admits a good moduli space $\phi\colon\mathcal{M}^{\rm Kss}_{n,v}\to M^{\rm Kps}_{n,v}$ as a separated algebraic space, whose geometric points are in bijection with $n$-dimensional K-polystable $\mathbb Q$-Fano varieties of volume $v$. 
\end{thm}
Though in general the properness of $M^{\rm Kps}_{n,v}$ remains a challenging problem, we have the following theorem established in \cite{LWX19} (see also \cite{Oda15}), whose proof relies on deep analytic results (c.f. \cite{DS14, CDS, Tia15}).
\begin{thm}\label{t-smoothablelwx}
 Let $k=\bC$. Denote by $\cM^{\rm sm, Kss}_{n,v}$ the open substack of $\cM^{\rm Kss}_{n,v}$ whose geometric points correspond to K-semistable smooth Fano varieties.   
Denote by $\overline{M}^{\rm sm, Kps}_{n,v}$ the closure of the image $\phi(\cM^{\rm sm, Kss}_{n,v})$ in $M^{\rm Kps}_{n,v}$, i.e., geometric points of $\overline{M}^{\rm sm, Kps}_{n,v}$ correspond to $n$-dimensional K-polystable Fano varieties with volume $v$ which are smoothable.
Then  $\overline{M}^{\rm sm, Kps}_{n,v}$  is proper. 
\end{thm}

We will also discuss the log version of the above theorems. In fact, Theorem \ref{c-kmoduli} can be extended to the log version, thanks to the recent work \cite{Kol19}.

We call $(X,cD)\to S$ a family of pairs (where $c$ is a rational number) if 
\begin{enumerate}
\item $X\to S$ is proper and flat;
\item $D$ is a K-flat family of divisors on $X$ (see \cite{Kol19}); and
\item $-K_{X/S}-cD$ is $\bQ$-Cariter.
\end{enumerate}
It is called a family of (resp. K-semistable) log Fano pairs if in addition the geometric fibers $(X_s, cD_s)$ are (resp. K-semistable) log Fano pairs.
Since $D$ is integral, the coefficients of $cD_s$ are contained in $I:=\{nc \ |\  n\in \bN\}\cap [0,1]$. 

Denote by $\cM^{\rm Kss}_{n,v,c}$ the functor
  \[
\mathcal M^{\rm Kss}_{n,v,c}(S) = 
 \left\{
  \begin{tabular}{c}
\mbox{ $(X,\Delta:=cD)\to S$ a family of K-semistable log Fano pairs,}\\
\mbox{ with $\dim(X)=n$ and $(-K_{X_s}-\Delta_s)^n=v$}
  \end{tabular}
\right\}.
\]

 \begin{thm}[K-moduli for pairs]\label{c-kmodulipair} 
The moduli functor $\mathcal{M}^{\rm Kss}_{n,v,c}$ is an Artin stack of finite type and admits a good moduli space $\phi\colon\mathcal{M}^{\rm Kss}_{n,v,c}\to M^{\rm Kps}_{n,v,c}$ as a separated algebraic space, whose geometric points are in bijection with $n$-dimensional K-polystable log Fano pairs of volume $v$ whose coefficients are in $I$. 
\end{thm}

The proof is just putting all known ingredients together, and one can exactly follow the proof of Theorem \ref{c-kmoduli}. In fact,  the boundedness result of \cites{Jia17} can be replaced by the corresponding result in the log version (see \cite{Che18, LLX18}) and other main technical results in \cites{BX19, ABHX19, BLX19, Xu19} are already established for log pairs. The only originally missing ingredient, which is the definition of locally stable family of log pairs over a general base, is now treated in \cite{Kol19}. Thus we can prove Theorem \ref{c-kmodulipair} in the way as the arguments for \cite{BX19}*{Corollary 1.4} and \cite{Xu19}*{Corollary 1.5}, which we include here for reader's convenience.  

\begin{proof}
By \cite{Che18} or \cite{LLX18}*{Corollary 6.14} the set $\mathcal{M}_{n,v,c}^{\rm Kss}(k)$ is bounded. Hence, there exists a positive integer $M$ so that $-M(K_{X}+\Delta)$ is a very ample Cartier divisor for all $[(X,\Delta)] \in \mathcal{M}_{n,v,c}^{\rm Kss}(k)$.
Furthermore, the set of Hilbert functions $m\mapsto \chi\big( \omega_{X}^{[-mM]}(-mM\Delta)\big)$ with $[(X,\Delta)] \in  \mathcal{M}_{n,v,c}^{\rm Kss}(k)$ is finite. Moreover, there are only finitely many possible degrees $d=D\cdot (-K_{X_s}-\Delta_s)^{n-1}$.

For every such Hilbert function $h$, consider the subfunctor $\mathcal{M}_{h}^{\rm Kss} \subset \mathcal{M}_{n,v,c}^{\rm Kss}$
parametrizing K-semistable log Fano pairs with Hilbert function $h$. Note that $\mathcal{M}_{n,v,c}^{\rm Kss}= \coprod_{h} \mathcal{M}_{h}^{\rm Kss}$. Set $N:=h(1)-1$, and let ${\rm Hilb}_h(\PP^N)$ be the Hilbert scheme parametrizing closed subschemes of $\PP^N$ 
with Hilbert polynomial $h$.

Next, let $U \subset  {\rm Hilb}_h(\PP^N)$ denote the open subscheme parameterizing normal, Cohen-Macaulay varieties.
By \cite[Theorem 98]{Kol19}, there is a separated $U$-scheme $W_1$ of finite type which parametrizes K-flat divisors $D$ with degree $d$ for all possible $d$ as above. Write $(\fX,\fD)\to  W_1$ for the corresponding universal family.

By \cite[Theorem 3.11]{HK04}, there is a locally closed subscheme $W_2 \subset W_1$ such that a map
$T\to W_1$ factors through $W_2$ if and only if there is an isomorphism 
$$\omega_{\fX_T/ T}^{[-M]}(-cM\fD_T) \simeq \mathcal{L}_T \otimes \cO_{\fX_T}(1),$$ 
where $\mathcal{L}_T$ is the pullback of a line bundle from $T$ and $\fD_T$ is the divisorial pull back of $\fD$. In particular, $(\fX_{W_2},\fD_{W_2})\to W_2$ is a $\bQ$-Gorenstein family of log Fano pairs. Since being klt is an open condition in $\bQ$-Gorenstein families, there exists an open subscheme $W_3$ of $W_2$ parametrizing log Fano pairs.
By \cite{BLX19} or \cite{Xu19} (applied to the family $(\fX,c\fD)$ of log pairs over the normalization of $W_3$), we see that
$$W: =\{t \in W_3 \, \vert \, (\fX_{\overline{t}},c\fD_{\overline{t}})  \text{ is K-semistable} \}$$
is open in $W_3$. 

As a consequence of the above discussion, $\mathcal{M}^{\rm Kss}_{h}\simeq [W/{\rm PGL}(N+1)]$ is an Artin stack of finite type. Taking the disjoint union over all $h$ yields $\cM^{\rm Kss}_{n,v,c}$. Since it satisfies $\Theta$-reductivity and $S$-completeness by \cite{ABHX19}, it admits a good moduli space by \cite{AHLH18} which is $M^{\rm Kps}_{n,v,c}$.
\end{proof}


The following theorem is a generalization of \cite{LWX19} (see  Theorem \ref{t-smoothablelwx}) to the log case.

\begin{thm}[{\cite[Theorem 1.1]{ADL19}}]\label{t-logsmoothable}
Let $k=\bC$. Denote by $\cM^{\rm sm, Kss}_{n,v,c}\subset  \cM^{\rm Kss}_{n,v,c}$ the open substack whose geometric points correspond to log Fano pairs $(X,cD)$ where $X$ and $D$ are smooth and $D\sim_{\mathbb{Q}} -rK_X$ for some $r\in \bQ_{+}$. Denote by $\overline{M}^{\rm sm, Kps}_{n,v,c}$ the closure of $\phi(\cM^{\rm sm, Kss}_{n,v,c})$ in $M^{\rm Kps}_{n,v,c}$, then  $\overline{M}^{\rm sm, Kps}_{n,v,c}$  is proper. 
\end{thm}

\subsection{CM line bundle}

In this section we recall the definition and some basic properties of CM line bundles.

\begin{defn}
Let $f\colon (X,\Delta)\to S$ be a family of pairs of relative dimension $n$ such that $-(K_{X/S}+\Delta)$ is $f$-ample. Let $s>0$ be an integer such that $L:=-s(K_{X/S}+\Delta)$ is Cartier. By \cite{MFK-GIT}*{Appendix to Chapter 5, Section D}, we have a Knudsen-Mumford expansion
\[
\det f_*\cO_X(mL) \cong \bigotimes_{i=0}^{n+1} \cM_i^{\otimes\binom{m}{i}}
\]
for all sufficiently large $m\in\bN$ and for some line bundles $\cM_i$ on $S$. The CM ($\bQ$-)line bundle of the family is then defined as
\[
\lambda_{f,\Delta} := -s^{-n-1} \cM_{n+1}.
\]

By \cite{CP18}*{Proposition 3.7}, this is equivalent to the original definition in \cite{PT09}. For generic log Fano families, we also have an intersection formula for the CM line bundle: if $f\colon (X,\Delta)\to S$ is a generic log Fano family $($see Definition \ref{d-genericlog}$)$ over a proper variety $S$, then 
\[
\lambda_{f,\Delta} = -f_*(-(K_{X/S}+\Delta))^{n+1}.
\]
\end{defn}

The formation of CM line bundle is also compatible with base change in the following sense.

\begin{prop}[\cite{CP18}*{Lemma 3.5 and Proposition 3.8}]
Let $f\colon (X,\Delta)\to S$ be a family of log Fano pairs and let $\phi\colon S'\to S$ be a morphism. Let $f'\colon (X',\Delta')\to S'$ be the base change of $f$ to $S'$. Then $\lambda_{f',\Delta'}=\phi^*\lambda_{f,\Delta}$.
\end{prop}

As a consequence, we get a well defined CM ($\bQ$-)line bundle $\lambda_{\CM}$ on the moduli stack $\cM^{\rm Kss}_{n,v,c}$ parametrizing K-semistable log Fano pairs with dimension $n$, anti-log-canonical volume $v$ and coefficient set $I=c \bN\cap[0,1]$.

\begin{prop}\label{p-CMLa}
There exists a positive integer $k$, such that $\lambda^{\otimes k}_{\CM}$ descends to the good moduli space $M^{\rm Kps}_{n,v,c}$.
\end{prop}
\begin{proof} 
This is well known. See e.g. \cite{LWX18}*{\S 5} or \cite[Lemma 10.2]{CP18}. 
\end{proof}

Therefore, $\lambda_{\CM}$ descends to a $\bQ$-line bundle on $M^{\rm Kps}_{n,v,c}$, which we denote by $\Lambda_{\CM}$.


\subsection{Harder-Narasimhan filtration} \label{sec:HN}In this section, we introduce some basic facts of the Harder-Narasimhan filtration. A similar study in the setting of Arakelov geometry appeared in \cite{Che10}.

Let $C$ be a smooth projective curve of genus $g$. Given a vector bundle $\cE$ on $C$, its slope is defined to be 
\[
\mu(\cE)=\frac{\deg(\cE)}{\rk(\cE)}.
\]
We also define $\mu_{\max}(\cE)$ to be the maximal slope of nonzero subbundles $\cE'\subseteq \cE$ and $\mu_{\min}(\cE)$ the minimal slope of nonzero quotient bundles $\cE\twoheadrightarrow \cE'$. For any vector bundle $\cE$ on the curve $C$, we can define a Harder-Narasimhan filtration $\cF_\HN$ on $\cE$ by setting 
$$\cF_\HN^\lambda\cE:= \mbox {union of all subbundles $\cE'\subseteq \cE$ with }\mu_{\min}(\cE')\ge \lambda.$$
In other words, $\cF_\HN^\lambda\cE$ is the subbundle $\cE_{i}$ in the Harder-Narasimhan filtration
$$0=\cE_0\subset \cE_1\subset \cdots \cE_{i-1}\subset \cE_{i}\subset \cE_{i+1}\subset \cdots =\cE$$ of $\cE$, such that the semistable vector bundle $\cE_i/\cE_{i-1}$ has slope at least $\lambda$ while the slope of $\cE_{i+1}/\cE_i$ is strictly less than $\lambda$. 

\medskip

Let $f:(X,\Delta)\to C$ be a surjective morphism from a normal projective pair and let $L$ be an $f$-ample Cartier divisor on $X$. Note that $f_*\cO_X(mL)$ is locally free since it is torsion free and $C$ is a smooth curve. We also fix a point $t\in C$ such that the restriction map $f_*\cO_X(mL) \to H^0(X_t, mL_t)$ is surjective for all $m\in \bN$ (this holds when $t\in C$ is general or $L$ is sufficiently ample). 




\begin{defn-lem}\label{l-HNfil}
Assume that the general fibers of $f$ are klt. Then by restricting to the fiber $(X_t, L_t)$, the HN-filtrations $\cF_\HN$ of $\cR_m:=f_*\cO_X(mL)$ $(m\in \bN)$ induce a linearly bounded multiplicative filtration $($which we also denote by $\cF_\HN )$ of
\[
R_t:=\bigoplus_{m\in \bN} R_{t,m} = \bigoplus_{m\in \bN} H^0(X_t, mL_t),
\]
called the Harder-Narasimhan filtration $($HN-filtration$)$ of $R_t$ induced by the family $f$.
\end{defn-lem}

\begin{proof}
Let $\cE$ be the image of the multiplication map
\[
\cF^{\lambda}_{\HN} \cR_m \otimes \cF^{\lambda'}_{\HN} \cR_{m'} \to \cR_{m+m'}.
\]
By \cite[Theorem 3.1.4]{HL00}, we have
$$\mu_{\min}(\cE)\ge \mu_{\min}(\cF^{\lambda}_{\HN} \cR_m \otimes \cF^{\lambda'}_{\HN} \cR_{m'})=\mu_{\min}(\cF^{\lambda}_{\HN} \cR_m)+\mu_{\min}(\cF^{\lambda'}_{\HN} \cR_{m'})\ge \lambda+\lambda',$$ 
hence $\cE\subseteq \cF^{\lambda+\lambda'}_{\HN} \cR_{m+m'}$, which implies that $\cF_\HN$ is multiplicative.

Fixed a point $P\in C$. Since $L$ is $f$-ample, we may choose some $c\in\bQ_{>0}$ such that $M=L+c f^*P$ is ample. Then for $m\gg 1$, $\lfloor mM \rfloor -(K_{X/C}+\Delta)$ is ample, hence by \cite{CP18}*{Proposition 6.4}, $f_*\cO_X(\lfloor mM \rfloor)=\cR_m\otimes \cO_C(\lfloor cmP \rfloor)$ is a nef vector bundle. In particular, $\mu_{\min}(\cR_m)\ge -cm$ for all $m\gg 1$, thus $\cF_\HN$ is linearly bounded from below.

Similarly, let $b\in\bQ_{>0}$ be such that $N=L-bf^*P$ is not pseudo-effective. Then for $m\gg 1$, we have 
$$H^0(C,\cR_m\otimes \cO_C(-\lfloor bmP \rfloor))=H^0(C,f_*\cO_X(\lceil mN \rceil))=H^0(X,\lceil mN \rceil)=0,$$ hence by \cite{CP18}*{Proposition 5.4}, $\mu_{\max}(\cR_m\otimes \cO_C(-\lfloor bmP \rfloor))<2g$; equivalently, we have $\mu_{\max}(\cR_m)<2g+bm$. This shows that $\cF_{\HN}$ is linearly bounded from above.
\end{proof}

Following the above argument, we define
$$\lambda_-(L)=\sup\{c\in\bR\,|\,L-cf^*P \mbox{ is nef}\,\}$$
where $P\in C$ is a closed point and 
$$\lambda_+(L)=\sup\{c\in\bR\,|\,L-cf^*P \mbox{ is pseudo-effective}\,\}.$$
Clearly the definition does not  depend on the choice of $P$.

We will use the following simple observation. 

\begin{lem} \label{l-nefinQ}
$\lambda_-(L)\in\bQ$.
\end{lem}

\begin{proof}
Since $\lambda:=\lambda_-(L)$ is the nef threshold of $L$ with respect to $f^*P$, $L-\lambda f^*P$ is nef but not ample. By the Nakai-Moishezon criterion, we have $((L-\lambda f^*P)^d\cdot Z)=0$ for some subvariety $Z\subseteq X$ of dimension $d$, which reduces to $(L^d\cdot Z)=d\lambda (L^{d-1}\cdot f^*P\cdot Z)$ (note that $(f^*P\cdot f^*P)=0$). It is then clear that $\lambda\in\bQ$.
\end{proof}

\begin{prop} \label{p-supp}
The Duistermaat-Heckman measure of the Harder-Narasimhan filtration $\cF_{\HN}$ precisely supports on $[\lambda_-(L),\lambda_+(L)]$.
\end{prop}

\begin{proof}
Let $\lambda_{\min}=\lambda_{\min}(\cF_{\HN})$ and $\lambda_{\max}=\lambda_{\max}(\cF_{\HN})$. From the proof of Lemma \ref{l-HNfil}, we have seen that $\lambda_{\min}\ge \lambda_-:=\lambda_-(L)$ and $\lambda_{\max}\le \lambda_+:=\lambda_+(L)$, so it suffices to establish the opposite inequalities.

We first prove that $\lambda_{\min}\le \lambda_-$. For any rational number $c > \lambda_-$, $M=L-cf^*P$ is not nef by our choice of $\lambda_-$. Since $(X,\Delta)$ is klt along the general fiber of $f$ and $L$ is $f$-ample, by \cite{CDB13}*{Corollary 4.7} we know that there exists a subvariety $Z\subseteq X$ such that $f(Z)=C$ and
\[
\mult_Z(\|M\|):=\lim_{m\to \infty} \frac{\mult_Z(|mM|)}{m}>0.
\]
Let $Z_t$ be the restriction of $Z$ to $X_t$, then since 
\[
\cF_{\HN}^{cm+2g} \cR_m \subseteq {\rm Im } (H^0(X,mM)\otimes \cO_C \to \cR_m)
\]
for all sufficiently divisible $m$ by \cite{CP18}*{Proposition 5.7}, we have 
$$\mult_{Z_t}(\cF_{\HN}^{cm+2g} R_{t,m})>\epsilon m$$
for some constant $\epsilon>0$ independent of $m$. It follows that $\vol(\cF_{\HN} R_t^{(c)})<\vol(L_t)$ and therefore $\lambda_{\min}<c$ (see e.g. \cite{BHJ17}*{Corollary 5.4}). Letting $c\to \lambda_-$, we obtain $\lambda_{\min}\le \lambda_-$.

\medskip

We next prove that $\lambda_{\max}\ge \lambda_+$. Let $c' \in (\lambda_-, \lambda_+)$ be a rational number. Then $M'=L-c'f^*P$ is big, thus for sufficiently divisible $m$, 
\[
H^0(X,mM')=H^0(C,\cR_m \otimes \cO_C(-mc'P))\neq 0.
\]
In particular, $\mu_{\max}(\cR_m\otimes \cO_C(-mc'P))\ge 0$, which implies that $\lambda_{\max}\ge c'$. Letting $c'\to \lambda_+$, we obtain $\lambda_{\max}=\lambda_+$. The proof is now complete.
\end{proof}

\section{Reduced uniform K-stability}\label{ss-ruk}
In this section, we discuss a relatively more recent notion, the {\it reduced uniform K-stability}. This concept gives a suitable extension of the definition of uniform K-stability to the case when the automorphism group is non-discrete. It was first introduced in \cite{His16}, and systematically developed in \cite{Li19}. In particular, it was shown there that it is equivalent to the existence of a K\"ahler-Einstein metric (see Theorem \ref{thm:Li's valuative T-uKs}). In this section, we will first give the introduction of the basic notion and results, mostly following \cite{Li19}. Then we will establish a few new results, some of which will be needed later to get our main theorems.   We also note that later in Section \ref{s-tbeta}, we will establish a criterion to test reduced uniform K-stability using the $\beta$-invariant (for filtrations) that we are going to introduce. 

\subsection{Definition and characterization} \label{ss-reduks}
In this section, we will recall the definition of reduced uniform K-stability and some related results. Most of them are from \cite{Li19}.

\medskip

Let $(X,\Delta)$ be a log Fano pair  with an action by a torus $T\cong \bG_m^s$. Fix some integer $r>0$ such that $L:=-r(K_X+\Delta)$ is Cartier and as before let $R=R(X,L)$. Let $M=\Hom(T,\bG_m)$ be the weight lattice and $N=M^*=\Hom(\bG_m,T)$ the co-weight lattice. Then $T$ naturally acts on $R$ and we have a weight decomposition $R_m=\bigoplus_{\alpha\in M} R_{m,\alpha}$ where 
\[
R_{m,\alpha}=\{s\in R_m\,|\,\rho(t)\cdot s = t^{\la \rho,\alpha \ra}\cdot s\text{ for all }\rho\in N\text{ and }t\in k^*\}.
\]
Consider a $T$-equivariant filtration $\cF$ on $R=\bigoplus_mH^0(X,mL)$, i.e., $s\in \cF^\lambda R$ if and only if $g\cdot s\in \cF^\lambda R$ for any $g\in T$. We then have a similar weight decomposition 
\[
\cF^\lambda R_m=\bigoplus_{\alpha\in M} (\cF^\lambda R_m)_\alpha
\]
where $(\cF^\lambda R_m)_\alpha := \cF^\lambda R_m \cap R_{m,\alpha}$.

\begin{defn}\label{d-quotient}
For $\xi\in N_{\bR}=N\otimes_\bZ \bR$, we define the $\xi$-twist $\cF_{\xi}$ of the filtration $\cF$ in the following way: for any $s\in  R_{m,\alpha}$, we have
$$s\in \cF^\lambda_{\xi} R_m \mbox{ if and only if } s\in \cF^{\lambda_0} R_m \mbox{ where }\lambda_0=\lambda-\la\alpha, \xi \ra,$$
in other words, 
 $$\cF^\lambda_{\xi} R_m = \bigoplus_{\alpha\in M}  \cF^{\lambda-\la\alpha, \xi \ra}{R}\cap R_{m,\alpha}.$$
\end{defn}

One can easily check that $\cF_{\xi}$ is a linearly bounded multiplicative filtration if $\cF$ is. 

Let $Z=X\!/\!/_{\rm chow}T$ be the Chow quotient (so $X$ is $T$-equivariantly birational to $Z\times T$). Then the function field $k(X)$ is (non-canonically) isomorphic to the quotient field of $$k(Z)[M]=\bigoplus_{\alpha\in M} k(Z)\cdot 1^\alpha.$$
For any valuation $\mu$ over $Z$ and $\xi\in N_\bR$, one can associate a $T$-invariant valuation $v_{\mu,\xi}$ over $X$ such that
\begin{eqnarray}\label{e-valuation}
v_{\mu,\xi}(f)=\min_{\alpha}(\mu(f_{\alpha})+\la \xi, {\alpha}\ra)
\end{eqnarray}
for all $f=\sum_{\alpha\in M} f_{\alpha}\cdot 1^\alpha \in k(Z)[M]$. Indeed, every valuation $v\in \Val^T(X)$ (i.e. the set of $T$-invariant valuations) is obtained in this way (see e.g. the proof of \cite{BHJ17}*{Lemma 4.2}) and we get a (non-canonical) isomorphism $\Val^T(X)\cong \Val(Z)\times N_\bR$. For any $v\in\Val^T(X)$ and $\xi\in N_{\bR}$, we can therefore define the twisted valuation $v_{\xi}$ as follows: if $v=v_{\mu,\xi'}$, then
$$ v_{\xi}:=v_{\mu, \xi'+\xi}. $$
One can check that the definition does not depend on the choice of the birational map $X\dashrightarrow Z\times T$. When $\mu$ is the trivial valuation, the valuations $\wt_\xi:=v_{\mu,\xi}$ are also independent of the birational map $X\dashrightarrow Z\times T$.

\begin{defn}
Let $T$ be a torus acting on a log Fano pair $(X,\Delta)$. For any $T$-equivariant filtration $\cF$ of $R$, its {\it reduced $\bfJ$-norm} is defined as:
$$\JNA_T(\cF):=\inf_{\xi\in N_\bR}\JNA(\cF_{\xi}).$$
The reduced $\bfJ$-norm $\JNA(\cX,\cL)$ of a $T$-equivariant test configuration $(\cX,\cL)$ of $(X,\Delta)$ is defined to be the reduced $\bfJ$-norm of its associated $\bZ$-filtration (Example \ref{expl:filtration from tc}).
\end{defn}

\begin{lem} \label{lem:DNA twist & JNA min}
Let $(X,\Delta)$ be a K-semistable log Fano pair with a torus $T$ action. Then for any filtration $\cF$ of $R=R(X,-r(K_X+\Delta))$, we have:
\begin{enumerate}
    \item $\DNA(\cF_\xi)=\DNA(\cF)$ and $S(\cF_\xi)=S(\cF)$ for any $\xi\in N_\bR$.
    \item The function $\xi\mapsto \JNA(\cF_\xi)$ is continuous and there exists $\xi\in N_\bR$ such that $\JNA_T(\cF)=\JNA(\cF_\xi)$.
\end{enumerate}
\end{lem}

\begin{proof}
Since $(X,\Delta)$ is K-semistable, the Futaki invariant vanishes on all product test configurations, i.e. $\Fut_{X,\Delta}(\xi)=0$ for all $\xi\in N$, thus (1) follows from \cite{Li19}*{Lemma 3.10} and (2) follows from \cite{Li19}*{Lemma 3.15}.
\end{proof}

\begin{defn}[Reduced uniform stability, \cites{His16, Li19}] \label{defn:red uks}
Let $\eta>0$. A log Fano pair $(X,\Delta)$ is called {\it reduced uniformly Ding-stable with slope at least $\eta$} if for some torus $T\subseteq \Aut(X,\Delta)$ and for any $T$-equivariant test configuration $(\cX,\cL)$ of $(X,\Delta)$, we have
\begin{equation} \label{eq:red-uks via tc}
    \DNA(\cX,\cL)\ge \eta\cdot \JNA_T(\cX,\cL).
\end{equation}
A log Fano pair $(X,\Delta)$ is said to be {\it reduced uniformly Ding-stable} if it is reduced uniformly Ding-stable with some slope $\eta>0$. We define the reduced uniform K-stability in a similar way by replacing $\DNA$ with the generalized Futaki invariant.
\end{defn}

\begin{rem} \label{rem:choice of T}
Our definition clearly does not depend on the torus $T$. In fact, from the definition we see that if $(X,\Delta)$ is reduced uniformly Ding-stable (resp. K-stable) with respect to some torus $T\subseteq \Aut(X,\Delta)$, then $(X,\Delta)$ is K-polystable and $T$ has to be a maximal torus of $\Aut(X,\Delta)$. Since any two maximal tori are conjugate to each other, 
to verify the above definition it is equivalent to consider one maximal torus $T\subset \Aut(X,\Delta)$.
\end{rem}
\begin{rem} It can be easily seen that our definition of reduced uniform stability is equivalent to the notion of $G$-uniform stability in \cite{Li19} as long as $G$ contains a maximal torus. There have been other attempts to define `uniform stability' when there is a group action. See \cite[Remark 1.6]{Li19} of the relation between these notions. For our study in this paper, only the concept of reduced uniform K-stability is relevant, which we believe to be the most intrinsic one, since we do not have to specify any group.
\end{rem}

The following theorem is one of the main results of \cite{Li19}.

\begin{thm}[\cite{Li19}*{Theorem 1.2 and 1.3}] \label{thm:Li's valuative T-uKs}
Let $(X,\Delta)$ be a log Fano pair and $T$ a maximal torus of $\Aut(X,\Delta)$. The following are equivalent:
\begin{enumerate}
\item $(X,\Delta)$ is reduced uniformly Ding-stable,
\item $(X,\Delta)$ is reduced uniformly K-stable, 
\item $(X,\Delta)$ is K-semistable, and there exists some $\delta>1$ such that for any $T$-invariant valuation $v$, there exists some $\xi\in N_\bR$, such that 
$A_{X,\Delta}(v_{\xi})\ge \delta \cdot S(v_{\xi})$. 
\end{enumerate}
When the base field $k=\bC$, they are also equivalent to
\begin{enumerate}
\setcounter{enumi}{3}
\item $(X,\Delta)$ admits a K\"ahler-Einstein metric.
\end{enumerate}
\end{thm}

The following is the main conjecture about reduced uniform K-stability (see e.g. \cite{Li19}*{Remark 1.6(2)}, which is the K-polystable version of the conjecture that K-stability implies uniform K-stability (see e.g. \cite{BX19}*{Conjecture 1.5}). 
\begin{conj}\label{c-kpoly}
If $(X,\Delta)$ is K-polystable, then it is reduced uniformly K-stable. 
\end{conj}

Next, we extend \eqref{eq:red-uks via tc} as well as \cite{Li19}*{Proposition 3.22} to arbitrary filtrations. It  will be needed in later arguments. 

\begin{prop} \label{prop:D>=J any filtration}
Let $(X,\Delta)$ be a log Fano pair and let $T\subseteq \Aut(X,\Delta)$ be a maximal torus. Assume that $(X,\Delta)$ is reduced uniformly K-stable with slope at least $\eta>0$. Then for any $T$-equivariant filtration $\cF$ of $R=R(X,-r(K_X+\Delta))$, we have 
\[
\DNA(\cF)\ge \eta\cdot \JNA_T(\cF).
\]
\end{prop}

\begin{proof}
By Corollary \ref{cor:DNA & JNA after translation}, we may replace $\cF$ by a translation of $\cF_\bZ$ and assume that $\cF$ is an $\bN$-filtration. Let $\cF_m$ be the $m$-th approximating filtration of $\cF$, i.e.,
\begin{enumerate}
    \item for $m'<m$, $\cF_m^\lambda R_{m'} = R_{m'}$ if $\lambda\le 0$ and $\cF_m^\lambda R_{m'}=0$ if $\lambda > 0$,
    \item for $m'=m$, $\cF_m^\lambda R_{m'} =\cF^\lambda R_{m'}$ for all $\lambda$,
    \item for $m'>m$, $\cF_m^\lambda R_{m'} =\sum_{\vec{\mu}} \cF^{\mu_1} R_m \cdots \cF^{\mu_s} R_m\cdot R_{m'-ms}$ where the sum runs through all $s\in \bN_+$ and $\vec{\mu}=(\mu_1,\cdots,\mu_s)\in\bR^s$ such that $ms\le m'$ and $\mu_1+\cdots+\mu_s\ge \lambda$ (it turns out that the sum can be written as a finite sum).
\end{enumerate}
(Roughly speaking, $\cF_m$ is the coarsest filtration of $R$ such that $\cF_m^\lambda R_m = \cF^\lambda R_m$ for all $\lambda$.) 
We claim that (c.f. \cite{Li19}*{Proposition 3.16})
\begin{equation} \label{eq:J_T,m converge to J_T}
    \mathbf{I}:=\lim_{m\to \infty} \JNA_T(\cF_m) = \JNA_T(\cF).
\end{equation}

To see this, first note that $\lim_{m\to \infty}\lambda_{\max}(\cF_m)=\lambda_{\max}(\cF)$ by \cite{BHJ17}*{Corollary 5.4}. Using Fujita's approximation theorem for graded linear series \cite{LM09}*{Theorem 3.5} we also have $\lim_{m\to \infty} \vol(\cF_m R^{(x)}) = \vol(\cF R^{(x)})$ for all $x\in\bR$ and hence $\lim_{m\to \infty} S(\cF_m)=S(\cF)$ by the dominated convergence theorem. Since $(X,\Delta)$ is reduced uniformly K-stable, it is K-polystable and thus by Lemma \ref{lem:DNA twist & JNA min} we have $S(\cF_\xi)=S(\cF)$ and $S(\cF_{m,\xi})=S(\cF_m)$  for any $\xi\in N_\bR$ and any $T$-equivariant filtration $\cF$ on $R$. 
 In particular, by Definition \ref{defn:D^NA},
$$\lim_{m\to \infty} \JNA(\cF_{m,\xi}) = \JNA(\cF_\xi)\mbox{ for any fixed }\xi\in N_\bR.$$ 

Let $\lambda_{m,\alpha}=\sup\{\lambda\,|\,(\cF^\lambda R_m)_\alpha \neq 0\}$ and let $\Gamma_m = \{\alpha\in M\,|\,R_{m,\alpha}\neq 0\}$. Then there exists a bounded region $P\subseteq M_\bR$ such that $\frac{1}{m}\Gamma_m\subseteq P$ for all $m\in\bN$. It follows that
$$\lambda_{\max}(\cF_{m,\xi}) = \sup_{\alpha\in \Gamma_m} \frac{\lambda_{m,\alpha}+\la \alpha,\xi \ra}{m} \le \lambda_{\max}(\cF)+\sup_{\alpha\in \Gamma_m} \frac{\la \alpha,\xi \ra}{m}\le \lambda_{\max}(\cF)+C|\xi|$$
for some constant $C>0$ that only depends on $P$. Thus the functions $\JNA(\cF_{m,\xi})$ ($m\in\bN$) are equicontinuous on $N_\bR$, as the above estimate implies for any $m$,
$$|\lambda_{\max}(\cF_{m,\xi})-\lambda_{\max}(\cF_{m,\xi_0})|\le C|\xi-\xi_0|.$$ On the other hand, as in the proof of \cite{Li19}*{Lemma 3.15 and Proposition 3.16}, there exist some constants $C_1,C_2>0$ such that 
$$\JNA(\cF_\xi)\ge C_1|\xi|-C_2\mbox{ and }\JNA(\cF_{m,\xi})\ge C_1|\xi|-C_2 \mbox{ for all $m\gg 0$ and }\xi\in N_\bR,$$ 
so the infima $\inf_{\xi\in N_\bR} \JNA(\cF_{m,\xi})$ and $\inf_{\xi\in N_\bR} \JNA(\cF_\xi)$ are achieved on a fixed compact subset $\Xi\subseteq N_\bR$. By the Arzel\`a-Ascoli theorem, the convergence 
$$\JNA(\cF_{m,\xi})\to \JNA(\cF_\xi) \ \ (m\to \infty)$$ 
is uniform over $\Xi$ and hence we also get the convergence of infima
$$\mathbf{I} = \lim_{m\to \infty} \inf_{\xi\in N_\bR} \JNA(\cF_{m,\xi}) = \inf_{\xi\in N_\bR} \JNA(\cF_\xi) = \JNA_T(\cF).$$
as claimed \eqref{eq:J_T,m converge to J_T}.

Since $\cF_m$ is finitely generated, it comes from a test configuration of $(X,\Delta)$ \cite{BHJ17}*{Proposition 2.15}, hence by the reduced uniform K-stability of $(X,\Delta)$ we have $\DNA(\cF_m)\ge \eta\cdot\JNA_T(\cF_m)$ for some constant $\eta>0$ depending only on $(X,\Delta)$. By construction, we have $\tI_m(\cF_m) = \tI_m(\cF)$ and $\tI_{m\ell}(\cF_m) = \tI_m(\cF_m)^\ell$ for all $m,\ell\in \bN$, it is thus easy to check that $\lim_{m\to \infty} \LNA(\cF_m)=\LNA(\cF)$ and (see e.g. \cite{Fuj18b}*{\S 4.2})
\begin{equation} \label{eq:ding filtration}
\lim_{m\to \infty} \DNA(\cF_m) = \DNA(\cF),
\end{equation}
hence combining with \eqref{eq:J_T,m converge to J_T} we obtain $\DNA(\cF)\ge \eta\cdot\JNA_T(\cF)$. 
\end{proof}

\subsection{Further properties}

The second property is related to the behaviour of the slope $\eta$ in family, which is a partial generalization of the result in \cite{BL18} to the twisted version. It is needed later in our study of the CM line bundle. 

We first need a simple lemma.
\begin{lem} \label{lem:only fg N-filt}
Let $\eta > 0$, let $(X,\Delta)$ be a log Fano pair and let $T\subseteq \Aut(X,\Delta)$ be a maximal torus. Then $(X,\Delta)$ is reduced uniformly K-stable with slope at least $\eta>0$ if and only if
\begin{equation} \label{eq:slope eta}
    \DNA(\cF)\ge \eta \cdot \JNA_T(\cF)
\end{equation}
for all finitely generated $T$-equivariant $\bN$-filtrations $\cF$ with $\lambda_{\max}(\cF)\le 1$.
\end{lem}

\begin{proof}
It suffices to prove the backward implication. Let $\cF$ be a filtration of $R$. Choose some $M_1,M_2\gg 0$ such that the filtration $\cG$ on $R$ defined by $\cG^\lambda R_m = \cF^{\lambda M_1 - m M_2} R_m$ satisfies $\cG^0 R = R$ and $\lambda_{\max}(\cG)\le 1$. It is not hard to check that $\DNA(\cF)=M_1\cdot \DNA(\cG)$ and $\JNA(\cF)=M_1\cdot\JNA(\cG)$, hence \eqref{eq:slope eta} holds for $\cF$ if and only if it holds for $\cG$. 

Replacing $\cF$ by $\cG$, we may assume that $\cF^0 R = R$ and $\lambda_{\max}(\cF)\le 1$. Similarly, as $\DNA(\cF_\bZ)=\DNA(\cF)$ and $\JNA(\cF_\bZ)=\JNA(\cF)$, it suffices to check \eqref{eq:slope eta} for the $\bN$-filtration $\cF_\bZ$. For each positive integer $m$, let $\cF_m$ be the $m$-th approximating filtration of $\cF_\bZ$, then $\cF_m$ is a finitely generated $T$-equivariant $\bN$-filtrations with $\lambda_{\max}(\cF_m)\le 1$. If \eqref{eq:slope eta} holds for all such filtrations, then letting $m\to \infty$ we see that \eqref{eq:slope eta} holds for $\cF_\bZ$ (by \eqref{eq:J_T,m converge to J_T}) and hence for $\cF$ as well.
\end{proof}

\begin{prop} \label{prop:slope at vg pt}
Let $T$ be a torus and let $f\colon (X,\Delta)\to B$ be a $\bQ$-Gorenstein family of log Fano pairs with a fiberwise $T$-action. Let $\eta > 0$. Assume that $(X_0,\Delta_0)$ is reduced uniformly Ding-stable with slope at least $\eta$ for some $0\in B$ and $T\subseteq \Aut(X_0,\Delta_0)$ is a maximal torus. Then the same is true for very general fibers of $f$. 
\end{prop}

\begin{proof}
Let $r>0$ be a sufficiently divisible integer and let 
$$\cR:=\bigoplus_{m\in \bN} \cR_m := \bigoplus_{m\in \bN} f_*\cO_X(-mr(K_{X/B}+\Delta)).$$ 
By Remark \ref{rem:choice of T}, it suffices to show that very general fibers of $f$ are reduced uniformly Ding-stable with slope at least $\eta$ and with respect to $T$.
By Lemma \ref{lem:only fg N-filt}, this reduces to showing that $\DNA(\cF)\ge \eta \cdot \JNA_T(\cF)$ for any $T$-equivariant finitely generated $\bN$-filtration $\cF$ on 
$$R_b = \cR\otimes k(b) = \bigoplus_{m\in \bN} H^0(X_b,-mr(K_{X_b}+\Delta_b))$$
with $\lambda_{\max}(\cF)\le 1$ (where $b\in B$ is very general).

We first construct the parameter spaces of such filtrations. Let $m\in\bN_+$ and let $\cR_m=\bigoplus_{\alpha\in M} \cR_{m,\alpha}$ be the weight decomposition. If $R$ is an $\bN$-graded ring, then we denote by $R^{(m)}$ the  subring of $R$ which only consists of the elements with degree divisible by $m$. 

Clearly giving a $T$-equivariant $\bN$-filtration $\cF$ on $R^{(m)}_b$ that is generated in degree $m$ is equivalent to giving an $\bN$-filtrations on $R_{m,\alpha,b}:=\cR_{m,\alpha}\otimes k(b)$ for each $\alpha$. If moreover $\lambda_{\max}(\cF)\le 1$, then this amounts to choosing a length $m$ decreasing sequence of subspaces $\cF^i R_{m,\alpha,b}$ ($i=1,2,\cdots,m$) of each $R_{m,\alpha,b}$. Such a sequence is parametrized by the relative flag variety (of length $m$) of $\cR_{m,\alpha}$ over $B$. In other words, if we denote by $F_m$ the fiber product of these relative flag varieties (for a fixed $m\in\bN_+$) over $B$ and let $\phi_m\colon F_m\to B$ be the natural projection, then for any $b\in B$, there is a bijection between the geometric points of $\phi_m^{-1}(b)$ and the $T$-equivariant $\bN$-filtrations $\cF$ on $R_b$ with $\lambda_{\max}(\cF)\le 1$ which are generated in degree $m$. In particular, for each $t\in F_m$, we denote by $\cF_t$ the corresponding filtration on $R_{\phi_m(t)}$. Note that each irreducible component of $F_m$ (corresponding to different choices of the signature of the flag) is proper over $B$. 

By \cite{BJ17}*{Corollary 2.10}, there exist constants $\epsilon_m \ge 0$ ($m\in\bN_+$) such that $\epsilon_m\to 1$ $(m\to \infty)$ and $S(\cF)\ge \epsilon_m \cdot S_m(\cF)$ for all $m$ and all $\bN$-filtration $\cF$ on $R_0$. By \cite{Li19}*{Lemma 3.15}, for any $T$-equivariant finitely generated filtration $\cF$ on $R_0$ there exists some $\xi\in N_\bR$ such that $\DNA(\cF)\ge \eta\cdot \JNA(\cF_\xi)$; since $(X_0,\Delta_0)$ is K-semistable, we also have $S(\cF_\xi)=S(\cF)$. It follows that
\begin{align*}
    \LNA(\cF) & \ge \eta \cdot \JNA(\cF_\xi) + \frac{S(\cF)}{r} \\
     & = \eta \cdot \lambda_{\max}(\cF_\xi)+ (1-\eta)\cdot \frac{S(\cF)}{r} \\
     & \ge \eta \cdot \lambda_{\max}(\cF_\xi)+ (1-\eta)\epsilon_m\cdot \frac{S_m(\cF)}{r}.
\end{align*}
Let $(Y,\Gamma)=(X,\Delta)\times_B F_m$ and let $g\colon (Y,\Gamma)\to F_m$ be the natural projection. Let $L=-mr(K_{Y/F_m}+\Gamma)$. By construction, there exists a universal flag of $T$-invariant subbundles 
$$\cF^m \subseteq \cdots \subseteq \cF^1\mbox{\ \ \ of \ \ }g_*\cO_Y(L)=\phi_m^* \cR_m$$ whose restriction to any geometric point $t\in F_m$ is the corresponding filtration $\cF_t$ on $R_{m,t}=\phi_m^* \cR_m\otimes k(t)$. 

Let $\cI_{m,i}$ be the image of the composition $g^*\cF^i\otimes \cO_Y(-L)\to g^*g_*\cO_Y(L) \otimes \cO_Y(-L) \to \cO_Y$ and let 
\[
\tcI_m = \cI_{m,m} + \cI_{m,m-1}\cdot t + \cdots + \cI_{m,1}\cdot t^{m-1} + (t^{m}) \subseteq \cO_{Y\times \bA^1}.
\]
Then we have $\cI_{m,i}\otimes k(y)=I_{m,i}(\cF_t)$ and $\tcI_m\otimes k(t)=\tI_m(\cF_t)$ for all $t\in F_m$. As $\cF_t$ is generated in degree $m$, we also have $\tI_{m\ell}(\cF_t)=\tI_m(\cF_t)^\ell$ for all $\ell\in\bN$ and hence 
\[
\LNA(\cF_t) = \lct \big(Y_t \times \bA^1,(\Gamma_t \times\bA^1)\cdot (\tcI_{m,t})^{\frac{1}{mr}};Y_t\times \{0\} \big) + \frac{1}{r} - 1.
\]
By inversion of adjunction, this implies that the function $t\mapsto \LNA(\cF_t)$ on $F_m$ is constructible and lower semicontinuous. 

Moreover, the functions $t\mapsto S_m(\cF_t)$ and $t\mapsto \lambda_{\max}(\cF_{t,\xi})$ (for each fixed $\xi\in N_\bR$) are constant on each irreducible component of $F_m$. It follows from the Noetherian induction that for each $m\in \bN_+$, there exists an open subset $U_m\subseteq F_m$ containing $\phi_m^{-1}(0)$ such that for all $t\in U_m$, there exists some $\xi\in N_\bR$ with
\[
\LNA(\cF_t)\ge \eta \cdot \lambda_{\max}(\cF_{t,\xi})+ (1-\eta)\epsilon_m\cdot \frac{S_m(\cF_t)}{r}.
\]
Since $F_m$ is proper over $B$, $V:=\bigcap_m \phi_m(U_m)$ is the complement of a countable union of closed sets. It is nonempty as $0\in V$. Further shrinking $V$, we may also assume that every fiber over $V$ is K-semistable by \cite{BL18}*{Theorem A}. 
We therefore deduce that for any $b\in V$ and any $T$-equivariant $\bN$-filtration $\cF$ of $R_b$ with $\lambda_{\max}(\cF)\le 1$ that is generated in degree $m$, we have
\begin{equation} \label{eq:DNA>=J_T,m}
    \LNA(\cF)\ge \eta \cdot \JNA_T(\cF) + \eta\cdot \frac{S(\cF)}{r} + (1-\eta)\epsilon_m\cdot \frac{S_m(\cF)}{r}.
\end{equation}

But a filtration that is generated in degree $m$ is also generated in degree $m\ell$ for all $\ell\in \bN$ and hence \eqref{eq:DNA>=J_T,m} remains true if we replace $m$ by $m\ell$. Letting $\ell\to \infty$ we obtain $\DNA(\cF)\ge \eta\cdot \JNA_T(\cF)$ for all $b\in V$ and all $T$-equivariant finitely generated $\bN$-filtrations $\cF$ of $R_b$ with $\lambda_{\max}(\cF)\le 1$. By Remark \ref{rem:choice of T} and Lemma \ref{lem:only fg N-filt}, we conclude that $(X_b,\Delta_b)$ is reduced uniformly K-stable with slope at least $\eta$ as long as $b\in V$. 
\end{proof}

\section{$\beta_{\delta}$-invariants for filtrations}\label{s-tbeta}

In this section, we define and study the $\beta_{\delta}$-invariants for filtrations of anti-canonical rings of log Fano pairs. In particular, we will complete the proof of Theorem \ref{t-main1}. We believe for many questions, this gives the appropriate extension of $\beta$-invariants in \cite{Li17, Fuj19} defined for valuations. 

As an immediate consequence, we will also give a more conceptual (in our opinion) proof of the semi-positivity of CM line bundles \cite{CP18}*{Theorem 1.8}. 

We fix the following notation: let $(X,\Delta)$ be a log Fano pair, let $r>0$ be an integer such that $L:=-r(K_X+\Delta)$ is Cartier and let $R=R(X,L)$.

\subsection{Definition of $\beta_{\delta}$-invariants}

\begin{defn} \label{d-beta}
Given a filtration $\cF$ of $R$ and some $\delta\in\bR_+$, we define the \emph{$\delta$-log canonical slope} (or simply \emph{log canonical slope} when $\delta=1$) $\mu_{X,\Delta,\delta}(\cF)$ (or $\mu_{\delta}(\cF)$ if the pair $(X,\Delta)$ is clear from the context) as
\begin{equation}\label{e-twistmu}
\mu_{X,\Delta,\delta}(\cF) = \sup \left\{t\in\bR\,|\,\lct(X,\Delta;I^{(t)}_\bullet)\ge \frac{\delta}{r}\right\}
\end{equation}
where $I^{(t)}_\bullet=I^{(t)}_\bullet(\cF)$ is as in Definition \ref{defn:base ideal} and define 
\[
\beta_{X,\Delta,\delta}(\cF):= \frac{\mu_{X,\Delta,\delta}(\cF)-S(\cF)}{r}.
\]
We will often suppress the subscripts when the pair $(X,\Delta)$ is clear or $\delta=1$. 
\end{defn}


We have the following properties which compares our $\beta$-invariants to the original definition in \cite{Fuj19, Li17}.

\begin{prop} \label{prop:tbeta property}
For any divisor $E$ over $X$, we have $\beta(E)\ge \beta(\cF_{\ord_E})$. Moreover, equality holds when $E$ is weakly special $($i.e. it is induced by a weakly special test configuration with an irreducible central fiber, see e.g. \cite{BLX19}*{Definition A.1} or Lemma \ref{l-weakly}). 
\end{prop}

\begin{proof}
Let $\cF=\cF_{\ord_E}$. By definition we have $\ord_E(I^{(t)}_\bullet(\cF))\ge t$, thus 
$$\lct(X,\Delta;I^{(t)}_\bullet)<\frac{1}{r}\qquad\mbox{ when } t>r\cdot A_{X,\Delta}(E).$$ It follows that $\mu(\cF)\le r\cdot A_{X,\Delta}(E)$ and hence $\beta(\cF)\le \beta(E)$ by definition. If $E$ is weakly special, then by \cite{BLX19}*{Theorem A.2} there exists an effective $\bQ$-divisor $D\sim_\bQ -(K_X+\Delta)$ such that $(X,\Delta+D)$ is lc and $A_{X,\Delta}(E)=\ord_E(D)$. This implies that 
\[
\lct(X,\Delta;I^{ (r\cdot A_{X,\Delta}(E))}_\bullet)\ge \lct(X,\Delta;rD)\ge \frac{1}{r}.
\]
Thus $\mu(\cF)\ge r\cdot A_{X,\Delta}(E)$ and $\beta(\cF) \ge \beta(E)$.
\end{proof}


In Lemma \ref{l-comple}, we will prove a partial converse of Proposition \ref{prop:tbeta property}. 

In the remaining part of this subsection, we show that the non-negativity of $\beta$-invariants (resp. $\beta_{\delta}$-invariants for some $\delta>1$) for filtrations characterizes K-semistability (resp. uniform K-stability). Later in $\S$\ref{ss-tbetaredK}, we shall see that the $\beta_\delta$-invariants can be used to detect reduced uniform K-stability (Definition \ref{defn:red uks}).

\begin{thm} \label{thm:tbeta>=D^NA}
Let $(X,\Delta)$ be a log Fano pair and let $r>0$ be an integer such that $L=-r(K_X+\Delta)$ is Cartier. Then we have 
\[
\beta(\cF)\ge \DNA(\cF)
\]
for any linearly bounded multiplicative filtration $\cF$ of $R=R(X,L)$.
\end{thm}

\begin{proof}
Denote by $\mu:=\mu_{X,\Delta}(\cF)$, $\lambda_{\max}:=\lambda_{\max}(\cF)$ and we have $\mu\le \lambda_{\max}$. If $\mu=\lambda_{\max}$, then it is clear that $\beta(\cF)=\JNA(\cF)\ge \DNA(\cF)$. Hence we may assume that $\lambda_{\max}>\mu$ in what follows. In particular, $I_{m,\lambda}\neq 0$ for some $\lambda>\mu m$.

For each $m$, the ideal $I_{m,\mu m+\epsilon}$ does not depend on the choice of $\epsilon>0$ as long as $\epsilon$ is sufficiently small and we set $\fa_m=I_{m,\mu m+\epsilon}$ where $0<\epsilon\ll 1$. It is easy to see that $\fa_\bullet$ is a graded sequence of ideals on $X$. By the definition of $\mu$, we have 
\[
m\cdot \lct(X,\Delta;\fa_m) = m\cdot \lct(X,\Delta;I_{m,\mu m+\epsilon})\le \lct(X,\Delta;I_\bullet^{(\mu+\epsilon /m)})\le \frac{1}{r}
\]
for all $m$. It follows that $\lct(X,\Delta;\fa_\bullet)\le \frac{1}{r}$ and hence by \cite{JM12}*{Theorem A} there is a valuation $v$ over $X$ such that 
\begin{equation} \label{eq:a<=v}
    a:=A_{X,\Delta}(v)\le \frac{1}{r}v(\fa_\bullet)<\infty.
\end{equation}
For each $\lambda\in \bR$, we set $f(\lambda)=v(I^{(\lambda)}_\bullet)$. Since $\lambda_{\max}>\mu$, there exists some $\epsilon>0$ such that $f(\lambda)<\infty$ for all $\lambda<\mu+\epsilon$. Since the filtration $\cF$ is multiplicative, we know that $f$ is a non-decreasing convex function. It follows that $f$ is continuous on $(-\infty,\mu+\epsilon)$ and from the construction we see that 
$$f(\mu)\le v(\fa_\bullet)\le \lim_{\lambda\to \mu+} f(\lambda)=f(\mu),$$ hence $f(\mu)=v(\fa_\bullet)\ge ar$ by \eqref{eq:a<=v}. We then have 
\begin{eqnarray}\label{e-convex}
f(\lambda)\ge f(\mu)+\xi (\lambda-\mu)\ge ar+\xi(\lambda-\mu) \mbox{\ \ where \ } \xi:=\lim_{h\to 0+}\frac{f(\mu)-f(\mu-h)}{h}
\end{eqnarray} 
for all $\lambda$ by the convexity of $f$. We claim that $\xi>0$. Indeed, it is clear that $\xi\ge 0$ since $f$ is non-decreasing. If $\xi=0$, then $f$ must be constant on $(-\infty,\mu]$; but this is a contradiction since $f(\mu)\ge ar>0$ while we always have $f(e_-)=0$. Hence $\xi>0$ as desired. Replacing $v$ by $\xi^{-1} v$, we may assume that $\xi=1$ and \eqref{e-convex} becomes
\begin{equation} \label{eq:f convex}
    f(\lambda)\ge \lambda +ar-\mu.
\end{equation}

Now let $\tv$ be the valuation on $X\times \bA^1$ given by the quasi-monomial combination of $v$ and $X_0$ with weight $(1,1)$. Using the same notation as in Definition \ref{defn:D^NA}, we have
\begin{align*}
    \tv(I_{m,me_-+i}\cdot t^{me-i}) & \ge m f\left( \frac{me_-+i}{m} \right) +(me-i) \\
     & \ge m\left(\frac{me_-+i}{m}+ar-\mu\right)+(me-i) \\
     & = m(e_++ar-\mu) \quad (\forall i\in\bN)\ ,
\end{align*}
where the first inequality follows from the definition of $f(\lambda)$ and the second inequality follows from \eqref{eq:f convex}. It follows that $\tv(\tI_m)\ge m(e_++ar-\mu)$ and hence by definition of $c_m$ we obtain
\[
c_m \le A_{(X,\Delta)\times \bA^1}(\tv)-\frac{\tv(\tI_m)}{mr}\le a+1-\frac{e_+ +ar-\mu}{r}=\frac{\mu-e_+}{r}+1
\]
for all $m\in\bN$. Thus $c_\infty\le \frac{\mu-e_+}{r}+1$ and we have
\[
\DNA(\cF) = c_\infty + \frac{e_+ - S(\cF)}{r} -1 \le \frac{\mu-S(\cF)}{r} = \beta(\cF)
\]
as desired.
\end{proof}

\begin{cor} \label{cor:beta hat>=0}
A log Fano pair $(X,\Delta)$ is K-semistable if and only if $\beta(\cF)\ge 0$ for any filtration $\cF$ of $R=R(X,-r(K_X+\Delta))$.
\end{cor}

\begin{proof}
Suppose that $(X,\Delta)$ is K-semistable. Then combining \eqref{eq:ding filtration} (see \cite{Fuj18b}*{\S 4.2}) and  \cite[Theorem 6.5]{Fuj19}, we have $\DNA(\cF)\ge 0$ for any linearly bounded filtration $\cF$ of $R$. Hence $\beta(\cF)\ge 0$ by Theorem \ref{thm:tbeta>=D^NA}. Conversely, if $\beta(\cF)\ge 0$ for any linearly bounded filtration $\cF$ of $R$, then by Proposition \ref{prop:tbeta property} we have $\beta(E)\ge 0$ for all divisors $E$ over $X$, thus $(X,\Delta)$ is K-semistable.
\end{proof}

More generally, we have:

\begin{prop} \label{prop:delta twisted beta>=0}
Let $(X,\Delta)$ be a log Fano pair, let $r\in \mathbb{N}$ be an integer such that $L=-r(K_X+\Delta)$ is Cartier and let $R=R(X,L)$. Then
\[
\delta(X,\Delta)=\sup\{\delta>0\,|\,\beta_\delta (\cF)\ge 0\mbox{ for any linearly bounded filtration } \cF \mbox{ of } R\}.
\]
\end{prop}

\begin{proof}
For any $\delta>\delta(X,\Delta)$, there exists some valuation $v$  (with $A_{X,\Delta}(v)<\infty$) such that $A_{X,\Delta}(v)<\delta\cdot S_{X,\Delta}(v)$ and from Definition \ref{d-beta} we see that 
\[
\beta_\delta(\cF_v) = \frac{\mu_{X,\Delta,\delta}(\cF_v)-S(\cF_v)}{r}\le \frac{A_{X,\Delta}(v)}{\delta}-S_{X,\Delta}(v)<0,
\]
where we use the fact that $v(I_{\bullet}^{(t)}(\cF_v))\ge t$ for any $t\in \bR_{\ge 0}$ and therefore if we take $t_0=\frac{A_{X,\Delta}(v)r}{\delta}$, then
$$\lct(X,\Delta; (I_{\bullet}^{(t_0)}(\cF_v))\le \frac{\delta}{r}.$$

Thus it remains to show that for all $0<\delta\le \delta(X,\Delta)$, we have $\beta_\delta (\cF)\ge 0$ for any filtration $\cF$ of $R$. The argument is very close to the proof of Theorem \ref{thm:tbeta>=D^NA}, so we only give a sketch. First we may assume that $\mu:=\mu_{X,\Delta,\delta}(\cF)<\lambda_{\max}:=\lambda_{\max}(\cF)$. Let $\fa_m=I_{m,\mu m+\epsilon}$ where $0<\epsilon\ll 1$. As in the proof of Theorem \ref{thm:tbeta>=D^NA}, we have $\lct(X,\Delta;\fa_\bullet)\le \frac{\delta}{r}$ and thus there exists a valuation $v$ over $X$ such that 
\[
a:=A_{X,\Delta}(v)\le \frac{\delta}{r}v(\fa_\bullet).
\]
For each $\lambda\in\bR$, let $f(\lambda)=v(I^{(\lambda)}_\bullet)$. Then we have $f(\mu)\ge \frac{ar}{\delta}$ by the above inequality. By convexity we also have 
\begin{equation}\label{e-conv2}
f(\lambda)\ge f(\mu)+\xi (\lambda-\mu)\ge \frac{ar}{\delta}+\xi (\lambda-\mu)
\end{equation} 
 for all $\lambda\in\bR$ (where $\xi=f'(\mu)>0$). Replacing $v$ by $\xi^{-1} v$, we may assume that $\xi=1$. After translating $\cF$ by $c=\frac{ar}{\delta}-\mu$ (which does not change the value of $\beta(\cF)$), we may further assume that $\frac{ar}{\delta}=\mu$ and hence $f(\lambda)\ge \lambda$ for all $\lambda$. In other words, $\cF^\lambda R \subseteq \cF_v^\lambda R$ for all $\lambda$ (where $\cF_v$ the filtration associated to the valuation $v$), hence $S(\cF)\le S(\cF_v)$. Since $\mu=\frac{ar}{\delta}$, from the definition of $\delta(X,\Delta)$, we see that
\[
\beta_\delta (\cF) = \frac{\mu-S(\cF)}{r} \ge \frac{a}{\delta} - \frac{S(\cF_v)}{r} \ge \frac{a}{\delta(X,\Delta)} - \frac{S(\cF_v)}{r} \ge 0
\]
as desired.
\end{proof}

\subsection{Semi-positivity of CM line bundle}\label{ss-semiposi}

Before we proceed to investigate more on $\beta$-invariants, let us show the criterion in Corollary \ref{cor:beta hat>=0} can be used to give a direct proof of the semi-positivity of CM line bundle. 

\begin{prop}\label{prop:cm>=beta}
Let $f\colon (X,\Delta)\to C$ be a generic log Fano family over a smooth projective curve $C$, let $t\in C$ be a point such that $(X_t,\Delta_t)$ is a log Fano pair and let $r>0$ be an integer such that $L=-r(K_X+\Delta)$ is Cartier. Then we have
\[
\deg \lambda_{f,\Delta} \ge (n+1) (-K_{X_t}-\Delta_t)^n \cdot  \beta_{X_t,\Delta_t}(\cF_{\HN}).
\]
where $n=\dim X_t$ and $\cF_\HN$ is the HN-filtration on $R:=R(X_t,L_t)$ $($see \S \ref{sec:HN}$)$.
\end{prop}

\begin{proof}
Let $\cR_m:=f_*\cO_X(mL)$ so that $R_m=\cR_m\otimes k(t)$. By definition, it is not hard to see that
\[
S_m(\cF_{\HN})=\frac{1}{m\dim R_m} \deg \cR_m,
\]
hence by Riemann-Roch calculation (see e.g. \cite{CP18}*{Lemma A.2}) we have
\begin{equation} \label{e-CM=S}
    \deg \lambda_{f,\Delta} = - (n+1) (-K_{X_t}-\Delta_t)^n \cdot \frac{S(\cF_\HN )}{r}.
\end{equation}
Thus it suffices to show that $\mu(\cF_\HN )\le 0$. Suppose that this is not the case, i.e. $\mu(\cF_\HN)>0$, then we also have $\mu_\delta(\cF_\HN)>0$ for some $\delta>1$ (c.f. Lemma \ref{lem:twisted mu}). Choose some $\epsilon\in\bQ$ such that $0<2\epsilon< \mu_\delta(\cF_\HN)$, then by the definition of ($\delta$-)log canonical slope, the pair $(X_t,\Delta_t + \frac{1}{rm}I_{m,2\epsilon m})$ is klt for sufficiently divisible $m$. 

On the other hand, recall that $\cF_{\HN}^{2\epsilon m}R_m$ is the stalk of $$\cF_{\HN}^{2\epsilon m}\cR_m=\cF_{\HN}^{\epsilon m}\left( \mathcal{R}_m\otimes \cO_C(-m\epsilon P) \right)$$
at $t\in C$, hence by \cite{CP18}*{Proposition 5.7}, once $\epsilon m\ge 2g$ every element of $\cF_{\HN}^{2\epsilon m}R_m$ can be lifted to a global section of 
$\mathcal{R}_m\otimes \cO_C(-m\epsilon P)$ (where $P\in C$), i.e. an element of $H^0(X,-mr(K_{X/C}+\Delta)-m\epsilon f^*P)$. Let 
$$f\in H^0(X,-mr(K_{X/C}+\Delta)-m\epsilon f^*P)$$ be a lift of a general member of $\cF_\HN^{2\epsilon m}R_m$ and let $D=\frac{1}{rm}(f=0)$. 

By construction we know that $$K_{X/C}+\Delta+D\sim_{\bQ} -\frac{\epsilon}{r}f^*P$$ and $(X_t,\Delta_t+D_t)$ is klt for general $t\in C$. But then the canonical bundle formula \cite{Kol07}*{Theorem 8.5.1} implies that $K_{X/C}+\Delta+D\sim_\bQ f^*Q$ for some pseudo-effective divisor $Q$ on $C$; as $\epsilon>0$, this a contradiction.
\end{proof}

\begin{cor}[{\cite{CP18}*{Theorem 1.8}}]\label{cor:cm>=0} 
Let $f\colon (X,\Delta)\to C$ be a generic log Fano family over a smooth projective curve $C$. Assume that the general fibers are K-semistable, then $\deg \lambda_{f,\Delta} \ge 0$.
\end{cor}

\begin{proof}
This is an immediate consequence of Corollary \ref{cor:beta hat>=0} and Proposition \ref{prop:cm>=beta}.
\end{proof}

\begin{rem}
Unlike \cite{CP18}, our proof does not use the product trick. We note that the question on K-(semi,poly)stability of product is recently settled in \cite{Zhu19}. 
\end{rem}

Using a similar strategy, we can also bound the nef threshold of $-(K_{X/C}+\Delta)$ with respect to the CM line bundle. This will be one of the key ingredients in proving the ampleness of CM line bundle.

\begin{prop} \label{prop:nef threshold}
Notation as in Proposition \ref{prop:cm>=beta}. Assume that $\beta_{X_t,\Delta_t,\delta}(\cF_\HN)\ge 0$ for some $\delta>1$. Then 
$$-(K_{X/C}+\Delta)+\frac{\delta}{(n+1)v(\delta-1)}f^*\lambda_{f,\Delta}$$ is nef, where $v:=(-K_{X_t}-\Delta_t)^n$.
\end{prop}

\begin{proof}
First assume that $\delta\in \bQ$. By our assumption, we have 
\[
\mu_\delta(\cF_\HN)\ge S(\cF_\HN)=-\frac{r\deg \lambda_{f,\Delta}}{(n+1)v}.
\]
Thus for any rational numbers $\lambda >\lambda' > \frac{\deg \lambda_{f,\Delta}}{(n+1)v}$, there exists $m\gg 0$ and some $G \in |\cF^{-mr\lambda'}_\HN R_m|$ such that $(X_t,\Delta_t + \frac{\delta}{mr}G)$ is klt. As before, by \cite{CP18}*{Proposition 5.7} we can lift $G$ to a section of 
\[
\cF^{2g}_\HN \big( \cR_m \otimes \cO_C(\lceil (mr\lambda'+2g)P \rceil) \big)  \subseteq |-mr(K_{X/C}+\Delta)+ mr\lambda f^*P|
\]
and hence we get an effective divisor $D\sim_\bQ -(K_{X/C}+\Delta)+ \lambda f^*P$ such that $(X_t,\Delta_t+\delta D_t)$ is klt. By \cite{Fuj18}*{Theorem 1.11}, this implies that $f_*\cO_X(m(K_{X/C}+\Delta+\delta D))$ is nef for all sufficiently divisible $m\in\bN$ and hence $$f_*\cO_X(m(K_{X/C}+\Delta+\delta D))\otimes \cO_C(2gP)$$ is globally generated by \cite{CP18}*{Proposition 5.7}. As 
$$K_{X/C}+\Delta+\delta D\sim_{C,\bQ} -(\delta-1)(K_{X/C}+\Delta)$$ 
is $f$-ample, it follow that $m(K_{X/C}+\Delta+\delta D)+2gf^*P$ is globally generated for sufficiently divisible $m\in\bN$. Letting $m\to\infty$ we deduce that 
$$K_{X/C}+\Delta+\delta D \sim_\bQ -(\delta-1)(K_{X/C}+\Delta)+\delta\lambda f^*P$$ is nef. As $\lambda > \frac{\deg \lambda_{f,\Delta}}{(n+1)v}$ is arbitrary, we see that $-(K_{X/C}+\Delta)+\frac{\delta}{(n+1)v(\delta-1)}f^*\lambda_{f,\Delta}$ is nef.

In the general case let $\delta'\in \bQ \cap (1,\delta)$. If $\beta_{X_t,\Delta_t,\delta}(\cF_\HN)\ge 0$, then we also have $\beta_{X_t,\Delta_t,\delta'}(\cF_\HN)\ge 0$. The previous case implies that 
$$-(K_{X/C}+\Delta)+\frac{\delta'}{(n+1)v(\delta'-1)}f^*\lambda_{f,\Delta}$$ is nef. Letting $\delta'\to \delta$ we finish the proof.
\end{proof}

\begin{cor}[\cite{CP18}*{Theorem 1.20}] \label{cor:nef threshold uks}
Notation as in Proposition \ref{prop:cm>=beta}. Assume that for a very general geometric point $t\in C$, $(X_t,\Delta_t)$ is uniformly K-stable and let $\delta=\delta(X_t,\Delta_t)$ (by \cite{BL18}*{Theorem B}, this is well defined), $v=(-K_{X_t}-\Delta_t)^n$. Then $-(K_{X/C}+\Delta)+\frac{\delta}{(n+1)v(\delta-1)}f^*\lambda_{f,\Delta}$ is nef.
\end{cor}

\begin{proof}
By Proposition \ref{prop:delta twisted beta>=0}, we have $\beta_\delta(\cF_\HN)\ge 0$, hence the statement follows immediately from Proposition \ref{prop:nef threshold}.
\end{proof}

\subsection{Relation to reduced uniform K-stability}\label{ss-tbetaredK}

We would like to have a similar statement on the nef thresholds as in Corollary \ref{cor:nef threshold uks} when general fibers are only reduced uniformly K-stable. By Proposition \ref{prop:nef threshold}, this would be true if $\beta_\delta(\cF_\HN)\ge 0$ for some $\delta>1$. However, if a general fiber has a non-discrete automorphism group, there are simple examples (e.g. $\bP^1$-bundles $f\colon X=\mathbf{F}_e\to C=\bP^1$ with $e>0$) where $\deg \lambda_f = 0$ while $-K_{X/C}$ is not nef (in particular, the nef threshold does not exists). 
 
It turns out that the right statement is that for a family of reduced uniformly K-stable log Fano pairs, the non-negativity $\beta_\delta(\cF)\ge 0$ (for some $\delta>1$) is true after a torus twist (see Theorem \ref{thm:beta>=0 after twist}). To obtain this result, we need to have a better understanding of the relation between $\beta_{\delta}$, $\DNA$ and $\JNA$ for a filtration $\cF$. More precisely, we want to establish the following technical statement.

\begin{prop} \label{prop:perturbed beta>=0}
Let $\alpha,\eta$ be two positive numbers and $n$ a positive integer. Then there exists some $\delta=\delta(\eta,n,\alpha)>1$ with the following property: 
for any $n$-dimensional log Fano pair $(X,\Delta)$ with $\alpha(X,\Delta)\ge \alpha$ and any linearly bounded multiplicative filtration $\cF$ of $R=R(X,L)$ $($where $L=-r(K_X+\Delta)$ is Cartier$)$ that satisfies $\DNA(\cF)\ge \eta \cdot \JNA(\cF)$, we have $\beta_\delta (\cF) \ge 0$.
\end{prop}

\begin{rem}
When $\cF$ is induced by a special test configuration and $v$ is the valuation induced by the special test configuration, then one can show as in Proposition \ref{prop:tbeta property} that
$$\beta_{\delta}(\cF_v)=\frac{A_{X,\Delta}(v)}{\delta}-\frac{1}{r}S(\cF_v).$$  Thus the claim is easy to see. However, (unless $\delta=1$) we are not able to argue as  in \cite{LX14} by using MMP to reduce the general case to the special test configurations. 
\end{rem}
Therefore, to prove  Proposition \ref{prop:perturbed beta>=0}, we will rely on a detailed study of the Duistermaat-Heckman measure. We first need to show a number of auxiliary results.

\begin{lem} \label{lem:twisted mu}
For any $0<s,\epsilon<1$, we have  
\begin{equation} \label{eq:twisted mu}
    \mu_{1+(1-\epsilon)s}(\cF) \ge s\cdot \mu_{\epsilon^{-1}}(\cF)+(1-s)\mu(\cF).
\end{equation}
\end{lem}

\begin{proof}
Let $\mu=\mu(\cF)$, $\mu_0=\mu_{\epsilon^{-1}}(\cF)$ and $\mu'=\mu_{1+(1-\epsilon)s}(\cF)$. We may assume that $\mu'<\lambda_{\max}(\cF)$ and $\mu'<\mu$, otherwise the statement is clear. Similar to the proof of Theorem \ref{thm:tbeta>=D^NA} we let $\fa'_m = \cI_{m,\mu' m+\epsilon'}$ where $0<\epsilon'\ll 1$. Then $\lct(\fa'_\bullet)\le \frac{1+(1-\epsilon)s}{r}$ and there exists a valuation $v$ over $X$ such that
\[
a:=A_{X,\Delta}(v)\le \frac{1+(1-\epsilon)s}{r}\cdot v(\fa'_\bullet)<\infty.
\]
For each $\lambda\in \bR$, we set $f(\lambda)=v(I^{(\lambda)}_\bullet)$. Hence we have
\begin{equation} \label{eq:f(mu')}
    f(\mu')=v(\fa'_\bullet)\ge \frac{ar}{1+(1-\epsilon)s}.
\end{equation}
On the other hand, by the definition of $\mu_\delta(\cF)$, we have 
$$f(\mu_0-\eta) \le \epsilon \cdot A_{X,\Delta}(v)\cdot r=\epsilon a r \mbox{ for any }\eta>0$$ and similarly $f(\mu-\eta) \le ar$. By the convexity and continuity of $f$ on $(-\infty,\lambda_{\max})$ we see that
\[
f(s\mu_0+(1-s)\mu)\le ar(\epsilon s+1-s)<\frac{ar}{1+(1-\epsilon)s}.
\]
Combined with \eqref{eq:f(mu')}, we get $f(\mu')>f(s\mu_0+(1-s)\mu)$ and hence $\mu'>s\mu_0+(1-s)\mu$ as $f$ is non-decreasing.
\end{proof}

\begin{lem} \label{lem:convex integral}
Let $0< \lambda <T$ and let $f(x)$ be a non-negative concave function on $(0,T)$. Then $$\int_0^\lambda f(x)^n \rd x \ge \left(\frac{\lambda}{T} \right)^{n+1} \int_0^T f(x)^n \rd x.$$
\end{lem}

\begin{proof}
We may assume that $f(\lambda)=1$. By assumption, we have $f(x)\ge \frac{x}{\lambda}$ when $x\le \lambda$ and $f(x)\le \frac{x}{\lambda}$ when $x\ge \lambda$. Hence
$$\int_0^\lambda f(x)^n \rd x \ge \frac{\lambda}{n+1}\mbox{\ \ and\ \  }\int_\lambda^T f(x)^n \rd x \le \frac{\lambda}{n+1}\left(\left(\frac{T}{\lambda}\right)^{n+1}-1\right)$$ by direct calculation and the lemma follows.
\end{proof}

\begin{lem} \label{lem:volume upper bound}
Let $(X,\Delta)$ be a log Fano pair of dimension $n$ and let $L=-(K_X+\Delta)$. Then we have
\[
\frac{\vol(L- \lambda E)}{\vol(L)}\le 1- \left(\frac{\lambda \alpha(X,\Delta)}{A_{X,\Delta}(E)} \right)^n
\]
for any divisor $E$ over $X$ and any $0\le \lambda\le \frac{A_{X,\Delta}(E)}{\alpha(X,\Delta)}$.
\end{lem}

\begin{proof}
Let $T=T_{X,\Delta}(E)$ (see \eqref{e-tinvariant}). Let $\pi:Y\to X$ be a log resolution such that $E\subseteq Y$. Let $f(x)=\vol_{Y|E}(\pi^*L-xE)$ where $\vol_{Y|E}(\cdot)$ denotes the restricted volume and $0\le x\le T$. By combining \cite[Theorem A]{BFJ09} and \cite[Theorem 5.2]{ELMNP09}, the function $f(x)^{\frac{1}{n-1}}$ is concave on $(0,T)$ and we have 
$$\vol(L)=\int_0^T f(x)\rd x\mbox{\ \ and \ \ }\vol(L)-\vol(\pi^*L-\lambda E)=\int_0^\lambda f(x)\rd x.$$ Thus we have
\[
\frac{\vol(L)-\vol(L- \lambda E)}{\vol(L)}\ge \left( \frac{\lambda}{T}\right)^n \ge \left(\frac{\lambda \alpha(X,\Delta)}{A_{X,\Delta}(E)} \right)^n
\]
where the first inequality follows from Lemma \ref{lem:convex integral} and the second inequality follows from the definition of alpha invariants. This proves the lemma.
\end{proof}

\begin{lem} \label{lem:volume lower bound}
Let $\nu$ be a probability measure on $\bR$ with compact support such that $\int_\bR \lambda \rd \nu =0$. Assume that $g(\lambda)=\nu\{x\ge \lambda\}^{1/n}$ is concave on $(-\infty,\lambda_{\max})$ where $\lambda_{\max}=\max \rm{supp}\,\nu$. Then 
$$g(-t\lambda_{\max})\ge 1-\frac{1}{\sqrt{nt}}\mbox{\ \ \ for all }t>0.$$
\end{lem}

\begin{proof}
The idea is similar to the proof of \cite{BHJ17}*{Lemma 7.10}. After rescaling, we may assume for simplicity that $\lambda_{\max}=1$. Since $\nu$ is the distributional derivative of $-g(\lambda)^n$, we have
\[
\int_0^1 g(\lambda)^n \rd \lambda = \int_0^1 \lambda \rd \nu = -\int_{-\infty}^0 \lambda \rd \nu = \int_{-\infty}^0 (1-g(\lambda)^n) \rd \lambda,
\]
where the first and third equalities follow from Fubini's theorem, and the second equality follows from the assumption that $ \int_{-\infty}^1 \lambda \rd \nu=0$.

Let $a=-g'_+(-t)\ge 0$ and $b=g(-t)\in [0,1]$. Since $g$ is concave on $(-\infty,1)$, we have 
$$g(\lambda)\le -a(\lambda+t)+b\mbox{\ \  on }(-\infty,1).$$ If $a=0$, then letting $\lambda\to -\infty$ we see that $b= 1$ and there is nothing left to prove. Therefore, we may and do assume $a>0$. Let $\lambda_0$ be such that $-a(\lambda_0+t)+b=1$. Then we have
\begin{align*}
    \int_0^1 (-a(\lambda+t)+b)^n \rd \lambda & \ge \int_0^1 g(\lambda)^n \rd \lambda \\
    &  = \int_{-\infty}^0 (1-g(\lambda)^n) \rd \lambda \ge \int_{\lambda_0}^0 (1-(-a(\lambda+t)+b)^n) \rd \lambda.
\end{align*}
Computing the integrals, we deduce that
\[
\frac{1-(b-at-a)^{n+1}}{a(n+1)}\ge -\lambda_0 = \frac{1-(b-at)}{a},
\]
hence $(n+1)u\ge n+(u-a)^{n+1}$ where $u=b-at$. Note that 
$$u-a=b-a(t+1)\ge g(1)\ge 0,$$ thus $u\ge \frac{n}{n+1}$. As $u+at=b=g(-t)\le 1$, we see that $u\le 1$ and $a\le \frac{1}{(n+1)t}$. We then have
\[
(n+1)u\ge n + (u-a)^{n+1} \ge n + u^{n+1} - (n+1)a u^n \ge n + u^{n+1} - \frac{u^n}{t}.
\]
It follows that 
\[
\frac{1}{t}\ge \frac{u^n}{t}\ge n+u^{n+1}-(n+1)u = (1-u)^2\sum_{i=1}^n \frac{1-u^i}{1-u}\ge n(1-u)^2.
\]
Therefore, $f(-t)=b\ge u\ge 1-\frac{1}{\sqrt{nt}}$ as desired.
\end{proof}

\begin{proof}[Proof of Proposition \ref{prop:perturbed beta>=0}]
After translating $\cF$ by $-S(\cF)$ (which does not change $\DNA(\cF)$, $\JNA(\cF)$ and $\beta_\delta(\cF)$), we may and do assume that $S(\cF)=0$. Let $\lambda_{\max} = \lambda_{\max}(\cF)$.  By Lemma \ref{l-concavedistri}, we can apply Lemma \ref{lem:volume lower bound} to the Duistermaat-Heckman measure of $\cF$ (see \S \ref{ss-filt}). So we have (recall that $L=-r(K_X+\Delta)$)
\begin{equation} \label{eq:vol(FR) lower bound}
    \frac{\vol(\cF R^{(-t\lambda_{\max})})}{(L^n)}\ge \left(1-\frac{1}{\sqrt{nt}}\right)^n >  1-\sqrt{\frac{n}{t}}
\end{equation}
for all $t>0$. Let $E$ be a divisor over $X$. Then we claim there is an inequality
\begin{equation} \label{eq:ord_E(I)}
    \frac{\ord_E(I_\bullet^{(-t\lambda_{\max})})}{r}< \frac{A_{X,\Delta}(E)}{\alpha(X,\Delta)}\sqrt[2n]{\frac{n}{t}}.
\end{equation}
Otherwise we have $\cF^{-mt\lambda_{\max}} R_m\subseteq \cF_E^{mr\lambda_0} R_m$ (where $\lambda_0$ is the right hand side of \eqref{eq:ord_E(I)}) for all $m\in\bN$ and thus by Lemma \ref{lem:volume upper bound},
\[
\frac{\vol(\cF R^{(-t\lambda_{\max})})}{(L^n)}\le \frac{\vol(-K_X-\Delta-\lambda_0 E)}{(-K_X-\Delta)^n}\le 1-\sqrt{\frac{n}{t}},
\]
contradicting \eqref{eq:vol(FR) lower bound}. 

After rescaling $\cF$ (which will not change our conclusion), we can assume $\lambda_{\max}=1$. Since $E$ is arbitrary, we infer from \eqref{eq:ord_E(I)} that
\[
\lct(X,\Delta;I^{(-t)}_\bullet)>\frac{\alpha(X,\Delta)}{r}\sqrt[2n]{\frac{t}{n}}\ge \frac{\alpha}{r}\sqrt[2n]{\frac{t}{n}}.
\]
Now choose $t=t_0:=n\left(\frac{2}{\alpha}\right)^{2n}$, then the above lct estimate becomes $\lct(X,\Delta;I^{(-t_0)}_\bullet)>\frac{2}{r}$ 
and thus $\mu_2(\cF)\ge -t_0$. By the assumption and Theorem \ref{thm:tbeta>=D^NA}, we also have 
$$\beta(\cF) \ge \DNA(\cF)\ge \eta \cdot \JNA(\cF)=\frac{\eta}{r},$$ hence $\mu(\cF)\ge \eta$ as $S(\cF)=0$. If we choose $\delta=1+\frac{\eta}{2(t_0+\eta)}$ (which only depends on $\eta,\alpha$ and $n$), $\epsilon=\frac{1}{2}$ and  $s=\frac{\eta}{t_0+\eta}$, then it follows from Lemma \ref{lem:twisted mu} that 
$$\beta_{\delta}(\cF)=\mu_{\delta}(\cF)\ge s\mu_2(\cF)+(1-s)\mu(\cF)\ge 0.$$
\end{proof}

\begin{cor} \label{cor:red uKs criterion}
Let $(X,\Delta)$ be a log Fano and $T\subseteq \Aut(X,\Delta)$ a maximal torus. Then the following are equivalent:
\begin{enumerate}
    \item $(X,\Delta)$ is reduced uniformly K-stable, 
    \item there exists some constant $\eta>0$ such that for any $T$-equivariant filtration $\cF$ on $R$, there exists some $\xi\in N_\bR$ such that $\DNA(\cF)\ge \eta\cdot \JNA(\cF_\xi)$,
    \item $(X,\Delta)$ is K-semistable, and there exists some constant $\delta>1$ such that for any $T$-equivariant filtration $\cF$ on $R$, there exists some $\xi\in N_\bR$ such that $\beta_\delta(\cF_\xi)\ge 0$.
\end{enumerate}
\end{cor}

\begin{proof}
By Lemma \ref{lem:DNA twist & JNA min} and Proposition \ref{prop:D>=J any filtration}, we have (1) implies (2). Now we assume (2), in particular $(X,\Delta)$ is K-semistable. By Proposition \ref{prop:perturbed beta>=0}, we then have (3) since $\DNA(\cF_{\xi})=\DNA(\cF)$ as $\Fut(\xi)=0$. It remains to show that (3) implies (1). 

Let $\delta>1$ be the constant for which (3) holds. Let $v$ be a $T$-invariant valuation and let $\cF=\cF_v$ be its induced filtration on $R$. Then there exists some $\xi\in N_\bR$ such that $\beta_\delta(\cF_{v_\xi})=\beta_\delta(\cF_\xi)\ge 0$ (the first equality holds since $\cF_{v_\xi}$ and $\cF_\xi$ only differ by a translation, see \cite{Li19}*{Proposition 3.8}). But it is clear from the definition that $\mu_\delta(\cF_{v_\xi})\le \frac{A_{X,\Delta}(v_\xi)}{\delta}$, hence we obtain $A_{X,\Delta}(v_{\xi})\ge \delta \cdot S(v_{\xi})$. By Theorem \ref{thm:Li's valuative T-uKs}, this implies (1). 
\end{proof}

Now we can complete the proof of Theorem \ref{t-main1}. 

\begin{proof}[Proof of Theorem \ref{t-main1}]
Theorem \ref{t-main1}(1) (resp. (3)) follows from Corollary \ref{cor:beta hat>=0} (resp. Corollary \ref{cor:red uKs criterion}). One direction of Theorem \ref{t-main1}(2) follows from Proposition \ref{prop:delta twisted beta>=0}, as if $(X,\Delta)$ is uniformly K-stable, then $\delta(X,\Delta)>1$. The converse can be derived from exactly the same argument as in the second paragraph of Corollary \ref{cor:red uKs criterion}, without taking a twist by $\xi$. 
\end{proof}

The following theorem, which will be needed later, is also an easy consequence of Proposition \ref{prop:perturbed beta>=0}.

\begin{thm} \label{thm:beta>=0 after twist}
Let $\alpha,\eta>0$, let $(X,\Delta)$ be a log Fano pair and let $T\subseteq \Aut(X,\Delta)$ be a maximal torus. Assume that $(X,\Delta)$ is reduced uniformly K-stable with slope at least $\eta$ and $\alpha(X,\Delta)\ge \alpha$, then there exists a constant $\delta>1$ depending only on $\eta$, $n=\dim X$ and $\alpha$ such that for any filtration $\cF$ on $R=R(X,-r(K_X+\Delta))$, there exists some $\xi\in N_\bR$ such that $\beta_\delta(\cF_\xi)\ge 0$.
\end{thm}

\begin{proof}
Let $\delta=\delta(\eta,n,\alpha)>1$ be the constant given by Proposition \ref{prop:perturbed beta>=0}. By definition, we have $\DNA(\cF)\ge \eta\cdot \JNA_T(\cF)$ and  $(X,\Delta)$ is K-semistable. By Lemma \ref{lem:DNA twist & JNA min}, this implies that there exists $\xi\in N_\bR$ such that $\DNA(\cF_\xi) \ge \eta\cdot \JNA(\cF_\xi)$. Then $\beta_\delta(\cF_\xi)\ge 0$ by Proposition \ref{prop:perturbed beta>=0} and our choice of $\delta>1$.
\end{proof}

\section{Twisted families} \label{sec:twisted family}

We will eventually apply Theorem \ref{thm:beta>=0 after twist} to the Harder-Narasimhan filtrations induced by generic log Fano families over curves (see Definition \ref{d-genericlog}). To get nef thresholds through Proposition \ref{prop:nef threshold}, we need to construct a twisted family whose HN-filtration is the twist of the original HN-filtration. In this section, we show that this can be done after a suitable modification.

\begin{defn}[Twisted family] \label{defn:twisted family}
Let $T$ be a torus, let $f\colon X\to S$ be a projective flat morphism with a fiberwise $T$-action and let $L$ be a $T$-linearized $f$-ample line bundle on $X$. We have the weight decomposition
$$\cR_m:=f_*\cO_X(mL)=\bigoplus_{\alpha \in M} \cR_{m,\alpha}\mbox{\ \ where }M=\Hom(T,\bG_m).$$ 
Let $A$ be a Cartier divisor on $S$ and let $\xi\in N=M^*=\Hom(\bG_m,T)$. Then the {\it $\xi$-twist of $f\colon (X,L)\to S$} along $A$ is defined to be
\begin{equation} \label{eq:f_xi}
    f_\xi \colon \big(X_{\xi}=\Proj_S\bigoplus_{m\in \bN}\bigoplus_{\alpha\in M}\cR_{m,\alpha}\otimes \cO_S(\la \alpha, \xi \ra \cdot A), L_{\xi}=\cO_{X_\xi}(1)\big)\to S.
\end{equation}
Note that Zariski locally over $S$, $(X_\xi,L_\xi)$ is isomorphic to $(X,L)$. If $Z\subseteq X$ is a $T$-invariant closed subscheme, then $Z_\xi$ is naturally a closed subscheme of $X_\xi$. In particular, if $f\colon (X,\Delta)\to S$ is a generic log Fano family with a fiberwise $T$-action, then $T$ naturally acts on $L=-r(K_{X/S}+\Delta)$ for some sufficiently divisible $r\in\bN_+$ and we define the $\xi$-twist $(X_\xi,\Delta_\xi)$ of $(X,\Delta)$ as the $\xi$-twist with respect to $L$. 
\end{defn}

\begin{expl}
Consider the trivial $\bP^1$-bundle $f\colon X=\bP^1\times \bP^1\to \bP^1$ with the canonical fiberwise $\bG_m$-action and let $\xi\in N\cong \bZ$ be a generator. Then the $\xi$-twist of $X$ along a divisor of degree $e>0$ on the base $\bP^1$ is isomorphic to the ruled surface $\mathbb{F}_e$. Therefore the above construction of twisted family can be viewed as a generalization of elementary transformations on ruled surfaces.
\end{expl}

\begin{lem} \label{lem:L=-K-Delta}
Let $f\colon (X,\Delta)\to S$ be a generic log Fano family with a fiberwise $T$-action and let $\xi\in N$. Then in the notation of Definition \ref{defn:twisted family}, we have $L_\xi\sim -r(K_{X_\xi/S}+\Delta_\xi)$.
\end{lem}

In particular, from \eqref{eq:f_xi} we see that for families over curves, the HN-filtration of the $\xi$-twist coincides with the $\xi$-twist of the HN-filtration of the original families.

\begin{proof}
By choosing local trivialization $\cO_U(A)\cong \cO_U$ (where $U\subseteq S$ is open), we get isomorphisms $\cR_{m,\alpha}\otimes \cO_S(\la \alpha, \xi \ra \cdot A) \cong \cR_{m,\alpha}$ for all $m,\alpha$ and hence an identification of $(X_\xi,\Delta_\xi,L_\xi)$ with $(X,\Delta,L)$ over $U$ and also an isomorphism $\cO_V(L_\xi)\cong \cO_V(-r(K_{X_\xi/S}+\Delta_\xi))$ since $L=-r(K_{X/S}+\Delta)$ (where $V=f^{-1}(U)$). Different trivializations $\cO_U(A)\cong \cO_U$ differ by a unit $u\in\cO_U^*$. It can be lifted to an automorphism of $(X,\Delta)$ over $U$ through the composition $ U\stackrel{u}{\to} \bG_m \stackrel{\xi}{\to} T$. The $T$-action on $\cO_V(1)$ and on $\cO_V(-r(K_{X/S}+\Delta))$ coincides, hence the action of $u$ commutes with the isomorphism $\cO_V(L_\xi)\cong \cO_V(-r(K_{X_\xi/S}+\Delta_\xi))$. Thus these isomorphisms glue to give an isomorphism $\cO_{X_{\xi}}(L_\xi)\cong \cO_{X_\xi}(-r(K_{X_\xi/S}+\Delta_\xi))$.
\end{proof}

\begin{cor} \label{cor:CM deg after twist}
Let $T$ be a torus and let $f\colon (X,\Delta)\to S$ be a generic log Fano family with a fiberwise $T$-action. Assume that the general fibers $(X_t,\Delta_t)$ are K-semistable. Then for any $\xi\in N$ and any Cartier divisor $A$ on $S$ we have $\lambda_{f,\Delta} \sim_\bQ \lambda_{f_\xi,\Delta_\xi}$, where $f_{\xi}\colon (X_{\xi},\Delta_{\xi})\to S$ is the $\xi$-twist of $f$ along $A$.
\end{cor}

\begin{proof}
By the definition of CM line bundle and Lemma \ref{lem:L=-K-Delta}, $-r^{n+1}\lambda_{f,\Delta}$ (resp. $-r^{n+1}\lambda_{f_\xi,\Delta_\xi}$) is the leading term of the Knudsen-Mumford expansion of $L$ (resp. $L_\xi$), e.g. $c_1(f_*\cO_X(mL))=-\frac{(mr)^{n+1}}{(n+1)!}\lambda_{f,\Delta}+O(m^n)$. By the construction of twisted family \eqref{eq:f_xi} we then have
\[
\lambda_{f_\xi,\Delta_\xi} \sim_\bQ \lambda_{f,\Delta} + (n+1)(-K_{X_t}-\Delta_t)^n\cdot \Fut_{X_t,\Delta_t}(\xi)\cdot A.
\]
Since $(X_t,\Delta_t)$ is K-semistable, we have $\Fut_{X_t,\Delta_t}(\xi) = 0$, hence the result follows.
\end{proof}

So far we realize the $\xi$-twists of an HN-filtration as HN-filtrations of twisted families for all $\xi\in N$. By passing to finite covers, one can also construct families that realize $\xi$-twists when $\xi\in N_\bQ$. However, the twisted family seems unlikely to exist if $\xi$ is not a rational vector. Fortunately, as we will show in the remaining part of this section, for HN-filtrations the twist vectors $\xi$ in Theorem \ref{thm:beta>=0 after twist} can be chosen to be rational. 



\begin{lem} \label{lem:S and lambda_min in Q}
Let $T$ be a torus and let $f\colon (X,\Delta)\to C$ be a generic log Fano family over a smooth curve with a fiberwise $T$-action. Let $\cF$ be the induced Harder-Narasimhan filtration on $R=R(X_t,-r(K_{X_t}+\Delta_t))$ where $t\in C$ is a general closed point. Then $S(\cF)\in\bQ$ and for any $\xi\in N_\bQ$, we have $\lambda_{\min}(\cF_\xi)\in \bQ$.
\end{lem}

\begin{proof}
By assumption, the filtration $\cF$ is $T$-equivariant. By \eqref{e-CM=S}, $S(\cF)$ is a rational multiple of $\deg \lambda_f$, hence is rational. Let $d$ be an integer such that $d\xi\in N_{\bZ}$. Let $C'\to C$ be a finite morphism of degree $d$ and let $f'\colon (X',\Delta')\to C'$ be the base change of $f$. Let $P\in C'$ be a smooth point and consider the $(d\xi)$-twist $g\colon (X'_\xi,\Delta'_\xi)\to C'$ of $f'$ along $P$. Let $\cG$ be the Harder-Narasimhan filtration on $R$ induced by $g$. 

Since the pull back of a semistable vector bundle with slope $\mu_C$ under $C'\to C$ is still semi-stable with slope $\deg(C'/C)\cdot \mu_C$, by Lemma \ref{lem:L=-K-Delta} we can check that $\cG^\lambda R_m = \cF^{\lambda/d}_\xi R_m$ for all $\lambda,m$, hence $\lambda_{\min}(\cG)=d\cdot \lambda_{\min}(\cF_\xi)$. By Lemma \ref{l-nefinQ} and Proposition \ref{p-supp} we have $\lambda_{\min}(\cG)\in\bQ$, thus $\lambda_{\min}(\cF_\xi)\in \bQ$ as well.
\end{proof}

\begin{prop} \label{prop:tbeta>=0 rat'l twist}
Notation as in Lemma \ref{lem:S and lambda_min in Q} and let $\alpha,\eta>0$. Assume that very general fibers $(X_t,\Delta_t)$ are reduced uniformly K-stable with slope at least $\eta>0$, $T\subseteq \Aut(X_t,\Delta_t)$ is a maximal torus and $\alpha(X_t,\Delta_t)\ge \alpha$. Then there exists some constant $\delta=\delta(\eta,n,\alpha)>1$ such that $\beta_\delta(\cF_\xi)\ge 0$ for some $\xi\in N_\bQ$.
\end{prop}

\begin{proof}
By assumption and Corollary \ref{cor:red uKs criterion}, there exists some $\xi_0\in N_\bR$ such that $\DNA(\cF_{\xi_0})\ge \eta\cdot \JNA(\cF_{\xi_0})$. We claim that
\begin{equation} \label{eq:D>=J rat'l twist}
    \DNA(\cF_\xi)\ge \frac{\eta}{2}\cdot \JNA(\cF_\xi)
\end{equation}
for some $\xi\in N_\bQ$. Indeed when $\JNA(\cF_{\xi_0})>0$ this follows from $\DNA(\cF_\xi)=\DNA(\cF)$ and the continuity of $\JNA(\cF_\xi)$ with respect to $\xi$ (Lemma \ref{lem:DNA twist & JNA min}), so it suffices to consider the case when $\JNA(\cF_{\xi_0})=0$, in other words, 
$$\lambda_{\max}(\cF_{\xi_0})=\lambda_{\min}(\cF_{\xi_0})=S(\cF_{\xi_0})=:\lambda_0.$$ As $(X_t,\Delta_t)$ is K-semistable, we have $S(\cF_{\xi_0})=S(\cF)$ by Lemma \ref{lem:DNA twist & JNA min}, thus $\lambda_0\in \bQ$ by Lemma \ref{lem:S and lambda_min in Q}. Let $\lambda_{m,\alpha}=\sup\{\lambda\,|\,(\cF_{\xi_0}^\lambda R_m)_\alpha = R_{m,\alpha}\}$ and let $\Gamma_m = \{\alpha\in M\,|\,R_{m,\alpha}\neq 0\}$. Let $P\subseteq M_\bR$ be the convex hull of $\cup_m \frac{1}{m}\Gamma_m$. As $\lambda_{\max}(\cF_{\xi_0})=\lambda_{\min}(\cF_{\xi_0})=\lambda_0$, we also have
$$\lim_{m\to \infty}\frac{\sup_{\alpha\in \Gamma_m} \lambda_{m,\alpha}}{m} = \lim_{m\to \infty}\frac{\inf_{\alpha\in \Gamma_m} \lambda_{m,\alpha}}{m}=\lambda_0.$$
It follows that for any fixed $\xi\in N_\bQ$, we have 
$$\lambda_{\min}(\cF_\xi) = \lim_{m\to\infty} \frac{\inf_{\alpha\in\Gamma_m} (\lambda_{m,\alpha}+\la \alpha, \xi-\xi_0 \ra)}{m} = \lambda_0 + \inf_{\alpha\in P} \la \alpha, \xi-\xi_0 \ra.$$
Since $\lambda_{\min}(\cF_\xi)\in \bQ$ by Lemma \ref{lem:S and lambda_min in Q}, we deduce that $\inf_{\alpha\in P} \la \alpha, \xi-\xi_0 \ra\in \bQ$ for all $\xi\in N_\bQ$. As $R=\bigoplus_{m,\alpha} R_{m,\alpha}$ is finitely generated, $P$ is a rational polytope. This implies that $\xi_0\in N_\bQ$ and in particular \eqref{eq:D>=J rat'l twist} holds with $\xi=\xi_0\in N_\bQ$. The lemma now follows directly from \eqref{eq:D>=J rat'l twist} and Proposition \ref{prop:perturbed beta>=0}.
\end{proof}

\begin{cor} \label{cor:integral twist}
Notation and assumptions as in Proposition \ref{prop:tbeta>=0 rat'l twist}. Then there exists a constant $\delta=\delta(\eta,n,\alpha)>1$ such that for any finite cover $C'\to C$ of sufficiently divisible degree, we can find $\xi\in N$ which satisfies that $\beta_\delta(\cG)\ge 0$ where $\cG$ is the HN-filtration induced by the $\xi$-twist $g\colon (X'_\xi,\Delta'_\xi)\to C'$ of $(X,\Delta)\times_C C'$ along a smooth point $P\in C'$.
\end{cor}

\begin{proof}
By Proposition \ref{prop:tbeta>=0 rat'l twist}, there exists $\delta=\delta(\eta,n,\alpha)>1$ and $\xi_0\in N_\bQ$ such that $\beta_\delta(\cF_{\xi_0})\ge 0$. Let $C'\to C$ be a finite cover of degree $d$ with $d\xi_0\in N$. Let $\xi=d\xi_0$ and let $g$ be the $\xi$-twist of $(X,\Delta)\times_C C'$ along a smooth point $P\in C'$. Then as in the proof of Lemma \ref{lem:S and lambda_min in Q}, we have $\cG^\lambda R_m = \cF^{\lambda/d}_{\xi_0} R_m$ for all $\lambda,m$, hence $\beta_\delta(\cG)=d\cdot \beta_\delta(\cF_{\xi_0})\ge 0$ as desired.
\end{proof}

\section{Ampleness lemma}\label{s-ampleness}

Our next ingredient is an enhanced version of the ampleness lemma (see \cite{Kol90}*{\S 3} and \cite{KP17}*{\S 5}) that gives a simultaneous treatment of all twisted families. We first introduce some definitions and notation that are necessary for the statement.

\begin{defn} \label{defn:condition (*)_d}
Let $f\colon (X,\Delta;L)\to S$ be a polarized family of normal pairs, i.e. $S$ is a normal variety, $f\colon X\to S$ is a flat projective morphism with normal fibers, $\Delta$ is a Weil $\bQ$-divisor on $X$ whose support does not contain any fiber of $f$ in its support and $L$ is an $f$-ample line bundle. Let $D=\Supp(\Delta)$ and let $d\in\bN_+$. We say that $f\colon (X,\Delta;L)\to S$ satisfies condition $(*)_d$ if the followings hold:
\begin{enumerate}
    \item $L$ is $f$-very ample,
    \item $H^j(X_s,mL_s)=H^j(D_s,mL_s)=0$ for all $s\in S$ and all $j,m\in\bN_+$,
    \item for every $s\in S$ the embeddings of $X_s$ and $D_s$ via $L_s$  are cut out (set theoretically) by degree $\le d$ equations, and
    \item both natural maps $\Sym^d H^0(X_s,L_s)\to H^0(X_s,dL_s)\to H^0(D_s,dL_s)$ are surjective.
\end{enumerate}
\end{defn}

\begin{rem}
By assumption, every component of $\Delta$ dominates $S$, hence $f|_D$ is flat over all codimension one point of $S$ \cite{Har77}*{Proposition III.9.7} and in particular there exists a big open set $S^\circ\subseteq S$ such that $f|_D$ is flat over $S^\circ$. Let $f^\circ\colon D^\circ=(f|_D)^{-1}(S^\circ)\to S^\circ$. Condition $(*)_d$ then implies that the formation of $f_*\cO_X(mL)$ and $f^\circ_*\cO_{D^\circ}(mL|_{D^\circ})$ commutes with base change for all $m\in\bN$ (see \cite{Har77}*{Theorem III.12.11}). Note also that if $f\colon (X,\Delta;L)\to S$ admits a fiberwise $T$-action (where $T$ is a torus) and satisfies condition $(*)_d$ for some $d\in\bN_+$, then so does its $\xi$-twists (for any $\xi\in N=\Hom(\bG_m,T)$ and along any Cartier divisor $A$ on $S$).
\end{rem}

\begin{defn} \label{defn:bir pullback family}
Let $\nu\colon S'\dashrightarrow S$ be a dominant rational map between quasi-projective normal varieties and let $f\colon (X,\Delta;L)\to S$ (resp. $f'\colon (X',\Delta';L')\to S'$) be a polarized family of normal pairs. We say that $f'$ is a birational pullback of $f$ if there exists an open subset $U\subseteq S'$ where $\nu$ is defined and a diagram
\begin{equation} \label{eq:diagram bir family}
    \xymatrix{
    (X',\Delta';L')\ar[d]_{f'} & (X_U,\Delta_U;L_U)\ar@{_{(}->}[l]\ar[d]^{f_U}\ar[r] & (X,\Delta;L)\ar[d]^f \\
    S' & U\ar@{_{(}->}[l]_{i_U} \ar[r]^{\nu|_{U}} & S
    }
\end{equation}
where both squares are Cartesian. 
\end{defn}

\begin{rem}
Similarly, we can define birational pullbacks between generic log Fano families $f\colon (X,\Delta)\to S$ using $L=-r(K_{X/S}+\Delta)$ for some sufficiently divisible $r\in\bN_+$. It is easy to see that if $f'$ (resp. $f''$) is a birational pullback of $f$ (resp. $f'$), then $f''$ is also a birational pullback of $f$. Moreover if $f$ admits a fiberwise $T$-action ($T$ being a torus), then any $\xi$-twist of $f$ (where $\xi\in N$) is a birational pullback of $f$ (over the same base $S$).
\end{rem}

\begin{notation} \label{notation for ampleness lem}
We keep the notation of Definition \ref{defn:bir pullback family}. Let $d\in\bN_+$ and consider a diagram as in \eqref{eq:diagram bir family} where both $f$ and $f'$ satisfy condition $(*)_d$. Let $\oS\supseteq S$ be a compactification and let $H$ be a line bundle on $\oS$. Let $\nu^*H$ be the rational pullback of $H$ (i.e. $\nu^*H=p'_*p^*H$ where $p'\colon \widetilde{S}\to S'$ is a log resolution that resolves the indeterminacy locus of $S'\dasharrow \oS$ and $p=\nu\circ p'\colon \widetilde{S}\to \oS$). Let 
$$D=\Supp(\Delta)\mbox{ and }W=f_*\cO_X(L), Q=f_*\cO_X(dL)\oplus f_*\cO_{D}(dL|_D);$$ 
similarly we define $D'$, $W'$ and $Q'$ with $f'$ in place of $f$. Let $w$ and $q$ be the ranks of $W$ and $Q$, respectively. Note that we have a natural surjective map $(\Sym^d W)^{\oplus 2}\to Q$ (similarly with $W'$, $Q'$ in place of $W$, $Q$) by condition $(*)_d$.
\end{notation}

\begin{thm} \label{thm:ampleness lemma}
In the situation of Notation \ref{notation for ampleness lem}, assume that the family $f\colon (X,\Delta)\to S$ has maximal variation. Then there exists some $m\in\bN_+$ depending only on $d$, $H$ and the family $f$ such that there is a non-zero map
\[
\Sym^{dqm} \Big( W'^{\oplus 4w} \Big) \to \cO_{S'}(-\nu^* H)\otimes \det(Q')^{\otimes m}
\]
for any birational pull back family as in Definition \ref{defn:bir pullback family}.
\end{thm}

\begin{proof}
We follow the argument of \cite{Kol90}*{\S 3} and \cite{KP17}*{\S 5}. First observe that if $(X,\Delta)\to S$ has maximal variation, then so does $(X,D)$ by Lemma \ref{lem:max var to reduced bdy}, hence we may assume $\Delta=D$. Replacing $S'$ and $S$ by some big open subset, We may also assume that $f|_D$ and $f'|_{D'}$ are flat (since $W'$ is locally free, the obtained non-zero map over the big open subset then extends to a non-zero map over the original $B'$ by pushforward). Let $V=W^{\oplus 2}$ and let $v=\rk(V)=2w$. Let $\bP=\bP_S (\oplus_{i=1}^v V^*)$ be the projectivized space of matrices with columns in $V$ and let $\pi\colon \bP\to S$ be the projection. Consider the universal basis map $\oplus_{i=1}^v \cO_\bP(-1)\to \pi^* V$, or equivalently
\[
\zeta\colon \cO_\bP^{\oplus v} \to \pi^* V \otimes \cO_\bP(1), 
\]
sending a matrix to its columns. Let $G\subseteq\bP$ be the divisor of matrices of determinant zero. Then $\zeta$ is surjective outside $G$. Taking symmetric power and composing with the surjective maps $\Sym^d V \to (\Sym^d W)^{\oplus 2} \to Q$, we get the following map
\[
U_{\Gr} \colon \Sym^d\left(\cO_\bP^{\oplus v}\right)\to \pi^* \Sym^d V \otimes \cO_\bP(d)\to \pi^* Q \otimes \cO_\bP(d),
\]
which is also surjective outside $G$. Further taking the $q$-th exterior power on both sides of $U_{\Gr}$, we obtain an induced map
\[
\widetilde{u}\colon \sum \cO_\bP \to \pi^*\det(Q)\otimes \cO_\bP(dq)
\]
which is again surjective over $\bP-G$. This gives a morphism
\[
u\colon \bP-G\to \Gr:=\Gr(w',q)\subseteq \bP^N
\]
where $w'$ is the rank of $ \Sym^d\left(\cO_\bP^{\oplus v}\right)$ and the Grassmannian is embedded in $\bP^N$ via the Pl\"ucker embedding. Let $g\colon \tbP\to \bP$ be the normalization of the blowup of the ideal sheaf corresponding to the image of $\widetilde{u}$. Then the map $u$ extends to $\tbP$ (which we still denote by $u$) and there exists an effective Cartier divisor $E\subseteq \tbP$ such that
\begin{equation} \label{eq:pullback of O_Gr(1)}
    g^*\left(\pi^*\det(Q)\otimes \cO_\bP(dq)\right) = u^*\cO_{\Gr}(1)\otimes \cO_{\tbP}(E).
\end{equation}
Let $Y$ be the image of the product map $(\pi\circ g,u)\colon \tbP\to S\times \Gr$, let $\oY$ be its closure in $\oS\times \Gr$ and let $\pi_1$ (resp. $\pi_2$) be the projection to $S$ (resp. $\Gr$). 
\begin{equation*}
    \xymatrix{
   \tbP\ar@{->}[rr]^g\ar[d]\ar[drr]^{(\pi\circ g,u)}  & &\bP\ar@{-->}[d] \\
    Y  \ar@{^{(}->}[rr] & & S\times \Gr  
    }
\end{equation*}
By assumption, every $X_s$ and $D_s$ are cut out by degree $\le d$ equations, thus the isomorphism class of $X_{\pi_1(t)}$ is uniquely determined by $\pi_2(t)$ for a general $t\in Y$. Since $f$ has maximal variation, it follows that $\pi_2\colon \oY\to \Gr$ is generically finite and hence $\pi_2^*\cO_{\Gr}(1)$ is big on $\oY$. In particular, there exists some $m\in\bN_+$ such that $\pi_2^*\cO_{\Gr}(m)\otimes \pi_1^*\cO_{\oS}(-H)$ on $\oY$ has a non-zero section. Pulling back to $\tbP$, we see that $u^*\cO_{\Gr}(m)\otimes g^*\pi^*\cO_S(-H)$ also has a non-zero section. By \eqref{eq:pullback of O_Gr(1)}, this implies
\[
H^0\big(\bP, \cO_\bP(dqm)\otimes \pi^*(\cO_S(-H)\otimes\det(Q)^m)\big)\neq 0.
\]
Pushing down to $B$, we obtain a nonzero map
\begin{equation} \label{eq:ampleness original base}
    \Sym^{dqm}\left(W^{\oplus 4w}\right)=\Sym^{dqm}\left(V^{\oplus v}\right)=(\pi_*\cO_\bP(dqm))^*\to \det(Q)^m \otimes \cO_S(-H).
\end{equation}

\smallskip

We claim that the same choice of $m$ works for the family $f'\colon (X',\Delta';L)\to S'$ as well. Indeed, most of the constructions here are functorial, namely, we have a corresponding $\pi'\colon \bP'=\bP_{S'}(\oplus_{i=1}^v V'^*)\to S'$ (where $V'=W'^{\oplus 2}$) and a rational map $u'\colon \bP' \dashrightarrow \Gr$ that extends to a proper birational model $g'\colon \tbP'\to \bP'$ such that
\begin{equation} \label{eq:pullback over B'}
    g'^*\left(\pi'^*\det(Q')\otimes \cO_{\bP'}(dq)\right) = u'^*\cO_{\Gr}(1)\otimes \cO_{\tbP'}(E')
\end{equation}
for some effective Cartier divisor $E'$ on $\tbP'$ (as before we still denote the induced map $\tbP'\to \Gr$ by $u'$). We claim that
\begin{equation} \label{eq:nonzero section on pullback}
    H^0(\tbP', u'^*\cO_{\Gr}(m)\otimes  g'^*\pi'^*\cO_{S'}(-\nu^*H))\neq 0.
\end{equation}
Indeed, by \eqref{eq:diagram bir family} we have $W'|_U = W|_U$, thus the restriction of $u'$ to $\bP'\times_{S'} U$ factors through $\bP$ and we may choose $\tbP'$ such that the restriction of $g'$ to $\bP'\times_{S'} U$ factors through $\tbP$ as well. In particular, we have the following commutative diagram
\[
\xymatrix{
\tbP' \ar[d] \ar@{-->}[rr]^{(\nu\circ \pi'\circ g',u')} & & \oY \ar[d] \ar[r] & \Gr \\
S' \ar@{-->}[rr] & & \oS
}
\]
where the rational $\tbP'\dashrightarrow \oY$ is dominant. Since $\pi_2^*\cO_{\Gr}(m)\otimes \pi_1^*\cO_{\oS}(-H)$ has a non-zero section, so does its rational pullback to $\tbP'$. As $u'\colon \tbP'\to \Gr$ is a morphism and both $\tbP'\to S'$ and $\oY\to \oS$ are proper, the rational pullback equals $u'^*\cO_{\Gr}(m)\otimes  g'^*\pi'^*\cO_{S'}(-\nu^*H)$ and this proves \eqref{eq:nonzero section on pullback}. By \eqref{eq:pullback over B'} and the same argument that proves \eqref{eq:ampleness original base}, this is enough to conclude the proof of the theorem.
\end{proof}

\section{Positivity of CM line bundle}\label{s-CMpositive}

In this section, we will put all ingredients together to prove Theorem \ref{t-main2}. The log case requires additional argument which will be developed in Section \ref{ss-perturb}. Our approach is inspired by the earlier works \cites{KP17, Pos19} in the log settings. 

\subsection{General setup} \label{sec:setup for cm>0}
Our goal is to show that the CM line bundle is big for any family of reduced uniformly K-stable Fano varieties with maximal variation. The idea is to apply \cite{BDPP} to show its intersection with any movable curve $C$ is at least the intersection of $C$ with some fixed big $\bQ$-line bundle (see \eqref{eq:CM>=(H.C)}). However, the situation is more complicated if we start with a proper subspace $M$ of the K-moduli (as in Theorem \ref{t-main2}), since a priori we only get a family over a big open set of some generically finite cover of $M$. For this reason, we consider the following somewhat technical set-up. 

\begin{notation}\label{no-setup}
Let $T$ be a torus, let $S$ be a normal projective variety and let 
\[
f\colon (X^\circ,\Delta^\circ)\to S^\circ
\]
be a generic log Fano family over an open subset $S^\circ$ of $S$ with maximal variation and a fiberwise $T$-action. Assume that very general fibers $(X_s,\Delta_s)$ are reduced uniformly K-stable with slope at least $\eta>0$ (for some $\eta>0$) and $T\subseteq \Aut(X_s,\Delta_s)$ is a maximal torus. We further introduce the following additional notation and assumptions, which will be fixed throughout the entire section.
\begin{enumerate}
    \item Let $n=\dim X_s$ and $v=(-K_{X_s}-\Delta_s)^n$.
    \item Choose some $\alpha>0$ such that $\alpha(X_s,\Delta_s)\ge \alpha$ holds for the general fibers $(X_s,\Delta_s)$ (such $\alpha$ exists by e.g. \cite{BL18}*{Theorem B}).
    \item Let $\delta=\delta(\alpha,\eta,n)>1$ be the constant given by Corollary \ref{cor:integral twist}. Decreasing $\delta$ if necessary, we may assume that $\delta\in\bQ$. Let $\gamma=\frac{\delta}{(\delta-1)(n+1)v}$ and let $0<\epsilon<\frac{1}{\delta-1}$ be such that the Weil $\bQ$-divisor $-(K_{X_s}+(1+\epsilon)\Delta_s)$ is big for all $s\in S^{\circ}$.
    \item Let $H$ be an ample line bundle on $S$. Let
    \begin{equation} \label{eq:covering family}
        \xymatrix{
        U\ar[d]_{u}\ar[r]^{p} & S \\
        V
        }
    \end{equation}
    be a covering family of curves on $S$, i.e. $u$ is a smooth projective morphism of relative dimension one, $p$ is dominant and does not contract the general fibers of $u$. We further assume that $p$ is generically finite (this can be achieved by taking hyperplane sections on $V$) and $V$ is smooth.
    \item Let $r,d\in\bN_+$ with $r\gamma\in\bN$ and let $g\colon (X,\Delta)\to U$ be a generic log Fano family with a fiberwise $T$-action that is birational to the pullback of $f\colon (X^\circ,\Delta^\circ)\to S^\circ$ such that $$L:=-r(K_{X/U}+\Delta)+2r\gamma g^*\lambda_{g,\Delta}$$ is Cartier and $g$-very ample and {\it all} fibers $(X_u,\Delta_u;L_u)$ of $g$ satisfy condition $(*)_d$ from Definition \ref{defn:condition (*)_d}.
    \item Let $C$ be the geometric generic fiber of $u$, let $q\colon C\to S$ be the induced map, let $h\colon (Y,\Delta_Y)\to C$ be the base change of $g$ (i.e. $Y=X\times_U C$ and $\Delta_Y$ is the divisorial pullback of $\Delta$) and let $\cF_C$ be the induced HN-filtration on $R=R(X_t,-r(K_{X_t}+\Delta_t))$ where $t\in C$ is a general closed point. Let $L_Y=\pi^*L$ where $\pi\colon Y\to X$ is the natural projection.
\end{enumerate}
\end{notation}
Our goal is to prove the following more general statement.

\begin{thm} \label{thm:cm>0 framework}
In the above setup, there exists a constant $c_0>0$ depending only on the family $f$, the line bundle $H$ and the integers $r,d$ $($in particular, it does not depend on $U$ or the birational pullback $(X,\Delta))$ such that 
\begin{equation} \label{eq:CM>=(H.C)}
    \deg (\lambda_{g,\Delta}|_C) = \deg \lambda_{h,\Delta_Y}\ge c_0 \deg q^*H.
\end{equation}
\end{thm}

As a first step, we use the twisted families introduced in Section \ref{sec:twisted family} to make the following reduction:

\begin{lem} \label{lem:reduction}
For the proof of Theorem \ref{thm:cm>0 framework}, we may assume that $\beta_\delta(\cF_C)\ge 0$.
\end{lem}

\begin{proof}
First note that we are free to replace $U$ by a finite cover or $V$ by a generically finite cover (and then replace $U$ and $(X,\Delta)$ by its corresponding base change): in either case we multiply the CM degree and $\deg q^*H$ by the same constant. As $(X_t,\Delta_t)$ is reduced uniformly K-stable with slope at least $\eta$ and maximal torus $T$ for some closed point $t\in C$, after replacing $U$ by a finite cover of sufficiently divisible degree, we may assume by Corollary \ref{cor:integral twist} that there exists some $\xi\in N=\Hom(\bG_m,T)$ such that $\beta_\delta(\cG)\ge 0$ where $\cG$ is the HN-filtration on $R=R(X_t,-r(K_{X_t}+\Delta_t))$ induced by the $\xi$-twist of $(Y,\Delta_Y)$ along a smooth point $P\in C$. Replacing $V$ by a generically finite cover, we may assume that $u$ admits a section $A$ which is Cartier as a divisor on $U$. Let $(X_\xi,\Delta_\xi)$ be the $\xi$-twist of $(X,\Delta)$ along $A$, then $(X_\xi,\Delta_\xi)\times_U C$ coincides with the $\xi$-twists of $(Y,\Delta_Y)$ along the smooth point $A\cap C$. Thus after replacing $(X,\Delta)$ by $(X_\xi,\Delta_\xi)$ (which is still birational to the pullback of $f$), we may assume (by our choice of $\xi$) that $\beta_\delta(\cF_C)\ge 0$ and this completes the proof.
\end{proof}

In view of Lemma \ref{lem:reduction}, we will henceforth add the following assumption to Notation \ref{no-setup}:
\begin{enumerate}
\item[(7)] We may assume  $\beta_\delta(\cF_C)\ge 0$.
\end{enumerate}

\subsection{Product trick}


The proof of Theorem \ref{thm:cm>0 framework} eventually boils down to comparing both sides of \eqref{eq:CM>=(H.C)} to certain degrees of $L_Y$. In this subsection, we combine the ampleness lemma with the product trick as in \cite[$\S$7]{KP17} or \cite[$\S$6.3]{Pos19} to provide the first part of the comparison.

\begin{lem} \label{lem:deg L>=deg H}
In the situation of Notation \ref{no-setup} $($in particular $\beta_\delta(\cF_C)\ge 0)$, there exists some constant $a_0>0$ depending only on the family $f$, the line bundle $H$ and the integers $r,d$ such that
\[
(L_Y^{n+1}) + (L_Y^n\cdot \Delta_Y) \ge a_0 \cdot \deg q^*H.
\]
\end{lem}

\begin{proof}
Let $D_0=X$, $D_1=D=\Supp(\Delta)$, $W=g_*\cO_X(L)$, $Q_i=g_*\cO_{D_i}(dL|_{D_i})$ ($i=0,1$) and $Q=Q_0\oplus Q_1$. Since $g|_D$ is flat over the codimension one point of $U$, we may shrink $V$ and assume that $g|_D$ is flat. As all the fibers $(X_u,\Delta_u;L_u)$ satisfy condition $(*)_d$, $W$ and $Q$ are locally free and their formation commutes with base change. By Theorem \ref{thm:ampleness lemma}, there exists $l,m\in\bN_+$ depending only on $r,d,H$ and the family $f$ such that there exists a non-zero map
\begin{equation}\label{e-ample}
W^{\otimes l} \to \cO_U(-p^*H)\otimes \det(Q)^{\otimes m}.
\end{equation}
We claim that $W|_C$ is nef. Indeed, since $\beta_\delta(\cF_C)\ge 0$, we see that 
$$-(K_{Y/C}+\Delta_Y)+\gamma h^*\lambda_{h,\Delta_Y}$$ 
is nef and $\deg \lambda_{h,\Delta_Y}\ge 0$ by Corollary \ref{cor:cm>=0} and Proposition \ref{prop:nef threshold}, hence $L_Y$ and 
\[
L_Y-(K_{Y/C}+\Delta_Y)=(r+1)(-(K_{Y/C}+\Delta_Y)+\gamma h^*\lambda_{h,\Delta_Y})+(r-1)\gamma h^*\lambda_{h,\Delta_Y}
\]
are both nef and $h$-ample on $Y$. It follows that $W|_C=h_*\cO_Y(L_Y)$ is nef by \cite{CP18}*{Proposition 6.4}. Now let $q_i=\rk Q_i$ ($i=0,1$) and consider the product
\[
Z=D_0^{(q_0)}\times_U D_1^{(q_1)}
\]
where we use the notation $X^{(a)}=X\times_U \cdots \times_U X$ ($a$ times) for a family $X\to U$. Since $g$ and $g|_D$ are both flat, the same holds for $\nu\colon Z\to U$ and it is not hard to see that $Z$ is reduced. Let $p_{ij}\colon Z\to D_i$ ($0\le i\le 1$, $1\le j\le q_i$) be the natural projections to factors. Let $L_Z=\bigotimes_{i,j} p_{ij}^*(L|_{D_i})$. Then by the flatness of $g|_{D_i}$ and the projection formula we have the equality $\nu_*\cO_Z(dL_Z)=\bigotimes_{i,j}Q_i$. Through the natural embeddings $\det(Q_i)\hookrightarrow \bigotimes_{j=1}^{q_i} Q_i$ we then get an embedding $\det(Q)\hookrightarrow \nu_*\cO_Z(dL_Z)$ over $U$ and hence by adjunction of $\nu_*$ and $\nu^*$ also a non-zero map $\nu^*\det(Q)\hookrightarrow \cO_Z(dL_Z)$. Composing with the map \eqref{e-ample} and restricting to $C$, we get a non-zero map
\begin{equation} \label{eq:H embed into L}
    \nu^*\left(W|_C^{\otimes l})\right) \to \cO_{Z_C}(dmL_Z-\nu^*q^*H)
\end{equation}
where $Z_C=Z\times_U C$. In particular, \eqref{eq:H embed into L} is non-zero on some irreducible component 
\[
Z'=\Delta^1\times_C\cdots\times_C \Delta^{q_0+q_1}
\] of $Z_C$, where each $\Delta^i$ is either $Y$ or an irreducible component of $\Delta_Y$. Let $p_i\colon Z'\to \Delta^i$ be the natural projections and let $L'=L_Z|_{Z'}$, then $L'=\bigotimes_{i=1}^q p_i^*(L_Y|_{\Delta^i})$. As $W|_C$ and $L_Y$ are both nef, by \eqref{eq:H embed into L} we see that $L'$ is nef and $dmL'-\nu^*q^*H$ is pseudo-effective on $Z'$, hence
\begin{align*}
    (dm+1)^{\dim Z'}\cdot \vol(L') & = \vol( (dm+1) L') \\
    & \ge \vol(L' + \nu^*q^*H) = \left( (L' + \nu^*q^*H)^{\dim Z'} \right) \\
    & \ge \left( (L')^{\dim Z' -1}\cdot \nu^*q^*H \right) = \vol(L'_t)\cdot \deg q^*H 
\end{align*}
where $t\in C$ is a general closed point.  It is clear that 
$$\vol(L'_t)=\prod_{i=1}^q \vol(L_Y|_{\Delta_t^i})$$ is bounded from below by some positive constant that only depends on the family $f$. Indeed we have $d_i:=\vol(L_Y|_{\Delta_t^i})=r^{n-1} (-K_{X_t}-\Delta_t)^{n-1}\cdot \Delta^i_t $ unless $\Delta^i=Y$, in which case $\vol(L_Y|_{\Delta_t^i})=r^n(-K_{X_t}-\Delta_t)^n$. Hence there exists some constant $a_1>0$ depending only on $r,d,H$ and the family $f$ such that
\[
\vol(L')\ge a_1\cdot \deg q^*H.
\]
On the other hand, it is straightforward to check that 
$$(L_Y^{n+1})+(L_Y^n\cdot \Delta_Y)\ge a_2\cdot \vol(L')$$ for some constant $a_2>0$ depending again only on $r,d,H$ and $f$ (indeed, $\vol(L')$ is a linear combination of $(L_Y^{n+1})$ and $(L_Y^n\cdot \Delta^i_Y)$ with positive coefficients depending on the various $d_i$; see e.g. \cite[(6.3.5.i)]{Pos19}). The lemma now follows immediately from the above two inequalities.
\end{proof}

\subsection{Perturbing the boundary} \label{ss-perturb}

In this subsection, we prove the other part of the comparison by perturbing the boundary.

\begin{lem} \label{lem:CM>=deg L}
In the situation of Notation \ref{no-setup} $($in particular $\beta_\delta(\cF_C)\ge 0)$, there exists some constant $b_0>0$ depending only on the family $f$ such that
\[
(L_Y^{n+1}) + (L_Y^n\cdot \Delta_Y) \le b_0\cdot \deg \lambda_{h,\Delta_Y}.
\]
\end{lem}

A key ingredient is given by the following result.

\begin{lem} \label{lem:pseff after perturb}
In the situation of Notation \ref{no-setup},  $-(K_{Y/C}+(1+\epsilon)\Delta_Y) +\gamma h^*\lambda_{h,\Delta_Y}$ is pseudo-effective $($as a Weil $\bQ$-divisor$)$.
\end{lem}

\begin{proof} 
We may assume that $\epsilon \in \bQ$. Since $(Y,\Delta_Y)$ is locally stable over $C$ and klt along a general fiber, it is klt. 
By \cite{BCHM}*{Corollary 1.4.4},  there exists a proper $\bQ$-factorial modification $\pi\colon Z\to Y$ which is small. Let $\Delta_Z$ be the birational transform of $\Delta_Y$ on $Z$. Denote by $\phi=h\circ \pi\colon Z\to C$.

As in the proof of Proposition \ref{prop:nef threshold}, for any rational number $\lambda>\frac{\deg \lambda_{h,\Delta_Y}}{(n+1)v}$, there exists some effective divisor $D\sim_\bQ -(K_{Y/C}+\Delta_Y)+ \lambda h^*P$ (where $P\in C$ is a smooth point) such that $(Y,\Delta_Y+\delta D)$ is klt along the general fibers of $h$. 
 It follows that the pair 
$$\big(Z, \Gamma := (1-\epsilon(\delta-1))\Delta_Z+\delta \pi^*D\big)$$ is also klt  along the general fibers of $\phi$ (note that $\Gamma$ is effective by our choice of $\epsilon$). Let $\psi\colon  W\to (Z,\Gamma)$ be a log canonical modification, whose existence follows from \cite{OX13}. Then $\phi$ is an isomorphism over a open set of $C$.
Write
$$\psi^*(K_Z+\Gamma)-G=K_W+\psi^{-1}_*\Gamma+{\rm Ex}(\psi),$$
then $G\ge 0$ by the definition and ${\rm Supp}(G)$ is vertical over $C$.

It is straightforward to verify that 
\[
K_{Z/C}+\Gamma\sim_\bQ -(\delta-1)\pi^*\big(K_{Y/C}+(1+\epsilon)\Delta_Y\big)+\delta\lambda \phi^*P
\]
and over a general closed point $t\in C$,
\[
K_{Z_t}+\Gamma_t=-(\delta-1)\pi^*(K_{X_t}+(1+\epsilon)\Delta_t),
\]
hence by our assumption $K_{Z_t}+\Gamma_t$ is big.  For any sufficiently large and divisible integers $m>0$, $\cE_m := \phi_*\cO_Z(m(K_{Z/C}+\Gamma))\neq 0$ and there is an exact sequence
$$0\to(\phi\circ\psi)_*(m(K_W+\phi^{-1}_*\Gamma+{\rm Ex}(\psi)))\to \cE_m \to G_m\to 0$$ 
for some skyscraper sheaf $G_m$.  
By \cite{Fuj17}*{Theorem 1.1}, we know $(\phi\circ\psi)_*(m(K_W+\phi^{-1}_*\Gamma+{\rm Ex}(\psi)))$ is a nef vector bundle, which implies that $\cE_m$ is nef. 

This means that for any ample line bundle $A$ on $C$ and any integer $a>0$, there exists some $b\in\bN_+$ such that $\Sym^{ab}(\cE_m)\otimes \cO_C(bA)$ is generically generated by global sections. Via the natural map
\begin{align*}
    \phi^*\left( \Sym^{ab}(\cE_m)\otimes \cO_C(bA)\right) & \to \Sym^{ab}(\phi^*\cE_m)\otimes \cO_Z(b \phi^*A) \\
     & \to \cO_Z(abm(K_{Z/C}+\Gamma)+b\phi^*A),
\end{align*}
it follows that $am(K_{Z/C}+\Gamma)+\phi^*A$ is effective. Letting $a\to \infty$ we see that $K_{Z/C}+\Gamma$ is pseudo-effective. Pushing forward to $Y$ and letting $\lambda\to \frac{\deg \lambda_{h,\Delta_Y}}{(n+1)v}$ we obtain the desired statement.
\end{proof}

\begin{proof}[Proof of Lemma \ref{lem:CM>=deg L}]
Recall that $L_Y$ is nef as in the proof of Lemma \ref{lem:deg L>=deg H}. It is not hard to check that
\begin{equation} \label{eq:deg(L)}
    (L_Y^{n+1})= (-K_{Y/C}-\Delta_Y)^{n+1} + (n+1)v\cdot 2\gamma \deg \lambda_{h,\Delta_Y} = \frac{\delta+1}{\delta-1}\deg \lambda_{h,\Delta_Y}
\end{equation}
as $\deg \lambda_{h,\Delta_Y} = - ((-K_{Y/C}-\Delta_Y)^{n+1})$. By Lemma \ref{lem:pseff after perturb}, $L_Y-\epsilon \Delta_Y$ is pseudo-effective, hence as $L_Y$ is nef we have $(L_Y^n\cdot (L_Y-\epsilon \Delta_Y))\ge 0$, or equivalently,
\begin{equation} \label{eq:deg(L|Delta)}
    \epsilon (L_Y^n\cdot \Delta_Y) \le (L_Y^{n+1}).
\end{equation}
Note that the constants $\delta$ and $\epsilon$ only depend on the family $f$, hence the result follows directly from \eqref{eq:deg(L)} and \eqref{eq:deg(L|Delta)}.
\end{proof}

\subsection{Proof of main results}

\begin{proof}[Proof of Theorem \ref{thm:cm>0 framework}]
This follows directly from Lemmas \ref{lem:reduction}, \ref{lem:deg L>=deg H} and \ref{lem:CM>=deg L}.
\end{proof}

We give some applications of Theorem \ref{thm:cm>0 framework}.

\begin{thm} \label{thm:cm>0 fix family}
Let $T$ be a torus, let $S$ be a normal projective variety and let $f\colon (X,\Delta)\to S$ be a generic log Fano family with maximal variation and a fiberwise $T$-action. Assume that very general fibers $(X_s,\Delta_s)$ are reduced uniformly K-stable and $T\subseteq \Aut(X_s,\Delta_s)$ is a maximal torus. Then the CM $\bQ$-line bundle $\lambda_{f,\Delta}$ on $S$ is big.
\end{thm}

\begin{proof}
We verify the conditions in \S \ref{sec:setup for cm>0}. By assumption and Proposition \ref{prop:slope at vg pt}, there exists some $\eta>0$ such that the very general fibers $(X_s,\Delta_s)$ are reduced uniformly K-stable with slope at least $\eta$. Fix $\gamma\in\bQ_+$ as in Notation \ref{no-setup}(3). Let $r\in \bN_+$ be such that $L:=-r(K_{X/S}+\Delta) +2 r\gamma f^*\lambda_{f,\Delta}$ is Cartier and $f$-very ample and choose $d\in\bN_+$ be such that all the fibers $(X_s,\Delta_s;L_s)$ satisfies the condition $(*)_d$ from Definition \ref{defn:condition (*)_d}. Then for any covering family of curves as in \eqref{eq:covering family}, the family $g\colon (X_U,\Delta_U)=(X,\Delta)\times_S U\to U$ satisfies all the assumptions of Theorem \ref{thm:cm>0 framework}, hence by Theorem \ref{thm:cm>0 framework}, for any fixed ample line bundle $H$ on $S$, there exists some constant $c_0>0$ depending only on $r,d,H$ and the family $f$ such that $(\lambda_{f,\Delta}\cdot C)=(\lambda_{g,\Delta_U}\cdot C)\ge c_0\cdot (H\cdot C)$ where $C$ is a very general member of the covering family. Since the covering family is arbitrary, it follows that $\lambda_{f,\Delta}-c_0 H$ is pseudo-effective by \cite{BDPP} and hence $\lambda_{f,\Delta}$ is big.
\end{proof}

 Let $\cM^{\rm Kss}_{n,v,c}$ be the Artin stack defined in $\S$\ref{ss-kmoduli} and $\phi\colon \cM^{\rm Kss}_{n,v,c}\to M^{\rm Kps}_{n,v,c}$ the corresponding good moduli space (see Theorem \ref{c-kmodulipair}). Let $\Lambda_{\CM}$ be the CM $\bQ$-line bundle on $M^{\rm Kps}_{n,v,c}$ (see Proposition \ref{p-CMLa}).

\begin{thm} \label{thm:cm>0 subsp of K-moduli}
In the above notation, let $M\subseteq M^{\rm Kps}_{n,v,c}$ be a proper algebraic subspace such that for a very general point $s\in M$, the corresponding K-polystable log Fano pair  is reduced uniformly K-stable. Then $\Lambda_{\CM}|_M$ is big.
\end{thm}

\begin{proof}
It suffices to show that $\pi^*\Lambda_{\CM}$ is big for any dominant generically finite morphism $\pi\colon S\to M$. We can also assume $S$ is a normal projective variety. 

Since $\phi\colon \cM: = \phi^{-1}(M) \to M$ is a good moduli space, \'etale locally around any point of $M$, it has the form $\phi_A\colon [{\rm Spec}(A)/G]\to {\rm Spec}(A^G)$ for a reductive group $G$ acting on an affine variety ${\rm Spec}(A)$. For the orbit geometry of $\phi_A$, see \cite{New78}*{Section 3.3}.
Let $\overline{\eta}(S)$ be the geometric generic point of $S$, i.e. $\overline{\eta}(S)={\rm Spec}(\overline{k(S)})$ for an algebraic closure of $k(S)\subset \overline{k(S)}$. Then the morphism $\overline{\eta}(S)\to S\to M$ has a lifting $\tilde{\pi}_{\eta}\colon\overline{\eta}(S)\to \cM$. Moreover,   replacing the image of $\tilde{\pi}_{\eta}$ by its K-polystable reduction, i.e. the unique closed point in the preimage  of  $\cM\times _M \overline{\eta}(S) $, we may assume the log Fano pair obtained by pulling back over  $\tilde{\pi}_{\eta}\colon \overline{\eta}(S) \to \cM$ is K-polystable. (We remark this uniqueness is proved in \cite{LWX18b}, and encoded in the good moduli space construction, see \cite{New78}*{Theorem 3.5.v} or \cite[Lemma 3.24]{AHLH18}). 
 This descends to a morphism ${\rm Spec}(T)\to \cM$, where
$k(S)\to k(T)$ is a finite extension. Spreading out ${\rm Spec}(T)\to \cM$, it implies we may replace $S$ by a finite cover, and  assume that there exists an open subset $S^\circ\subseteq S$ such that the map $\pi^\circ:=\pi|_{S^\circ}$ lifts to $\widetilde{\pi}^\circ\colon S^\circ\to \cM$. We may further assume (after possibly shrinking $S^\circ$) that for all $s\in S^\circ$, $\widetilde{\pi}^\circ (s)$ is the unique closed point lying over $\pi(s)\in M$, as the closed point locus form a constructible subset of $\cM$. Indeed the constructibility of such locus can be seen \'etale locally  on $[{\rm Spec}(A)/G]\to {\rm Spec}(A^G)$,
where the locus consists of $[V/G]$ for 
$$V:=\{x\in {\rm Spec} (A) \ | \ \dim (G\cdot x)\le \dim (G\cdot y), \mbox{ for any }y\in \phi_A^{-1}(\phi_A(x)) \}.$$
As a consequence, we get a $\bQ$-Gorenstein family $f\colon (X^\circ,\Delta^\circ)\to S^\circ$ of K-polystable log Fano pairs induced by the morphism $\widetilde{\pi}^\circ$. 
Since the automorphism functor $\Aut_{S^{\circ}}(X^{\circ},\Delta^{\circ})$ of (polarized) pairs is represented by an algebraic group scheme over $S^{\circ}$ (of finite type), shrinking $S^{\circ}$ we can assume $\Aut_{S^{\circ}}(X^{\circ},\Delta^{\circ})\to S^{\circ}$ is a smooth group scheme.
Then replacing by another finite cover and shrinking $S^\circ$, we may also assume that the maximal torus of $\Aut_{S^{\circ}}(X^{\circ},\Delta^{\circ})$ is split over $S^{\circ}$, i.e. that $f$ admits a fiberwise $T$-action ($T$ being a torus) such that $T\subseteq \Aut(X_s,\Delta_s)$ is a maximal torus. By assumption, any very general fiber $(X_s,\Delta_s)$ is reduced uniformly K-stable, hence by Proposition \ref{prop:slope at vg pt} we may find some $\eta>0$ such that a very general fiber $(X_s,\Delta_s)$ is reduced uniformly K-stable with slope at least $\eta$. 

Now fix the constant $\gamma\in\bQ_+$ as in Notation \ref{no-setup}(3) and choose $r,d\in\bN_+$ such that
\begin{enumerate}
    \item $r\gamma \pi^*\Lambda_\CM$ is Cartier,
    \item for any K-semistable log Fano pairs $(X,\Delta)$ with $\dim X=n$, $\vol(-K_X-\Delta)=v$ and $\Delta=cD$ for some integral Weil divisor $D$, we have $-r(K_X+\Delta)$ is Cartier and very ample and the triple $(X,\Delta;-r(K_X+\Delta))$ satisfies the condition $(*)_d$ from Definition \ref{defn:condition (*)_d} (this is possible since the set of such log Fano pairs is bounded by \cite{Jia17,Che18, LLX18}).
\end{enumerate}
Let $H$ be a fixed ample line bundle on $S$. Let $V\leftarrow U\to S$ be a covering family of curves as in \eqref{eq:covering family}. By \cite{AHLH18}*{Theorem A.8}, after possibly replacing $U$ by a finite cover and shrinking $V$, we may extend the birational pullback of $f$ to a $\bQ$-Gorenstein family $g\colon (X,\Delta)\to U$ of K-polystable log Fano pairs. In addition, since $\cM$ is $\Theta$-reductive (see \cite{AHLH18}*{Definition 3.10} and \cite{ABHX19}*{Theorem 1.1}, this is part of the requirement for a stack to have a good moduli space) and $[(X_u,\Delta_u)]\in \cM$ is a closed point for every $u\in U$, every $\bG_m$-action on the generic fiber of $g$ induces a $\bG_m$-action on $(X_u,\Delta_u)$ for every codimension one point $u\in U$. In particular, after possibly shrinking $V$ we may assume that the family $g$ has a fiberwise $T$-action. Note that we also have $\pi_U^*\Lambda_\CM = \lambda_{g,\Delta}$ where $\pi_U\colon U\to M$ is the induced map. By our choice of $r$ and $d$, it is clear that the family $g$ satisfies the assumptions of Theorem \ref{thm:cm>0 framework} and hence there exists some constant $c_0>0$ depending only on $r,d,H$ and the family $f$ such that 
$$(\pi^*\Lambda_\CM\cdot C)=(\lambda_{g,\Delta}\cdot C)\ge c_0\cdot (H\cdot C).$$ Since the covering family of curves is arbitrary, it follows that $\pi^*\Lambda_\CM - c_0 H$ is pseudo-effective and therefore $\pi^*\Lambda_\CM$ is big as desired.
\end{proof}

The following is a natural generalization of Theorem \ref{t-main2} into a log version.
\begin{thm} \label{t-main2log}
In the above notation, let $M\subseteq M^{\rm Kps}_{n,v,c}$ be a proper algebraic subspace such that every geometric point $s\in M$ parametrizes a reduced uniformly K-stable log Fano pair. Then $\Lambda_{\CM}|_M$ is ample.
\end{thm}

\begin{proof}
Since $\Lambda_{\CM}|_M$ is nef by \cite{CP18}*{Theorem 1.8} (or Corollary \ref{cor:cm>=0}), this directly follows from Theorem \ref{thm:cm>0 subsp of K-moduli} and the Nakai-Moishezon criterion.
\end{proof}

Using analytic tools and Theorem \ref{t-logsmoothable},  we also have the following theorem. 

\begin{thm}\label{t-logample}
Let $k=\bC$.  Notation as in  Theorem \ref{t-logsmoothable}. Then the restriction of $\Lambda_{\CM}$ on $\overline{M}^{\rm sm, Kps}_{n,v,c}$  is ample.   
\end{thm}

\begin{proof}
By Theorem \ref{t-logsmoothable}, $\overline{M}^{\rm sm, Kps}_{n,v,c}$ is known to be proper. By \cite{TW19} (see also \cite{ADL19}*{Theorem 3.6}), we know the log Fano pairs parametrized by $\bC$-valued points of  $\overline{M}^{\rm sm, Kps}_{n,v,c}$ all admit weak conical K\"ahler-Einstein metrics, thus they are reduced uniformly K-stable by \cite{Li19} (see also \cite{BBJ15}). Therefore, we conclude by Theorem \ref{t-main2log}.
\end{proof}

\appendix
\section{Reduced $\delta$-invariants}\label{a-reduced}
In this section, we develop a reduced version of $\delta$-invariant for log Fano pairs $(X,\Delta)$ with a torus group action. Results in this section are not needed in the main part of the current article. 

Throughout, we fix a torus group $T$ and let $N=\Hom(\bG_m,T)=M^*$. 
Let $(X,\Delta)$ be a log Fano pair with a $T$-action. Recall from \S \ref{ss-reduks} (see also \cite{Li19}*{\S 2.3}) that any $\xi\in N_{\bR}$ determines a valuation $\wt_{\xi}$ given by 
$$\wt_{\xi}\colon f=\sum_{\alpha\in M, f_{\alpha}\neq 0} f_{\alpha} \mapsto \min\langle \alpha, \xi  \rangle.$$

\begin{defn}
Using the notation from \S \ref{ss-reduks}, for any $T$-equivariant valuation $v$ which is not of the form $\wt_{\xi}$ and $A_{X,\Delta}(v)<\infty$, we can define  \footnote{We want to thank the referee for suggesting this definition.}
\begin{equation} \label{e-tdelta(v)}
    \delta_{X,\Delta,T}(v)=1+\sup_{\xi\in N_\bR}\frac{\beta_{X,\Delta}(v)}{S_{X,\Delta}(v_{\xi})}
\end{equation}
(or abbreviated as $\tdelta(v)$ if $(X,\Delta)$ is clear). We define the {\it $T$-reduced $\delta$-invariant} as
\begin{eqnarray}
\tdelta(X,\Delta):=\inf_{v} \delta_{X,\Delta,T}(v),
\end{eqnarray}
where $v$ runs through all $T$-equivariant valuations with $A_{X,\Delta}(v)<\infty$ which are not of the form $\wt_{\xi}$.
\end{defn}

\begin{rem} \label{r-tdelta>=1}
We are mostly interested in the case where $\Fut(\xi)=0$ for all $\xi\in N$, e.g. $(X,\Delta)$ is K-semistable. In this case, we have $\beta(v)=\beta(v_{\xi})$ by \cite{Li19}*{Proposition 3.12}, 
thus 
 $$ \delta_{X,\Delta,T}(v)=\sup_{\xi\in N_\bR}\frac{A_{X,\Delta}(v_{\xi})}{S_{X,\Delta}(v_{\xi})}\qquad{\mbox{and}}\qquad \tdelta(X,\Delta)=\inf_{\mu}\sup_{\xi} \frac{A_{X,\Delta}(v_{\mu,\xi})}{S_{X,\Delta}(v_{{\mu,\xi}})},$$
where in the second expression the first infimum runs through all non-trivial valuations $\mu$   on the Chow quotient $Z$ such that $A_{X,\Delta}(v_{\mu,{\xi}})<\infty$ for some (and equivalently any) $\xi \in {N}_{\bR}$, and the second supremum runs through all $\xi \in {N}_{\bR}$. 
\end{rem}

\begin{rem}
By Theorem \ref{thm:Li's valuative T-uKs}, a K-semistable log Fano pair $(X,\Delta)$ is reduced uniformly K-stable if and only if $\tdelta(X,\Delta)>1$ for some maximal torus $T$ in $\Aut(X,\Delta)$.
\end{rem}

\begin{rem} \label{r-max}
If $\beta_{X,\Delta}(v)\ge 0$ and $\Fut(\xi) = 0$ for all $\xi\in N$, then the supremum in \eqref{e-tdelta(v)} is a maximum. Indeed, 
by \cite{BJ17}*{Proposition 3.11}, we have
\begin{equation} \label{e-S>=J}
    S(v_\xi) \ge \frac{1}{n+1} T(v_\xi) \ge \frac{1}{n+1}\JNA(\cF_{v_\xi})
\end{equation}
where $n=\dim X$, hence by the properness estimate in \cite{Li19}*{Lemma 3.15}, we know that it suffices to take the supremum in \eqref{e-tdelta(v)} over a compact subset of $N_\bR$ and therefore it is achieved for some $\xi$ by the continuity of $\xi \mapsto S(v_\xi)$. 
\end{rem}

The above definition is a modification of the characterization of $\delta$-invariant given by \cite[Theorem C]{BJ17}. We want to prove the following theorem which is an analogue of \cite[Theorem 4.5]{BLX19}.

\begin{thm}\label{t-redqm}
Let $(X,\Delta)$ be a log Fano pair with a $T$-action. If $(X,\Delta)$ is K-semistable and $\tdelta(X,\Delta)= 1$, then it can be calculated by a $T$-invariant quasi-monomial valuation $v\neq \wt_{\xi}$, i.e. there exists a quasi-monomial $T$-invariant valuation $v$ which is not of the form $\wt_{\xi}$ and satisfies that
$$\frac{A_{X,\Delta}(v_\xi)}{S_{X,\Delta}(v_\xi)}=\tdelta(v)=\tdelta(X,\Delta)=1$$
for all $\xi \in N_\bR$.
\end{thm}

We first prove a number of lemmas. To this end, let $(X,\Delta)$ be a log Fano pair (not necessarily K-semistable), let $r$ be a positive integer such that $-r(K_X+\Delta)$ is Cartier and let $R$ be the section ring $R(X,-r(K_X+\Delta))$.

\begin{lem}\label{l-twistlct}
Let $\cF$ be a $T$-equivariant filtration on $R$ and $\mu=\mu(\cF)$.  Let $v$ be a $T$-invariant valuation that computes the log canonical threshold of $I^{(\mu)}_{\bullet}(\cF)$.
Then $\mu=\mu(\cF_{\xi})$ and $v_{\xi}$ computes the log canonical threshold of $I_{\bullet}^{(\mu)}(\cF_{\xi})$.
\end{lem}

\begin{proof}
Let $I_{m,\lambda}$ (resp. $I^\xi_{m,\lambda}$) be the base ideals of $\cF$ (resp. $\cF_\xi$) and let $I_{m,\lambda,\alpha}$ (resp. $I^\xi_{m,\lambda,\alpha}$) ($\alpha\in M$) be their weight-$\alpha$ part, e.g. $I_{m,\lambda,\alpha} = {\rm Im}((\cF^\lambda R_m)_\alpha \otimes \cO_X(mr(K_X+\Delta))\to \cO_X)$. 
As in the proof of Theorem \ref{thm:tbeta>=D^NA}, after rescaling $v$, by \eqref{eq:f convex}, we have 
$$v(I_{m,\lambda,\alpha})\ge v(I_{m,\lambda})\ge m\cdot v(I_\bullet^{(\lambda/m)})\ge \lambda+m({rA_{X,\Delta}(v)}-\mu)$$
for any $\alpha\in M$. Then for any $t$, let $\theta_{\xi}(v)=A_{X,\Delta}(v_{\xi})-A_{X,\Delta}(v)$, we have
\begin{eqnarray*}
v_{\xi}(I^{\xi}_{m,tm,\alpha})&=&v_{\xi}(I_{m,tm-\la \xi,\alpha\ra,\alpha})\\
&=&v(I_{m,tm-\la \xi,\alpha\ra,\alpha})+\la\xi,\alpha \ra+mr \theta_{\xi}(v)\\
&\ge& m(t+r \theta_{\xi}(v)+{rA_{X,\Delta}(v)}-\mu)\\
&=&m(t+rA_{X,\Delta}(v_{\xi})-\mu),
\end{eqnarray*}
where the second equality follows from \cite{Li19}*{Proposition 3.8}.

Taking $t=\mu$, we see that $v_{\xi}(I^{\xi}_{m,\mu m,\alpha})\ge mr A_{X,\Delta}(v_{\xi})$ and hence
\begin{eqnarray}\label{e-twistlct}
v_{\xi}(I^{\xi}_{m,\mu m})\ge mr A_{X,\Delta}(v_{\xi}).
\end{eqnarray}
Thus $\mu(\cF_{\xi})\le \mu(\cF)$. Since we can take a $(-\xi)$-twist of $\cF_{\xi}$ to get $\cF$, this implies that 
$$\mu(\cF_{\xi})\le \mu(\cF)\le \mu(\cF_{\xi}).$$ 
It follows that equality holds and then \eqref{e-twistlct} implies that $v_{\xi}$ computes the log canonical threshold of $I^{(\mu)}_{\bullet}(\cF_\xi)$. 
\end{proof}

Recall that for an lc pair $(X,\Delta)$, an {\it lc place $E$} is a divisorial valuation over $X$ such that $A_{X,\Delta}(E)=0$. If $(X,\Delta)$ is a log Fano pair, {\it an $N$-complement} is a $\bQ$-divisor $D$ such that $N(K_X+\Delta+D)\sim 0$ and $(X,\Delta+D)$ is log canonical.  {\it A $\bQ$-complement} of a log Fano pair is an $N$-complement for some $N\in \bN_+$. 

\begin{lem}\label{l-comple}
A divisor $E$ over $X$ is an lc place of a $\bQ$-complement of $(X,\Delta)$ if and only if ${\rm gr}_{\cF_E}R:=\bigoplus_{m,i\in\bN} \Gr^i_{\cF_E} R_m$ is finitely generated and $\mu(\cF_E)=rA_{X,\Delta}(E)$. 
\end{lem}

\begin{proof}
Assuming $E$ is an lc place of a $\bQ$-complement of $(X,\Delta)$, then  ${\rm gr}_{\cF_E}R$ is finitely generated and $\mu(\cF_E)=rA_{X,\Delta}(E)$ by Proposition \ref{prop:tbeta property}. Conversely, if $\mu(\cF_E)=rA_{X,\Delta}(E)$, then since the function $t\mapsto \lct(X,\Delta;I^{(t)}_\bullet)$ is continuous on $(0,T_{X,\Delta}(E))$ (which in turn follows from the fact that the function $t\mapsto v(I^{(t)}_\bullet)$ is convex for any valuation $v$ on $X$), we have $\lct(X,\Delta;I^{(r\cdot A_{X,\Delta}(E))}_\bullet)\ge \frac{1}{r}$ by the definition of log canonical slope. On the other hand, it is clear that 
\[
\lct(X,\Delta;I^{(r\cdot A_{X,\Delta}(E))}_\bullet)\le \frac{A_{X,\Delta}(E)}{\ord_E(I^{(r\cdot A_{X,\Delta}(E))}_\bullet)}\le \frac{A_{X,\Delta}(E)}{r\cdot A_{X,\Delta}(E)} = \frac{1}{r},
\]
thus $\lct(X,\Delta;I^{(r\cdot A_{X,\Delta}(E))}_\bullet) = \frac{1}{r}$.
If in addition ${\rm gr}_{\cF_E}R$ is finitely generated, then
\[
\lct(X,\Delta;I^{ (r\cdot A_{X,\Delta}(E))}_\bullet)=m\cdot \lct(X,\Delta;I_{m, mrA_{X,\Delta}(E))})= \frac{1}{r}.
\]
for some sufficiently divisible $m$. This means there is a divisor $D\in |-mr(K_X+\Delta)|$ with $\ord_E(D)\ge mrA_{X,\Delta}(E)$ and $(X,\Delta+\frac{1}{mr}D)$ is log canonical. Thus $E$ is an lc place of $(X,\Delta+\frac{1}{mr}D)$.  
\end{proof}

For the next lemma we use the following notation: if $E$ is a $T$-invariant divisor over $X$, $v=\ord_E$ and $\xi\in N_\bQ$, then $v_\xi$ is also a divisorial valuation over $X$ and we define $E_\xi$ as the divisor over $X$ such that $v_\xi=c\cdot \ord_{E_\xi}$ for some $c\in\bQ$.

\begin{lem}\label{l-twistinvariant}
If $E$ is an lc place of a $\bQ$-complement, then for any $\xi\in N_{\bQ}$, $ E_{\xi}$ is also an lc place of a $\bQ$-complement.
\end{lem}

\begin{proof}
Since ${\rm gr}_{\cF_E}R$ is finitely generated by assumption and ${\rm gr}_{\cF_E}R\cong {\rm gr}_{\cF_{E_{\xi}}}R$, we know the latter is also finitely generated. By Lemma \ref{l-comple}, it suffices to prove $\mu(\cF_{E_{\xi}})=rA_{X,\Delta}(E_{\xi})$, or equivalently $$\mu(\cF_{v_\xi})=rA_{X,\Delta}(v_\xi)$$
where $v=\ord_E$. Since $\mu(\cF_v)=rA_{X,\Delta}(v)$ by Lemma \ref{l-comple}, we have $\mu((\cF_v)_\xi)=rA_{X,\Delta}(v)$ by Lemma \ref{l-twistlct}. By \cite{Li19}*{Proposition 3.8}, $\cF_{v_\xi}$ differs from $(\cF_v)_\xi$ by a translation of $r \cdot \theta_{\xi}(v)$. It is then clear that 
$$\mu(\cF_{v_\xi})=\mu((\cF_v)_\xi)+r \cdot \theta_{\xi}(v) = r(A_{X,\Delta}(v)+\theta_{\xi}(v)) = rA_{X,\Delta}(v_{\xi})$$
and we are done.
\end{proof}

\begin{lem}[{\cite{BLX19}}]\label{l-weakly}
Let $(X,\Delta)$ be a log Fano pair and let $E$ be a divisor over $X$. There exists an integer $N>0$ depending only on $\dim (X)$ and the coefficients of $\Delta$ such that the following are equivalent:
\begin{enumerate}
\item $E$ is an lc place of a $\bQ$-complement;
\item $E$ is an lc place of an $N$-complement;
\item $E$ is induced by a weakly special test configuration with irreducible central fiber. 
\end{enumerate}
\end{lem}
\begin{proof}The equivalence between (1) and (2) follows from \cite[Theorem 3.5]{BLX19} whose proof relies on \cite{Bir19}; whereas the equivalence between (2) and (3) follows from \cite[Theorem A.2]{BLX19}. 
\end{proof}

To prove Theorem \ref{t-redqm}, we also need a constructibility result (similar to \cite{BLX19}*{Proposition 4.1}) of $\tdelta(v)$ when the valuation varies in a family. The following definition is a refinement of \cite{BLX19}*{Definition 2.2} and will be needed in the proof of Theorem \ref{t-redqm}.

\begin{defn} \label{d-flogres}
Let $f_i\colon (X_i,\Delta_i+M_i)\to B$ $(i=1,\cdots,m)$ be projective pairs over $B$ (where each $\Delta_i$ is a divisor on $X_i$ and $M_i$ is a $\bQ$-linear system, i.e. $M_i=a_i\cM_i$ for some $a>0$ and some linear series $\cM_i$ on $X_i$) such that the $X_i$'s are birational to each other over $B$. We say that $\phi \colon Y\to B$ gives a simultaneous fiberwise log resolution of the $f_i$'s if 
\begin{enumerate}
    \item There are proper birational morphisms $g_i\colon Y\to X_i$ such that $\phi=f_i\circ g_i$ for all $i$.
    \item We can write $g_i^*\cM_i=\Phi_i+F_i$ where $F_i$ (resp. $\Phi_i$) is the fixed (resp. movable) part over $B$ such that $\Phi_i$ is base point free over $B$, $G:=\Supp(\sum_{i=1}^m({\rm Exc}(g_i)+(g_i^{-1})_*\Delta_i+F_i))$ is an snc divisor and each stratum of $G$ is smooth over $B$ with irreducible fibers.
\end{enumerate}
\end{defn}

Consider now the following setup: let $B$ be a smooth variety and let $(\cX,\cD)\to B$ be a $\bQ$-Gorenstein family of log Fano pairs with a fiberwise $T$-action. Let $\cM\sim_\bQ -(K_{\cX/B}+\cD)$ be a $T$-invariant $\bQ$-linear system such that $(\cX_b,\cD_b+\cM_b)$ is lc for all $b\in B$ and let $g\colon \cY\to (\cX,\cD+\cM)$ be a fiberwise $T$-equivariant log resolution (i.e. $g$ is $T$-equivariant and is a fiberwise log resolution in the sense of Definition \ref{d-flogres}).

\begin{lem} \label{l-constructible}
In the above setup, let $\cE$ be a toroidal divisor over $\cY$ with respect to $\cG$ such that $A_{\cX,\cD+\cM}(E)<1$. Then $\delta_{\cX_b,\cD_b,T}(\cE_b)$ is independent of $b\in B$.
\end{lem}

\begin{proof}
We follow the proof of \cite{BLX19}*{Proposition 4.1}, which is in turn based on \cite{HMX13}*{Theorem 1.8}. We may assume $B$ is affine and $\cE$ is a prime divisor on $\cY$ (by repeatedly blowup centers of $\cE$ on $\cY$). We aim to show that the natural restrictions
\begin{equation} \label{e-invplurigenera}
    H^0(\cY,-m g^*(K_{\cX}+\cD)-\ell \cE) \to H^0(\cY_b,-m g^*(K_{\cX_b}+\cD_b)-\ell \cE_b)
\end{equation}
are surjective for all sufficiently divisible integers $m,\ell \in\bN$.

By Bertini's theorem, there are effective $\bQ$-divisors $H\sim_\bQ -(K_{\cX/B}+\cD)$ and $M\in \cM$ such that $g$ is also a fiberwise log resolution of $(\cX,\Gamma=\cD+\epsilon H +(1-\epsilon)M)$, $(\cX_b,\Gamma_b)$ is klt for all $b\in B$ and $A_{\cX,\Gamma}(\cE)<1$ (note that $(\cX,\Gamma)$ no longer has a $T$-action but this does not affect the proof). We may write
\[
K_{\cY} + a\cE + \Gamma_1 - \Gamma_2 = g^*(K_{\cX}+\Gamma) \sim_\bQ 0
\]
where $a=1-A_{\cX,\Gamma}(\cE)$, $\Gamma_1$ and $\Gamma_2$ are effective without common component and $\Gamma_2$ is $g$-exceptional. Since $(\cX_b,\Gamma_b)$ is klt, so does $(\cY_b,(\Gamma_1)_b)$ for all $b\in B$. We then have
\[
-m g^*(K_{\cX}+\cD)-\ell \cE + \frac{\ell}{a}\Gamma_2 \sim \frac{\ell}{a}(K_{\cY} + \Gamma_1)-m g^*(K_{\cX}+\cD) \sim \frac{\ell}{a}(K_{\cY} + \Gamma_1 + H')
\]
for some effective $H'\sim_\bQ -\frac{am}{\ell}g^*(K_{\cX}+\cD)$ such that $(\cY_b,(\Gamma_1)_b+H'_b)$ is klt for all $b\in B$. From the proof of \cite{HMX13}*{Theorem 1.8(1)}, we see that the natural maps
\[
H^0\left(\cY,-m g^*(K_{\cX}+\cD) - \ell \cE + \frac{\ell}{a}\Gamma_2\right) \to H^0\left(\cY_b,-m g^*(K_{\cX_b}+\cD_b)-\ell \cE_b + \frac{\ell}{a}(\Gamma_2)_b\right)
\]
are surjective for all sufficiently divisible $m,\ell \in\bN$. But since $(\Gamma_2)_b$ is $g_b$-exceptional, the two $H^0$'s above can be identified with the ones in \eqref{e-invplurigenera} and thus \eqref{e-invplurigenera} follows.

Since $\cY\to B$ admits a fiberwise $T$-action, the maps in \eqref{e-invplurigenera} are $T$-equivariant and hence are also surjective on each component of the weight decomposition. It follows that the for each sufficiently divisible $m,\ell\in\bN$ and each $\alpha\in M$, $\dim (\cF^\ell_{\cE_b} R_{b,m})_\alpha$ is independent of $b\in B$ (where $R_b$ is the section ring of $-r(K_{\cX_b}+\cD_b)$). Recall that $\cF_{v_\xi}$ differs from $(\cF_v)_\xi$ by a translation of $r \cdot \theta_{\xi}(v)$ (see \cite{Li19}*{Proposition 3.8}) and $\lambda_{\min}(\cF_v)=0$ for any valuation $v$, then for each $\xi\in  N_\bR$, 
$$\theta_\xi(v_b) =  -\lambda_{\min}((\cF_{v_b})_\xi)$$ is independent of $b\in B$ (where $v_b=\ord_{\cE_b}$). Clearly $A_{\cX_b,\cD_b}(v_b)$ is also independent of $b\in B$. As a consequence,
\[
A_{\cX_b,\cD_b}((v_b)_\xi) \quad \text{and} \quad S_{\cX_b,\cD_b}((v_b)_\xi)
\]
are both independent of $b\in B$ as well. It is now evident from the definition \eqref{e-tdelta(v)} that $\delta_{\cX_b,\cD_b,T}(\cE_b)$ is independent of $b\in B$.
\end{proof}

We are now ready to present the proof of Theorem \ref{t-redqm}. The strategy is quite similar to the proof of \cite{BLX19}*{Theorem 4.5}: using the constructibility result Lemma \ref{l-constructible}, we aim to find a sequence of lc places of a \emph{fixed} complement that approximately computes $\delta_T(X,\Delta)$ and take their limit in the dual complex. However, we are in trouble if the limit valuation is of the form $\wt_\xi$ for some $\xi\in N_\bR$. A na\"ive approach is to twist the valuations by some $\xi\in N_\bR$ before taking the limit, but a priori this also changes the complement and the resulting valuations may no longer be lc places of a \emph{fixed} pair. To avoid these issues, we take a common log resolution of $(X,\Delta+M)$ (where $M$ is the complement) and (suitable compactification of) $T\times (X\!/\!/T)$ in family (i.e. over some parameter space of bounded complement) in the sense of Definition \ref{d-flogres}. The additional $T\times (X\!/\!/T)$ then takes care of the $\wt_\xi$ component of the lc places and ensures that it stays constant when moving along the family.

\begin{proof}[Proof of Theorem \ref{t-redqm}]
Let $N$ be the integer from Lemma \ref{l-weakly}. By Lemma \ref{l-twistinvariant} and Lemma \ref{l-weakly}, if $E$ is a $T$-invariant divisor over $X$ that is an lc place of an $N$-complement, then so is $E_{\xi}$ for any $\xi\in N_{\bQ}$. 

We first prove that there exists a sequence of $T$-invariant divisors $E_i$ over $X$, each of which is an lc place of an $N$-complement, such that $\ord_{E_i}\neq \wt_{\xi}$ for any $\xi\in N_{\bQ}$ and $\lim_{i\to \infty}\tdelta(E_i)=1$. In fact, if this fails, 
then by Lemma \ref{l-weakly}, there exists some constant $a>0$ such that for any divisorial valuation $v=\ord_E$ that is induced by a $T$-equivariant special test configuration $(\cX^s,\Delta_{\cX^s})$, we have $\tdelta(v) > 1+a$. Thus by the definition of $\tdelta$, there is a twist $\xi\in N_\bR$ such that
$$\beta_{X,\Delta}(v)\ge a\cdot S(v_\xi)\ge \frac{a}{n+1}\JNA(\cF_{v_\xi})\ge \frac{a}{n+1} \JNA_T(\cF_v)$$
where the first inequality follows from the definition of $\tdelta(v)$, the second inequality follows from \eqref{e-S>=J}, and the last inequality follows from Corollary \ref{cor:DNA & JNA after translation} and the fact that $\cF_{v_\xi}$ differs from $(\cF_v)_\xi$ by a translation. Since $\beta_{X,\Delta}(v)=\DNA(\cX^s,\Delta_{\cX^s})$ by \cite{Fuj19}*{Theorem 5.1}, it follows that
\[
\DNA(\cX^s,\Delta_{\cX^s})\ge \frac{a}{n+1} \JNA_T(\cX^s,\Delta_{\cX^s})
\]
for any $T$-equivariant special test configuration $(\cX^s,\Delta_{\cX^s})$ of $(X,\Delta)$. By \cite{Li19}*{\S 4} (which uses equivariant MMP and a similar argument as in \cite{LX14}), this implies
$$\DNA(\cX,\Delta_{\cX})\ge \frac{a}{n+1} \cdot  \JNA_T(\cX,\Delta_{\cX})$$
for any $T$-equivariant test configuration $(\cX,\Delta_{\cX})$, hence $(X,\Delta)$ is reduced uniformly K-stable and $\tdelta(X,\Delta)>1$, a contradiction. 

Fix a sequence $E_i$ ($i\in\bN$) with the aforementioned properties and a $T$-equivariant birational map $X\dashrightarrow \oT\times Z$ where $Z$ is proper and $\oT$ is a toric variety that compactifies $T$. Via \eqref{e-valuation} we get a (non-canonical) isomorphism $\Val(Z)\times N_\bR \cong \Val^T(X)$ sending $(\mu,\xi) \mapsto v_{\mu,\xi}$ and let $\pi\colon \Val^T(X)\to N_\bR$ be the induced projection. By Lemma \ref{l-twistinvariant}, we may replace each $E_i$ by a twist and assume that $\pi(\ord_{E_i})=0$. 

We now run a modified argument of \cite{BLX19}*{Proof of Theorem 4.5}. Consider the parameter space $B$ of $T$-invariant linear series $\cM_b\subseteq |-N(K_X+\Delta)|$ that give strict $N$-complement of $(X,\Delta)$, i.e. $\lct(X,\Delta;\cM_b)=\frac{1}{N}$. There exists a locally closed decomposition $B=\cup B_j$ and \'etale maps $B'_j\to B_j$ such that the $B'_j$'s are smooth and the universal family $(X\times B,\Delta\times B;\cM)$ together with $(\oT,\Gamma)\times Z\times B$ (where $\Gamma$ is the sum of torus invariant divisors on $\oT$) admits a simultaneous fiberwise $T$-equivariant log resolution over each $B'_j$.

For any $E_i$,  the linear system 
$$M_i:=\cF_{E_i}^{N\cdot A_{X,\Delta}(E_i)}H^0\big(\cO_X(-N(K_X+\Delta))\big)\subset H^0\big(\cO_X(-N(K_X+\Delta))\big)$$ is a $T$-invariant linear system which satisfies that $\lct(X,\Delta';M_i)=\frac{1}{N}$ and $E_i$ is an lc place of $(X,\Delta+\frac{1}{N}M_i)$. In particular, $M_i$ yields a point on $B$.  Passing to a subsequence of $E_i$ and restricting to some component $B'_j$, we may and do assume that $B$ is irreducible, simultaneous fiberwise $T$-equivariant log resolutions 
    \begin{equation*} 
        \xymatrix{
    &    \cY\ar[dl]_{g}\ar[dr]^{g'}  &\\
        (X\times B,\Delta\times B+\cM) & & (\oT,\Gamma)\times Z\times B
        }
    \end{equation*}
exists over $B$ and every $E_i$ is an lc center of $(X,\Delta+\frac{1}{N}\cM_{b_i})$ for some $b_i\in B$. In particular, there exists an lc place $\cE_i$ of $(X\times B,\Delta\times B+\frac{1}{N}\cM)$ that restricts to $E_i$ over $b_i$. 

Fix $b_0\in B$. 
Since $\pi((\cE_i)_b)=0$ if and only if the center of $(\cE_i)_b$ is not contained in any component of $g_b'^*(\Gamma\times Z)$ and the latter statement is independent of $b\in B$, we see that $\pi((\cE_i)_{b_0})=0$ as the same holds over $b_i$. By Lemma \ref{l-constructible}, we also have 
$$\delta_{X,\Delta,T}(E_i)=\delta_{X,\Delta,T}((\cE_i)_{b_0}).$$ Therefore, we may replace the sequence $E_i$ by $(\cE_i)_{b_0}$ and assume that the $E_i$'s are lc places of a fix lc pair $(X,\Delta+\frac{1}{N}\cM_{b_0})$.

By \cite{BLX19}*{Lemma 2.3}, we know that $v_i:=\frac{1}{A_{X,\Delta}(E_i)}(\ord_{E_i})$ converges to a $T$-invariant quasi-monomial valuation $v$ over $X$. Since $\pi(v_i)=0$ and $A_{X,\Delta}(v_i)=1$, we see that $\pi(v)=0$ and $A_{X,\Delta}(v)=1$ as well; in particular, $v\neq \wt_\xi$ for any $\xi\in N_\bR$. We will show for such $v$, $\tdelta(v)=1$.

After twisting by $\xi$, we also have $(v_i)_\xi \to v_\xi$ and $A_{X,\Delta}((v_i)_\xi)\to A_{X,\Delta}(v_\xi)$. By \cite{BLX19}*{Proposition 2.4}, we have $S(v_i)\to S(v)$ and therefore as 
$$S(v_\xi)=A_{X,\Delta}(v_\xi)-A_{X,\Delta}(v)+S(v)+\Fut(\xi) \qquad \mbox{ (by \cite{Li19}*{(130)}}),$$ we also have $S((v_i)_\xi)\to S(v_\xi)$ for all $\xi\in N_\bR$. It follows that for any $\xi \in N_\bR$, 
\[
\frac{A_{X,\Delta}(v_\xi)}{S(v_\xi)} = \lim_{i\to \infty} \frac{A_{X,\Delta}((v_i)_\xi)}{S((v_i)_\xi)} \le \lim_{i\to \infty} \tdelta(v_i) = 1
\]
and hence $\tdelta(v)\le 1$. Since we always have $\tdelta(v)\ge 1$ (Remark \ref{r-tdelta>=1}), thus $\tdelta(v)=1$. By Remark \ref{r-max}, we also know that $\tdelta(v)=\frac{A_{X,\Delta}(v_\xi)}{S(v_\xi)}=1$ for some $\xi\in N_\bR$. It follows that $\beta(v)=\beta(v_\xi)=0$ for all $\xi \in N_\bR$. In other words, $\frac{A_{X,\Delta}(v_\xi)}{S(v_\xi)}=1$ for all $\xi \in N_\bR$. 
\end{proof}

\begin{conj}\label{c-divi}
Let $(X,\Delta)$ be a K-semistable log Fano pair and $T$ a torus acting on $(X,\Delta)$. If $\tdelta(X,\Delta)=1$, then there exists a divisorial valuation $v$ not of the form $\wt_\xi$ such that
$$\frac{A_{X,\Delta}(v)}{S_{X,\Delta}(v)}=\delta_{X,\Delta,T}(v)=\delta_T(X,\Delta)=1.$$
\end{conj}

\begin{rem}
Theorem \ref{t-redqm} answers the expectation in \cite{Li19}*{Remark 3.25}. By \cite{BX19}*{Theorem 4.1}, Conjecture \ref{c-divi} implies Conjecture \ref{c-kpoly}.
\end{rem}

\end{document}